\documentclass{amsart}

\usepackage{amssymb}
\usepackage{hyperref}
\usepackage{amsfonts}
\usepackage{amsmath}
\usepackage{tikz}
\usepackage{tikz-cd}
\usepackage{mathtools}
\usepackage{amsthm}
\usepackage{dsfont}
\usepackage{verbatim}

\newtheorem{thm}{Theorem}[section]
\newtheorem{prop}[thm]{Proposition}
\newtheorem{lem}[thm]{Lemma}
\newtheorem{cor}[thm]{Corollary}
\newtheorem{fact}[thm]{Fact}

\theoremstyle{definition}
\newtheorem{definition}[thm]{Definition}
\newtheorem{example}[thm]{Example}
\newtheorem{warning}[thm]{Warning}
\newtheorem{question}[thm]{Question}

\theoremstyle{remark}
\newtheorem{remark}[thm]{Remark}

\numberwithin{equation}{section}

\DeclareMathOperator{\liet}{\mathfrak{t}}

\DeclareMathOperator{\liei}{\mathfrak{i}}
\DeclareMathOperator{\liep}{\mathfrak{p}}

\DeclareMathOperator{\cO}{\mathcal{O}}
\DeclareMathOperator{\cK}{\mathcal{K}}
\DeclareMathOperator{\cI}{\mathcal{I}}
\DeclareMathOperator{\cG}{\mathcal{G}}
\DeclareMathOperator{\cU}{\mathcal{U}}
\DeclareMathOperator{\cT}{\mathcal{T}}
\DeclareMathOperator{\cH}{\mathcal{H}}
\DeclareMathOperator{\cE}{\mathcal{E}}

\DeclareMathOperator{\cR}{\mathcal{R}}

\DeclareMathOperator{\cF}{\mathcal{F}}

\DeclareMathOperator{\cD}{\mathcal{D}}
\DeclareMathOperator{\cC}{\mathcal{C}}
\DeclareMathOperator{\cHom}{\mathcal{H}om}
\DeclareMathOperator{\cJ}{\mathcal{J}}
\DeclareMathOperator{\cP}{\mathcal{P}}
\DeclareMathOperator{\cL}{\mathcal{L}}

\DeclareMathOperator{\cA}{\mathcal{A}}

\DeclareMathOperator{\lie}{Lie}
\DeclareMathOperator{\spec}{Spec}
\DeclareMathOperator{\Map}{Map}
\DeclareMathOperator{\Sym}{Sym}
\DeclareMathOperator{\im}{im}
\DeclareMathOperator{\mmod}{-mod}
\DeclareMathOperator{\comod}{-comod}
\DeclareMathOperator{\mset}{-set}
\DeclareMathOperator{\Ind}{Ind}
\DeclareMathOperator{\Hom}{Hom}
\DeclareMathOperator{\Res}{Res}

\DeclareMathOperator{\cone}{cone}
\DeclareMathOperator{\SVect}{SVect}
\DeclareMathOperator{\Vect}{Vect}
\DeclareMathOperator{\rank}{rank}
\DeclareMathOperator{\Coh}{Coh}
\DeclareMathOperator{\Perf}{Perf}
\DeclareMathOperator{\QCoh}{QCoh}
\DeclareMathOperator{\Sh}{Sh}

\DeclareMathOperator{\bX}{\mathbb{X}}
\DeclareMathOperator{\bC}{\mathbb{C}}
\DeclareMathOperator{\bG}{\mathbb{G}}

\DeclareMathOperator{\bF}{\mathbb{F}}
\DeclareMathOperator{\bZ}{\mathbb{Z}}
\DeclareMathOperator{\b1}{\mathds{1}}
\DeclareMathOperator{\bD}{\mathbb{D}}
\DeclareMathOperator{\bL}{\mathbb{L}}

\DeclareMathOperator{\bbN}{\mathbf{N}}

\DeclareMathOperator{\mup}{\mu_p}

\DeclareMathOperator{\bmp}{B\mu_p}
\DeclareMathOperator{\bfp}{\mathbb{F}_p}

\makeatletter
\newcommand{\colim@}[2]{%
  \vtop{\m@th\ialign{##\cr
    \hfil$#1\operator@font colim$\hfil\cr
    \noalign{\nointerlineskip\kern1.5\ex@}#2\cr
    \noalign{\nointerlineskip\kern-\ex@}\cr}}%
}
\newcommand{\colim}{%
  \mathop{\mathpalette\colim@{\rightarrowfill@\textstyle}}\nmlimits@
}
\makeatother

\makeatletter
\newcommand{\nlim@}[2]{%
  \vtop{\m@th\ialign{##\cr
    \hfil$#1\operator@font lim$\hfil\cr
    \noalign{\nointerlineskip\kern1.5\ex@}#2\cr
    \noalign{\nointerlineskip\kern-\ex@}\cr}}%
}
\newcommand{\nlim}{%
  \mathop{\mathpalette\nlim@{\leftarrowfill@\textstyle}}\nmlimits@
}
\makeatother

\def\presuper#1#2%
  {\mathop{}%
   \mathopen{\vphantom{#2}}^{#1}%
   \kern-\scriptspace%
   #2}

\usepackage{mathabx,epsfig}
\def\acts{\mathrel{\reflectbox{$\righttoleftarrow$}}}

\newcommand{\vsim}{\mathbin{\rotatebox[origin=c]{-90}{$\sim$}}}
\newcommand{\vcong}{\mathbin{\rotatebox[origin=c]{-90}{$\cong$}}}
\newcommand{\veq}{\mathbin{\rotatebox[origin=c]{-90}{$=$}}}

\begin{document}

\title{Steenrod operators, the Coulomb branch and the Frobenius twist, I}

\author{Gus Lonergan}
\address{Department of Mathematics, Massachusetts Institute of Technology, 
Cambridge, MA 02139}
\email{gusl@mit.edu}

\begin{abstract}
In Part I, we use Steenrod's construction to prove that the quantum Coulomb branch is a Frobenius-constant quantization. We will also demonstrate the corresponding result for the $K$-theoretic version of the quantum Coulomb branch. In Part II, we use the same method to construct a functor of categorical $p$-center between the derived Satake categories with and without loop-rotation, which extends the Frobenius twist functor for representations of the dual group. 
\end{abstract}

\maketitle
\tableofcontents

\addtocontents{toc}{\protect\setcounter{tocdepth}{1}}

\section{Introduction}

\subsection{}This paper is about power operations.

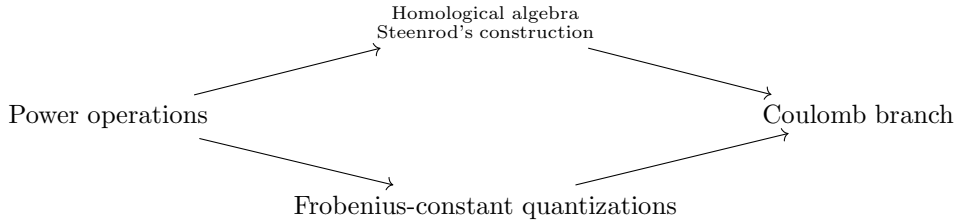
\begin{equation*}
\begin{tikzcd}
& \substack{\text{Homological algebra}\\ \text{Steenrod's construction}} \arrow{rd} &\\
\text{Power operations} \arrow{ru}\arrow{rd}&& \text{Coulomb branch}\\
& \text{Frobenius-constant quantizations}\arrow{ru}
\end{tikzcd}
\end{equation*}

A power operation is an enhanced version of a $p^{th}$-power map. One of the most famous examples is Steenrod's operations \cite{Steenrod}, a cornerstone of algebraic topology. In Section \ref{sec2}, we will give an account of Steenrod's construction in the language of derived categories. In these terms, the construction itself is very simple, and it yields not only Steenrod's cohomology operations but also operations in Borel-Moore homology, which are presumably related to the Kudo-Araki-Dyer-Lashof operations \cite{KA},\cite{DL}. A reader who knows about equivariant constructible derived categories on complex algebraic varieties will be able to understand these constructions even if they do not know any homotopy theory. Perhaps this is an advantage.\

\subsection{}In Section \ref{sec3}, we will introduce a different type of power operation, due to Bezrukavnikov-Kaledin \cite{BK}, which is an important tool in non-commutative algebraic geometry. Such a power operation is known as a Frobenius-constant quantization. Essentially, a Frobenius-constant quantization of a commutative algebra $A$ over $\bfp$ is a $1$-parameter flat deformation $A_\hbar$ of $A$ in associative algebras which has a large center; see Subsection \ref{prel} for a precise definition which also justifies regarding such a thing as a power operation. The main example is the Weyl algebra
$$
\bfp[\hbar]\langle x,\partial\rangle/([\partial,x]=\hbar)
$$
which contains $x^p$, $\partial^p$ in its center.\

\subsection{}We will then illustrate a general method to apply Steenrod's construction to produce Frobenius constant quantizations. It is not completely clear just how general this method may be, but heuristically it ought to work whenever the multiplication in $A_\hbar$ is somehow related to, if not directly inherited from, the homotopy-commutative multiplication of a based loop group. The example which we use to illustrate the method is the \emph{quantum Coulomb branch} of Braverman-Finkelberg-Nakajima \cite{BFN} - or rather, its natural characteristic $p$ version. That is, we prove:
\begin{thm}\label{asdf} For any complex reductive algebraic group $G$, and finite-dimensional representation $\bbN$ of $G$, and any odd prime $p$, the quantum Coulomb branch is a Frobenius-constant quantization.\end{thm}
The Coulomb branch is the $G$-equivariant Borel-Moore homology of a certain algebraic space $\cR$; the quantum Coulomb branch is obtained by switching on loop-rotation equivariance. The key geometric insight behind the Theorem is that, following ideas of Beilinson-Drinfeld \cite{BD}, one may deform the space $\cR$ with its $\mup$-action by loop rotation, to $\cR^p$ with its $\mup$-action by permuting the factors cyclically.\

\subsection{}This is already quite a broad class of examples. For instance, it includes partially spherical rational Cherednik algebras, see \cite{BEF}, \cite{Web}. It is expected that the same underlying geometry will lead to the discovery of large centers of related algebras. In fact, in Section \label{sec4} we indicate how the same underlying geometry shows that the $K$-theoretic version of \emph{integral} quantum Coulomb branch, which is itself a $q$-deformation of the $K$-theoretic version of the Coulomb branch, admits a large center when $q$ is evaluated at any complex root of unity (not necessarily of prime order). Essentially the only difference with the homological case is to replace Steenrod's construction with a so-called `Adams construction' which is to Adams operations as Steenrod's construction is to Steenrod's operations.

\begin{question}Are there such things as equivariant elliptic Borel-Moore homology and equivariant elliptic $K$-homology? Can we similarly deduce the analogous `large center' statements for these in the loop group situation?\end{question}

\subsection{}In part II of this work, we apply the same method in the case $\bbN=0$ but to the entire Satake category. We thus obtain a triangulated monoidal central functor
$$
F_\hbar: D^b_{G(\cO)}(Gr_G,\bfp)\to D^b_{G(\cO)\rtimes\bC^*}(Gr_G,\bfp)[\hbar^{-1}]
$$
where $\hbar$ is the first Chern class of $B\bC^*$. The approach is entirely analogous to Gaitsgory's `central sheaves' \cite{Gaits}. There is a `perverse Tate cohomology' functor $\presuper{p}H^0_{\text{Tate}}:D^b_{G(\cO)\rtimes\bC^*}(Gr_G,\bfp)[\hbar^{-1}]\to \text{Perv}_{sph}(Gr,\bfp)$, and we show:

\begin{thm}\label{hoho} Let $G^\vee$ denote the Langlands dual group to $G$ over $\bfp$. Let $S:\text{Rep}(G^\vee)\xrightarrow{\sim} \text{Perv}_{sph}(Gr,\bfp)$ be the geometric Satake equivalence \cite{MV}. Then the functor
$$
S^{-1}\circ \presuper{p}H^0_{\text{Tate}}\circ F_\hbar\circ S:\text{Rep}(G^\vee)\to \text{Rep}(G^\vee)
$$
is isomorphic to the Frobenius twist functor.\end{thm}

The operation of perverse Tate cohomology is strictly necessary in general. However, if a representation of $G^\vee$ has a $\bZ_p$-free lift, then we can cook up its Frobenius twist `on the nose' using the same geometry. We will formulate this precisely, and propose some links to derived geometric Satake and the Finkelberg-Mirkovic conjecture, in the next installment.\

\begin{warning}The proof of Theorem \ref{asdf} relies quite heavily on the theory of placid ind-schemes, dimension theories etc., see \cite{Rask}. The first half of Section \ref{sec3} simultaneously reviews this theory and introduces the examples which are relevant for us. As such it is written to be reasonably convincing, with the key facts explained in full detail, but with some details missing. All of the details are available in \cite{Rask}, which the reader is strongly recommended to read.\end{warning}

\subsection*{Acknowledgements}I wish to thank many people for their interest, help and encouragement, even if they are not aware of it, including D. Ben-Zvi, A. Braverman, P. Etingof, V. Ginzburg, I. Loseu, D. Vogan. Special thanks go to D. Nadler and V. Drinfeld for inviting me to speak about this project at their seminars, causing a dramatic improvement in my productivity levels; to G. Williamson, for raising the question\footnote{At the AIM Workshop: Sheaves and modular representations of reductive groups, March 28 to April 1, 2016.}: `What is the representation-theoretic meaning of the Steenrod algebra action on mod-p cohomology?', which this project is an attempt to answer; to my adviser R. Bezrukavnikov, for eternal patience and support; and to J. Louveau, whom this is for.

\addtocontents{toc}{\protect\setcounter{tocdepth}{2}}

\section{Steenrod's construction}\label{sec2}

\subsection{Overview}Let $p$ be an odd prime number, and let $\mup$ be the group of complex $p^{th}$ roots of unity. Let $R$ be a commutative ring. Let $k$ be a field of characteristic $p$, and let $F:k\to k$ be the Frobenius map. Let $X$ be a topological space and let $D^b(X,R)$ denote the bounded derived category of sheaves of $k$-modules on $X$. We write $X^{\mup}$ for $\Map(\mup,X)$. Following Steenrod \cite{Steenrod}, we construct a functor
$$
St:D^b(X,R)\to D^b_{\mup}(X^{\mup},R)
$$
where $D^b_{\mup}(X^{\mup},R)$ denotes the bounded $\mup$-equivariant derived category of sheaves of $R$-modules on $X^{\mup}$. This functor is not linear or triangulated, but nonetheless if we take $R=k$, compose with restriction to the diagonal and apply to morphisms between shifted constant sheaves we obtain linear maps 
\begin{align}\label{st1}
F^*H^n(X,k)\to H^{pn}_{\mup}(X,k)\cong \bigoplus_{i+j=pn}H^i(X,k)\otimes H^j(\bmp,k)
\end{align}
for each $n\geq0$. Recall that $H^*(\bmp,k)=k[a,\hbar]$ is the super-polynomial algebra in one variable $a$ of degree $1$ and one variable $\hbar$ of degree $2$. Here $\hbar$ is the first Chern class of the tautological complex line bundle on $\bmp$ arising from the embedding $\mup\subset \bC^*$.\

The direct sum of the maps of equation \ref{st1} is not in the most naive sense an algebra homomorphism. This fact led Steenrod to introduce certain correction factors which make it so; his famous cohomology operations are then defined to be the coefficients of the resulting algebra homomorphism in the monomial basis of $k[a,\hbar]$. However, the sum of maps of equation \ref{st1} \emph{does} give a homomorphism of super-graded algebras
\begin{align}\label{st2}
H^*(X,k)^{(1)}\to H^*_{\mup}(X,k)\cong H^*(X,k)[a,\hbar]
\end{align}
where $H^*(X,k)^{(1)}$ denotes the \emph{Frobenius twist} of $H^*(X,k)$. Naively one might think that this is just the $p$-dilation of $F^*H^*(X,k)$. This is wrong: rather, the natural and correct definition of the Frobenius twist of an algebra $A$ in any symmetric monoidal category over $k$ is as the Tate cohomology:
$$
A^{(1)}:= \hat{H}^0_{\mup}(A^{\otimes\mup})
$$
where the symmetric monoidal structure endows $A^{\otimes\mup}$ with the structure of $\mup$-equivariant algebra. In the case of the super-graded $k$-algebra $H^*(X,k)$, the underlying super-graded $k$-module of this construction is the same as the $p$-dilation of $F^*H^*(X,k)$, but the multiplication differs by a sign\footnote{When $p\equiv 3\mod4$.}, removal of which is part of the purpose of Steenrod's correction factors.\

We prefer therefore to use Steenrod's operations in their raw form - that is, without the correction factors and packaged as in equation \ref{st2}. This has the advantage of revealing the fundamental connected between Steenrod's operations and the Artin-Schreier map, which is obscured by the correction factors:

\begin{fact}\label{torusfact} Let $X=BT$ for some complex torus $T$. Then the Picard group of $X$ is canonically isomorphic to the character lattice $\bX^\bullet(T)$ of $T$, and the cohomology ring is the polynomial algebra
$$
H^*(X,\bZ) = \Sym_{\bZ}\bX^\bullet(T)
$$
with $\bX^\bullet(T)$ in degree $2$. This is equal to the ring $\cO(\liet_{\bZ})$ of polynomial functions on the canonical $\bZ$-form of the scheme $\liet=\lie(T)$. Likewise we have
$$
H^*(X,k) = \cO(\liet_{k})
$$
where $\liet_{k}$ denotes the canonical $k$-form of $\liet$. Under this identification, the map of equation \ref{st2} factors as
$$
\cO(\liet_{k})^{(1)}\xrightarrow{AS_\hbar} \cO(\liet_{k})[\hbar] \subset \cO(\liet_{k})[a,\hbar]
$$
where $AS_\hbar$ corresponds, on the level of $\bar{k}$-points, to the $\bar{k}^\times$-equivariant family, parameterized by $\hbar\in \bar{k}$, of additive maps of free $\bar{k}$-modules
$$\begin{matrix}
\bar{k}\otimes\liet_{k} & \to & \bar{k}^{(1)}\otimes\liet_{k}\\
\sum_ix_i\otimes v_i & \mapsto & \sum_i(x_i^p-\hbar^{p-1}x_i)\otimes v_i\end{matrix}
$$
for a basis $\{v_i\}$ of $\liet_{\bfp}$. This family interpolates between the usual Artin-Schreier map for $\hbar=1$ and the Frobenius map for $\hbar=0$.\end{fact}
\begin{remark}The appearance of $AS_\hbar$ in the topological setting was the first indication that Steenrod's construction might be related to the theory of Frobenius-constant quantizations, where $AS_\hbar$ plays a central role, see Fact \ref{frobconstfact}.\end{remark}

\subsection{Steenrod's construction}\label{stconst}

Recall that $p$ is an odd prime, $\mup$ is the group of complex $p^{th}$ roots of unity, $R$ is a commutative ring, $k$ is a field of characteristic $p$ and $X$ is a topological space. We denote by $C^b(X,R)$, $D^b(X,R)$ the (bounded) cochain, derived categories of sheaves of $R$-modules on $X$. If $Y$ is a topological space with an action of $\mup$, we denote by $C^b_{\mup}(Y,R)$, $D^b_{\mup}(Y,R)$ the corresponding $\mup$-equivariant categories. Since $\mup$ is a finite group, these are the same as the (bounded) cochain, derived categories of $\mup$-equivariant sheaves of $R$-modules on $Y$.\

Consider the functor of $p^{th}$ external tensor power
$$
C^b(X,R)\xrightarrow{\boxtimes p} C^b(X^{\mup},R).
$$
It sends quasi-isomorphisms to quasi-isomorphisms, and so descends to a (non-triangulated) functor 
$$
D^b(X,R)\xrightarrow{\boxtimes p} D^b(X^{\mup},R)
$$
by the universal property of derived categories. Notice that the cochain-level functor factors as
$$
\boxtimes p:C^b(X,R)\xrightarrow{St_C}C^b_{\mup}(X^{\mup},R)  \xrightarrow{} C^b(X^{\mup},R).
$$
To make this explicit, we first choose an isomorphism 
$$
\mup\cong \bZ/p\cong\{1,\ldots,p\}=:[p];
$$
the result will be independent of this choice. Write $\sigma$ for the generator of $\mu_p$ corresponding to $1$ under the isomorphism. Then, for a complex $A^\bullet$, we give the complex
$$
(A^\bullet)^{\boxtimes p}=\left(\bigoplus_{i_1+\ldots+ i_p=\bullet}A^{i_1}\boxtimes\ldots\boxtimes A^{i_p}\right)^\bullet
$$
the $\mu_p$-equivariant structure by letting the generator $\sigma$ act by the direct sum of the canonical isomorphisms of sheaves
$$
A^{i_1}\boxtimes\ldots\boxtimes A^{i_p}\cong \sigma^*(A^{i_2}\boxtimes\ldots\boxtimes A^{i_p}\boxtimes A^{i_1})
$$
each twisted by the sign $(-1)^{i_1(n-i_1)}$. The sign twist is the natural (Koszul) choice which makes the action of $\mu_p$ commute with the differential. Moreover, given a chain map $f:A^\bullet\to B^\bullet$, $f^{\boxtimes p}$ is automatically a $\mu_p$-equivariant chain map. Since the functor $C^b_{\mup}(X^{\mup},R)\to C^b(X^{\mup},R)$ reflects quasi-isomorphisms, it follows immediately that $St_C$ descends to a functor $St_D$ as below:
$$
\boxtimes p:D^b(X,R)\xrightarrow{St_D}D^b_{\mup}(X^{\mup},R)  \xrightarrow{} D^b(X^{\mup},R).
$$
Writing $\Sigma$ for the suspension functor, we have $St_D\Sigma\cong\Sigma^pSt_D$. Also, $St_D$ is not triangulated, nor additive or even linear. The following two propositions control the failure of linearity.

\begin{prop}\label{induced}Suppose given two parallel morphisms $f,g:A^\bullet\to B^\bullet$ in $D^b(X,R)$. Then the morphism
$$
St_D(f+g)-St_D(f)-St_D(g):St_D(A^\bullet)\to St_D(B^\bullet)
$$
is an induced map. That is, there exists some \textbf{non-equivariant} map
$$
h:(A^\bullet)^{\boxtimes\mup}\to (B^\bullet)^{\boxtimes\mup}
$$
such that the \textbf{equivariant} map
$$
Av(h)=\sum_{x\in\mup}xhx^{-1}: St_D(A^\bullet)\to St_D(B^\bullet)
$$
is equal to $St_D(f+g)-St_D(f)-St_D(g)$.\end{prop}

\begin{proof}Let us right away replace $A^\bullet,B^\bullet$ by isomorphic objects so that $f,g$ become genuine maps of complexes. Let $\underline{f},\underline{g}$ denote the constant functions $\mup\to\{f,g\}$ with respective values $f,g$. Then $\mup$ acts freely on $\{f,g\}^{\mup}-\{\underline{f},\underline{g}\}$; choose a set $\{h_1,\ldots,h_n\}$ of orbit representatives ($n=(2^p-2)/p$). Then each $h_i$ determines a non-equivariant map $(A^\bullet)^{\boxtimes\mup}\to (B^\bullet)^{\boxtimes\mup}$, hence so does their sum $h$. Then, we have
$$
(f+g)^{\boxtimes\mup}-f^{\boxtimes\mup}-g^{\boxtimes\mup}=\sum_{x\in\mup} xhx^{-1}
$$
where, by definition $xhx^{-1}$ is the composition:
$$
xhx^{-1} : (A^\bullet)^{\boxtimes\mup}\cong x^*(A^\bullet)^{\boxtimes\mup}\xrightarrow{x^*(h)} x^*(B^\bullet)^{\boxtimes\mup}\cong (B^\bullet)^{\boxtimes\mup}
$$
where the two isomorphisms are given by the equivariant structures.\end{proof}

\begin{prop}\label{Frobmult}$St_D$ is Frobenius-multiplicative with respect to the action of the multiplicative monoid $R$ on hom-sets. That is, $St_D$ determines a functor
$$
St_D:\Ind_R^RD^b(X,R)\to D^b_{\mup}(X^{\mup},R)
$$
which respects multiplication by $R$. Here the category on the left is obtained from $D^b(X,R)$ by regarding each hom-set as a set with an action of the multiplicative monoid $R$ and inducing along the $p^{th}$-power map of monoids $R\to R$.\end{prop}

Note that $\Ind_R^RD^b(X,R)$ is not an additive category in general. However, suppose that $R=k$ and $k$ is perfect. In that case, the Frobenius map of monoids is actually a map of rings $F$, and is moreover bijective. Write $M:k\mmod\to k\mset$ for the forgetful functor, where $k\mset$ denotes the category\footnote{The reader may prefer to replace this by its full subcategory of all $k^\times$-sets with a unique stable point.} of sets with action of the multiplicative monoid $k$. We have the following:
\begin{lem}\label{perf}Suppose that $k$ is a perfect field of characteristic $p$. Then we have 
$$
M\circ F^*\cong \Ind_k^k\circ M.
$$
\end{lem}
\begin{proof}Indeed, in that case both $F^*$ and $\Ind_k^k$ are equivalent to functors which do not change the underlying abelian group/set, and only change the way that $k$ acts.\end{proof}

It follows that if $k$ is a perfect field of characteristic $p$, we have produced a $k$-multiplicative functor
$$
St_D:F^*D^b(X,k)\to D^b_{\mup}(X^{\mup},k).
$$
Since $F^*D^b(X,k)$ is triangulated, we find this statement somewhat nicer that the version for general $R$.

\subsection{Localization}The category $D^b_{\mup}(X^{\mup},R)$ is enriched, in a triangulated sense, over $H^*_{\mup}(X^{\mup},R)$. That is, the monoidal structure of $D^b_{\mup}(X^{\mup},R)$ gives maps of $R$-modules
$$
H^n_{\mup}(X^{\mup},R)\cong \Hom_{D^b_{\mup}(X^{\mup},R)}(R,\Sigma^nR)\to \Hom_{D^b_{\mup}(X^{\mup},R)}(Id,\Sigma^n)
$$
for each $n\geq0$, whose sum is a map of algebras. In particular, $D^b_{\mup}(X^{\mup},R)$ is enriched in the same sense over 
$$
H^*_{\mup}(*,R)\cong H^*(\bmp,R)\cong R\otimes_{\bZ}^L (\bZ[\hbar]/p\hbar).
$$
Here $\hbar$ is the first Chern class of the tautological line bundle on $\bmp$ corresponding to the embedding $\mup\to \bC^*$. In particular this super-commutative ring receives a map from $R\otimes_{\bZ}^L\bZ[\hbar] = R[\hbar]$ so that $D^b_{\mup}(X^{\mup},R)$ is enriched over $R[\hbar]$. Thus we may consider the 2-periodic $R$-linear triangulated category $D^b_{\mup}(X^{\mup},R)[\hbar^{-1}]$, which is enriched over
$$
R\otimes_{\bZ}^L (\bZ[\hbar]/p\hbar)[\hbar^{-1}]\cong R\otimes_{\bZ}^L \bfp[\hbar^{\pm1}].
$$
The degree $0$ component of this ring is $R/p$, and the natural map from $R$ to here is the modular reduction map. In particular, there is a Frobenius map of rings
$$
F: R\to R\otimes_{\bZ}^L \bfp[\hbar^{\pm1}]
$$
In this way, it makes sense to ask whether a functor from a triangulated category enriched over $R$ to one enriched over $R\otimes_{\bZ}^L \bfp[\hbar^{\pm1}]$ is Frobenius-linear.

\begin{prop}\label{Tate}The composition
$$
St_D':D^b(X,R)\xrightarrow{St_D}D^b_{\mup}(X^{\mup},R)\to D^b_{\mup}(X^{\mup},R)[\hbar^{-1}]
$$
is exact, Frobenius-linear and preserves direct sums.\end{prop}

\begin{proof}First we prove that $St_D'$ preserves direct sums. Let $A^\bullet$, $B^\bullet$ be complexes in $D^b(X,R)$. We argue as in the proof of Proposition \ref{induced} that we have
$$
St_D(A^\bullet\oplus B^\bullet)\cong St_D(A^\bullet)\oplus St_D(B^\bullet)\oplus \Ind_1^{\mup}C^\bullet
$$
for some complex $C^\bullet$ in $D^b(X^{\mup},R)$. Here $\Ind_1^{\mup}$ is the averaging functor
$$
\Ind_1^{\mup}: D^b(X^{\mup},R)\to D^b_{\mup}(X^{\mup},R)
$$
bi-adjoint to the restriction functor $\Res_1^{\mup}$. Therefore, it suffices to prove that the composition
$$
D^b(X^{\mup},R)\xrightarrow{\Ind_1^{\mup}} D^b_{\mup}(X^{\mup},R)\to D^b_{\mup}(X^{\mup},R)[\hbar^{-1}]
$$
is isomorphic to $0$. This follows by adjunction from the fact that $Res_1^{\mup}(\hbar)=0$.\

Next we prove Frobenius-linearity. By Proposition \ref{Frobmult}, it suffices to prove that $St_D'$ respects addition of parallel morphisms. By Proposition \ref{induced}, it is enough to see that the image of an induced morphism in $D^b_{\mup}(X^{\mup},R)$ in the localized category $D^b_{\mup}(X^{\mup},R)[\hbar^{-1}]$ is $0$. This we have shown in the previous paragraph.\

Finally, we prove exactness. First we must specify an exact structure, i.e. an isomorphism $e:\Sigma St_D'\cong St_D'\Sigma$, which makes the image under $St_D'$ of any triangle a triangle. Note that since $p$ is odd, there is a morphism in $D^b_{\mup}(X^{\mup},R)$ of functors $\hbar^{(p-1)/2}:\Sigma\xrightarrow{}\Sigma^p$ which becomes an isomorphism when $\hbar$ is inverted. Already on the level of complexes we have a canonical isomorphism $\Sigma^pSt_C\cong St_C\Sigma$. The exact structure is taken to be the composition
$$
e:\Sigma St_D'\xrightarrow{((p-1)/2)!\hbar^{(p-1)/2}}\Sigma ^pSt_D'\cong St_D'\Sigma.
$$
This is indeed an isomorphism since the localized category is enriched over $\bfp$. The reason for the factor $((p-1)/2)!$ will be explained shortly. Thus, given a triangle\footnote{We have immediately rotated the arbitrary exact triangle $A^\bullet\xrightarrow{f}B^\bullet\xrightarrow{g}C^\bullet\xrightarrow{h}A^\bullet$ for convenience later.}
$$
B^\bullet\xrightarrow{g}C^\bullet\xrightarrow{h}\Sigma A^\bullet\xrightarrow{-\Sigma f}\Sigma B^\bullet
$$
in $D^b(X,R)$, we have a triangle
$$
\begin{matrix} St_D'(B^\bullet) & \xrightarrow{St_D'(g)} & St_D'(C^\bullet) & \xrightarrow{St_D'(h)} & St_D'(\Sigma A^\bullet) & \xrightarrow{St_D'(-\Sigma f)=-St_D'(\Sigma f)} & St_D'(\Sigma B^\bullet)\\
&&&&&& \Big\downarrow{\vcong~e^{-1}}\\
&&&&&& \Sigma St_D'(B^\bullet)\\
\end{matrix}
$$
in $D^b_{\mup}(X^{\mup},R)[\hbar^{-1}]$. We must show that this is exact whenever the original triangle is. Since any exact triangle in a derived category is isomorphic to a semi-split one, we may assume that $f$ is a chain map and $C^\bullet$ is the usual mapping cone satisfying
$$
C^n = A^{n+1}\oplus B^n
$$
with differential
$$
\begin{pmatrix}
-d & 0\\ 
f & d
\end{pmatrix}.
$$
So we have
$$
St_C(C^\bullet) = \left(\bigoplus_{i_1+\ldots+ i_p=\bullet}(A^{i_1+1}\oplus B^{i_1})\boxtimes\ldots\boxtimes(A^{i_p+1}\oplus B^{i_p}) \right)^\bullet
$$
with some differential\footnote{Which the reader may write down if desired. Its exact form is not important.}. This complex has a $(p+1)$-step equivariant increasing filtration
$$
St_C(B^\bullet) = F_0St_C(C^\bullet)\subset\ldots\subset F_pSt_C(C^\bullet) = St_C(C^\bullet)
$$
where $F_iSt_C(C^\bullet)$ consists of the subcomplex of $St_C(C^\bullet)$ in which at most $i$ summands $A^?$ are taken in the expansion of the external tensor product. The inclusion of the zeroth piece of this filtration is equal to $St_C(g)$, while the quotient map
$$
St_C(C^\bullet)\twoheadrightarrow St_C(C^\bullet)/F_{p-1}St_C(C^\bullet) \cong St_C(\Sigma A^\bullet)
$$
is equal to $St_C(h)$. Furthermore, arguing as in Proposition \ref{induced}, we see that $F_iSt_C(C^\bullet)/F_{i-1}St_C(C^\bullet)$ is an induced complex for each $1\leq i\leq p-1$. Therefore, the map
$$
St_C(C^\bullet)/St_C(B^\bullet)\twoheadrightarrow St_C(\Sigma A^\bullet)
$$
becomes an isomorphism in $D^b_{\mup}(X^{\mup},R)[\hbar^{-1}]$. Consider now the commutative diagram:
$$
\begin{matrix} St_C(B^\bullet) & \xrightarrow{St_C(g)} & St_C(C^\bullet) & \xrightarrow{~~~\alpha~~~} & St_C(C^\bullet)/St_C(B^\bullet) & \xrightarrow{~~~~~~~~~~~\beta~~~~~~~~~~~} & \Sigma St_C(B^\bullet)\\
 \Big\downarrow{\veq} &  & \Big\downarrow{\veq} &  & \Big\downarrow{\gamma} &  & \\
St_C(B^\bullet) & \xrightarrow{St_C(g)} & St_C(C^\bullet) & \xrightarrow{St_C(h)} & \Sigma^pSt_C( A^\bullet) & \xrightarrow{~~~~~~-\Sigma^pSt_C( f)~~~~~~} & \Sigma^pSt_C(B^\bullet)\\
\end{matrix}
$$
of equivariant chain complexes. Here $\alpha,\beta$ are the usual chain maps which make the image of the top row in $D^b_{\mup}(X,R)$ an exact triangle, $\gamma$ is the quotient map, whose image in $D^b_{\mup}(X,R)[\hbar^{-1}]$ is an isomorphism, and $St_C\Sigma$ has been identified with $\Sigma^pSt_C$. Comparing with the definition of the triangle obtained by applying $St_D'$ to $B^\bullet\xrightarrow{g}C^\bullet\xrightarrow{h}\Sigma A^\bullet\xrightarrow{-\Sigma f}\Sigma B^\bullet$, we see that it is enough to prove that the diagram
$$
\begin{matrix} St_D(C^\bullet)/St_D(B^\bullet) & \xrightarrow{~~~~~~~~~~~\beta~~~~~~~~~~~} & \Sigma St_D(B^\bullet)\hfill\\
\Big\downarrow{\gamma} &  & \Big\downarrow{{((p-1)/2)!\hbar^{(p-1)/2}}} \\
\Sigma^pSt_D(A^\bullet) & \xrightarrow{~~~~~~-\Sigma^pSt_D(f)~~~~~~} & \Sigma^pSt_D(B^\bullet)\hfill\\
\end{matrix}
$$
commutes in $D^b_{\mup}(X,R)[\hbar^{-1}]$. Here we have written $\beta,\gamma$ for their own images in $D^b_{\mup}(X,R)$. Let $D^\bullet$ be the cocone of $A^\bullet\xrightarrow{id}A^\bullet$, so that $f$ induces a map $D^\bullet\to C^\bullet$. We have a commutative diagram
$$
\begin{matrix}
St_C(D^\bullet)/St_C(A^\bullet) & \xrightarrow{\epsilon} & \Sigma St_C(A^\bullet)\\
\Big\downarrow{\delta} &  & \Big\downarrow{\Sigma St_C(f)} \\
St_C(C^\bullet)/St_C(B^\bullet) & \xrightarrow{\beta} & \Sigma St_C(B^\bullet)\\
\end{matrix}
$$
where $\epsilon$ is the standard boundary map, and the vertical arrows are induced\footnote{Without additional signs - the sign changes in the horizontal rows are complementary.} by $f$. Now $\gamma\delta$ becomes an isomorphism in $D^b_{\mup}(X^{\mup},R)$ for the same reason that $\gamma$ does; therefore so does $\delta$. So it is enough to show that the composition of these two diagrams is commutative. By functoriality of $\hbar$, the resulting composition equals
$$
\begin{matrix}
St_D(D^\bullet)/St_D(A^\bullet) & \xrightarrow{\epsilon} & \Sigma St_D(A^\bullet)\hfill\\
\Big\downarrow{\gamma\delta} &  & \Big\downarrow{\Sigma^p St_D(f)\circ{((p-1)/2)!\hbar^{(p-1)/2}}} \\
\Sigma^pSt_D(A^\bullet) & \xrightarrow{~~~~~~-\Sigma^pSt_D(f)~~~~~~} & \Sigma^pSt_D(B^\bullet)\hfill\\
\end{matrix}
$$
Let
$$
\liep = (R\to R[\mup]\xrightarrow{\sigma-1}R[\mup]\xrightarrow{N}\ldots\xrightarrow{\sigma-1}R[\mup])
$$
be the equivariant resolution of the trivial $R[\mup]$-module supported in degrees $(1-p),\ldots,0$. We have the standard chain maps $\liep\to R$, which is an isomorphism in the equivariant derived category, and $\liep\to \Sigma^{p-1}R$, which equals $\hbar^{(p-1)/2}$ by definition. Now $\epsilon$ is a chain map, which is an isomorphism in the equivariant derived category but not in the equivariant complex category. However, there \emph{is} a chain map
$$
\liep\otimes\Sigma St_C(A^\bullet)\xrightarrow{\zeta} St_C(D^\bullet)/St_C(A^\bullet)
$$
such that $\epsilon\zeta$ is induced by the standard chain map $\liep\to R$ (i.e. counit in degree $0$). To see this, it is enough to do the case $A^\bullet=R$ and then tensor on the right with $St_C(A^\bullet)$. In that case, we are looking for an equivariant chain map from the complex
$$
R\to R[\mup]\xrightarrow{\sigma-1}R[\mup]\xrightarrow{N}\ldots\xrightarrow{\sigma-1}R[\mup]
$$
supported in degrees $-p,\ldots,-1$ to the complex $E^\bullet$ satisfying
$$
E^i=\bigoplus_{\substack{S\subset [p]\\ |S|=-i}}R_S
$$
where $R_S$ is a copy of $R$, and with differential sending $R_S$ to $\bigoplus_{s\in S}R_{S-\{s\}}$ by $(1,-1,1,\ldots)$. Let us write $1_S$ for the canonical generator of $R_S$. One example of such a map is the map which sends the element $1$ in the degree $-(2i+1)$ copy of $R[\mup]$ to the term:
$$
-i!\sum_{\substack{T\subset \{2,\ldots,p\}\\ |T|=2i\\ \text{even block lengths}}}1_{\{1\}\cup T}
$$
and sends the element $1$ in the degree $-(2i+2)$ copy of $R[\mup]$ to the term:
$$
-i!\sum_{\substack{T\subset \{3,\ldots,p\}\\ |T|=2i\\ \text{even block lengths}}}1_{\{1\}\cup\{2\}\cup T}.
$$
The sign is chosen so that $\epsilon\zeta$ is induced by the standard chain map $\liep\to R$. We compute:
$$
\gamma\delta\zeta = -((p-1)/2)!\hbar^{(p-1)/2}.
$$
Therefore, the two paths $St_D(D^\bullet)/St_D(A^\bullet)\to \Sigma^pSt_D(B^\bullet)$ in this composed diagram are equalized by $\zeta$. Since $\zeta$ is an isomorphism in $D^b_{\mup}(X^{\mup},R)[\hbar^{-1}]$, they coincide in the localized category as required.
\end{proof}

\begin{cor}Suppose $R=k$ is a field of characteristic $p$. Then we have a triangulated $k$-linear functor
$$
St_D':F^*D^b(X,k)\to D^b_{\mup}(X^{\mup},R)[\hbar^{-1}].
$$
\end{cor}

\subsection{Six functors}\label{six}We are mainly concerned with the case where $X$ is the (Borel) quotient $EG\frac{\times}{G}Y$ of a complex algebraic variety $Y$ by the action of some affine algebraic group $G$, and $R$ is a Noetherian ring of finite homological dimension. In this case we will replace the category $D^b$ with its constructible analogue $D^b_c$. That is, any $G$-equivariant constructible sheaf on $Y$ descends to a sheaf on $X$, and $D^b_c(X,R)$ is the thick subcategory of $D^b(X,R)$ generated by all such sheaves. We will usually write $D^b_{c,G}(Y,R)$ instead of $D^b_c(X,R)$. 

The constructions of the previous section preserve constructibility, so we have a Steenrod construction
$$
St_D:D^b_{c,G}(Y,R)\to D^b_{c,G^{\mup}\rtimes\mup}(Y^{\mup},R).
$$
Recall that we have the six functor formalism for constructible derived categories. We assume that the reader is familiar with this material, but remind him/her of the standard notation: for a $G$-equivariant algebraic map $f:Y\to Y'$, we have the adjoint pairs of exact functors
$$
{f^*\negmedspace: D^b_{c,G}(Y',R)\rightleftarrows D^b_{c,G}(Y,R):\negmedspace f_*}
$$
$$
{f_!\negmedspace: D^b_{c,G}(Y,R)\rightleftarrows D^b_{c,G}(Y',R):\negmedspace f^!}
$$
and also a pair of bi-exact bifunctors
$$
(-)\otimes(-):D^b_{c,G}(Y,R)\times D^b_{c,G}(Y,R)\to D^b_{c,G}(Y,R)
$$
$$
\cHom(-,-):D^b_{c,G}(Y,R)^{op}\times D^b_{c,G}(Y,R)\to D^b_{c,G}(Y,R)
$$
related by a tensor-hom adjunction. There is also a Verdier duality functor $\bD$, and an exceptional tensor product $\otimes^!$, which can be written in terms of the other functors, as can the external tensor product $\boxtimes$. We call the collection of all of these functors the \emph{six plus functors}. Notice that $G^{\mup}\rtimes\mup$ is also an affine algebraic group, so the six functor formalism exists for the target category of $St_D$. Also if $f:Y\to Y'$ is $G$-equivariant then $f^{\mup}:Y^{\mup}\to (Y')^{\mup}$ is $G^{\mup}\rtimes\mup$-equivariant. The following fact is essentially a consequence of the same fact for $^{\boxtimes p}$:
\begin{prop}Steenrod's construction is compatible with the six functor formalism. That is, we have canonical isomorphisms
$$
\begin{matrix} \hfill (f^{\mup})^*St_D & \cong & St_Df^* \hfill \\
 \hfill (f^{\mup})_*St_D & \cong & St_Df_* \hfill \\
 \hfill (f^{\mup})_!St_D & \cong & St_Df_! \hfill \\
 \hfill (f^{\mup})^!St_D & \cong & St_Df^! \hfill \\
 \hfill St_D(-)\otimes St_D(-) & \cong & St_D(-\otimes -) \hfill \\
 \hfill \cHom(St_D(-),St_D(-)) & \cong & St_D\cHom(-,-) \hfill \end{matrix}
$$
commuting with any and all adjunction morphisms of the six functor formalism.
\end{prop}

We have the same compatibilities with functors $St_D'$.

\begin{remark}Many of the six plus functors are defined in much more general contexts than the constructible derived category. For instance $f^*,f_*$ are defined in complete generality, and their compatibilities with Steenrod's construction hold in that generality. One expects that, in some sense, any time any of these functors is defined it is compatible with Steenrod's construction. But we wish to avoid making any precise statement along these lines.\end{remark}

\subsection{Tate's construction}Note that the object $St_D(\Sigma^nR)$ is canonically isomorphic to $\Sigma^{pn}R$ with its trivial $\mu_p$-equivariant structure. Here $R$ is the constant sheaf. Since a degree $n$ cohomology class is just a morphism $R\to \Sigma^nR$ in $D^b(X,R)$, we thus obtain a Frobenius-multiplicative map of multiplicative $R$-sets:
$$
St^H:H^n(X,R)\to H^{pn}_{\mup}(X^{\mup},R).
$$
This map is not Frobenius-linear, but as in Proposition \ref{induced}, its deviation from additivity is by a class induced from $H^{pn}(X^{\mup},R)$. To see how these maps interact with multiplication, we make the following definition.

\begin{definition}Let $(\cC,*,\b1,e,a,s)$ be a symmetric monoidal abelian category enriched over some commutative ring $R$, i.e. $\cC$ is an abelian category, $*$ is a bi-exact $R$-linear functor $\cC\times\cC\to\cC$, $\b1$ is an object of $\cC$, $e$ is a pair of equivalences $\b1*Id\cong Id\cong Id*\b1$, $a$ is an associativity constraint for $*$ and $s$ is a commutativity constraint for $*$, satisfying natural compatibilities. Let $A$ be an object of $\cC$. Then $s$ determines an action of $\mup$ on $A^{*\mup}$, and we define
$$
A^{(1)} := \hat{H}_{\mup}^0(A^{\mup}) := \ker_{A^{*\mup}}(1-\sigma)/\im_{A^{*\mup}}(N).
$$
This is the so-called \emph{Tate construction}. For a morphism $f:A\to B$ the morphism $f^{*\mup}:A^{*\mup}\to B^{*\mup}$ is $\mup$-equivariant and so induces a morphism $A^{(1)}\to B^{(1)}$, so that $(-)^{(1)}$ becomes a functor.\end{definition}

\begin{lem}$(-)^{(1)}$ is additive over $\bZ$. \end{lem}
\begin{proof}We essentially rehash the proof of Propositions \ref{induced}, \ref{Tate}. First we show that $(-)^{(1)}$ is linear over $\bZ$. Suppose $f,g:A\to B$ are two parallel morphisms in $\cC$. Let $\underline{f},\underline{g}$ denote the constant functions $\mup\to\{f,g\}$ with respective values $f,g$. Then $\mup$ acts freely on $\{f,g\}^{\mup}-\{\underline{f},\underline{g}\}$; choose a set $\{h_1,\ldots,h_n\}$ of orbit representatives ($n=(2^p-2)/p$). Then each $h_i$ determines a non-equivariant map $A^{*\mup}\to B^{*\mup}$, and we have
$$
(f+g)^{*\mup}-f^{*\mup}-g^{*\mup}=\sum_{x\in\mup}\sum_{i=1}^n xh_ix^{-1}.
$$
Restricting to $\ker_{A^{*\mup}}(1-\sigma)$, this becomes
$$
((f+g)^{*\mup}-f^{*\mup}-g^{*\mup})_{\ker_{A^{*\mup}}(1-\sigma)}=\sum_{x\in\mup}\sum_{i=1}^n xh_i=N\sum_{i=1}^n h_i
$$
which factors through $\im_{B^{*\mup}}(N)$ as required.

Next we show that $(-)^{(1)}$ preserves direct sums. Let $\underline{A},\underline{B}$ denote the constant functions $\mup\to\{A,B\}$ with respective values $A,B$. Then $\mup$ acts freely on $\{A,B\}^{\mup}-\{\underline{A},\underline{B}\}$; choose a set $\{C_1,\ldots,C_n\}$ of orbit representatives. Then each $C_i$ determines an object of $\cC$, and as a $\mup$-module in $\cC$ we have
$$
(A\oplus B)^{*\mup}\cong A^{*\mup}\oplus B^{*\mup}\oplus \bigoplus_{i=1}^nC_i[\mup].
$$
The result then follows from the fact that $\hat{H}^0_{\mup}(k[\mup])=0$ in $R\mmod$.\end{proof}

Let $\SVect_k$ denote the symmetric monoidal category of $\bZ/2$-super graded $k$-vector spaces.


\begin{lem}Suppose that $R=k$ is a field of characteristic $p$ and that $\cC$ admits a super-fiber functor $\cC\to \SVect_k$. Then $(-)^{(1)}$ is exact, monoidal and Frobenius-linear over $k$. In the case $\cC=\SVect_k$, $(-)^{(1)}$ is equivalent to the functor $k\otimes_F$(-) which tensors the $k$-linear structure along the Frobenius map $F:k\to k$.\end{lem}

\begin{proof}Since Tate's construction commutes with the fiber functor, it is enough to take $\cC= \SVect_k$, where it is a simple calculation using bases.\end{proof}

Now suppose that $A,B$ are objects of $\cC$. We have a $\mup$-equivariant isomorphism $(A*B)^{*\mup}\cong (A^{*\mup}*B^{*\mup})$. We have also the natural inclusions
$$
\ker_{A^{*\mup}}(1-\sigma)*\ker_{B^{*\mup}}(1-\sigma)\to \ker_{A^{*\mup}*B^{*\mup}}(1-\sigma)
$$
$$
\im_{A^{*\mup}}(N)*\ker_{B^{*\mup}}(1-\sigma)\to \im_{A^{*\mup}*B^{*\mup}}(N)
$$
$$
\ker_{A^{*\mup}}(1-\sigma)*\im_{A^{*\mup}}(N)\to \im_{A^{*\mup}*B^{*\mup}}(N)
$$
which induce a map $A^{(1)}*B^{(1)}\to (A*B)^{(1)}$. Suppose that $(A,\b1_A,m_A)$ is a unital ring in $\cC$. Then $A^{(1)}$ still has a multiplication
$$
m_{A^{(1)}}:A^{(1)}*A^{(1)}\to (A*A)^{(1)}\xrightarrow{m_A^{(1)}}A^{(1)}.
$$
Also, there is a canonical isomorphism $\ker_{\b1^{*\mup}}(1-\sigma)\cong\b1$, hence a canonical surjection $\b1\to \b1^{(1)}$ which determines a map
$$
\b1_{A^{(1)}}:\b1\to \b1^{(1)}\xrightarrow{\b1_A^{(1)}}A^{(1)}.
$$
One may check that this makes $A^{(1)}$ into a ring, and moreover that $A^{(1)}$ is associative or commutative if $A$ is. The following lemma explains how this looks in the main example.

\begin{lem}\label{signchange}Suppose $A=\bigoplus_{i\in \bZ/2}A_i$ is a unital ring in $\cC=\SVect_k$. Then the ring structure on $A^{(1)}$ corresponds under the identification $A^{(1)}=k\otimes_FA$ to the ring structure with unit $1\otimes_F\b1_A$ and multiplication
$$
m_{A^{(1)}}(r\otimes a,r'\otimes a')=(-1)^{ij{p\choose 2}}rr'\otimes m_A(a,a')
$$
for $a\in A_i,a'\in A_j$ and $r,r'\in k$.\end{lem}

\begin{proof}The isomorphism of $k\otimes_FA_{i}$ with $(A^{(1)})_{i}$ sends the element $r\otimes a$ to the class of $r.\underbrace{a\otimes\ldots\otimes a}_{\text{$p$ times}}$. The natural map 
$$
(A^{(1)})_{i}\otimes (A^{(1)})_{j}\cong \hat{H}^0_{\mup}((A_i)^{*\mup})\otimes \hat{H}^0_{\mup}((A_j)^{*\mup})\to \hat{H}^0_{\mup}((A_i\otimes A_j)^{*\mup})
$$
sends the class of $\underbrace{a\otimes\ldots\otimes a}_{\text{$p$ times}}\otimes \underbrace{a'\otimes\ldots\otimes a'}_{\text{$p$ times}}$ to the class of $(-1)^{ij{p\choose 2}}\underbrace{(a\otimes a')\otimes\ldots\otimes (a\otimes a')}_{\text{$p$ times}}$, since it entails permuting the $x^{th}$ copy of $A_i$ with the $y^{th}$ copy of $A_j$ for every $p\geq x>y\geq 1$. \end{proof}

Arguing the same way, we have the following:

\begin{prop}Let $A$ be a Hopf algebra in $\SVect_k$. Then $A^{(1)}$ is naturally a Hopf algebra in $\SVect_k$. It has multiplication and unit given as in \ref{signchange}, comultiplication given by
$$
\Delta_{A^{(1)}}(r\otimes a)=r\otimes (-1)^{{p\choose 2}\deg\otimes\deg}\Delta_{A}(a),
$$
counit given by $\epsilon_{A^{(1)}}(r\otimes a) = r(\epsilon_A(a))^p$ and antipode given by $S_{A^{(1)}}(r\otimes a) = r\otimes S_A(a)$. Moreover the functor $(-)^{(1)}$ on $\SVect_k$ upgrades to  functor
$$
(-)^{(1)}:A\comod\to A^{(1)}\comod.
$$
For an $A$-comodule $M$, the $A^{(1)}$-comodule structure on $M^{(1)}$ is given by
$$
\Delta_{M^{(1)}}(r\otimes m) = r\otimes(-1)^{{p\choose2}\deg\otimes\deg}\Delta_M(m).
$$
\end{prop}


\begin{example}Suppose $\cO$ is a commutative Hopf algebra in $\SVect_k$. Then the monoidal category $\cC$ of $\cO$-comodules is symmetric. Taking $p^{th}$ powers gives a (Frobenius) map of Hopf algebras $F_{\cO}:\cO^{(1)}\to\cO$. Then, Tate's construction on $\cC$ factors as
$$
\cO\comod\xrightarrow{(-)^{(1)}}\cO^{(1)}\comod\xrightarrow{F_{\cO}^*}\cO\comod.
$$
For instance, we could take $\cO$ to be the ring of functions $\cO(\bG_m)$ on the multiplicative group $\bG_m$ over $k$, concentrated in degree $0\in\bZ/2$. Then $\cO(\bG_m)\comod$ contains as a full subcategory over $\SVect_k$ the category of $\bZ$-super graded vector spaces, and Tate's construction there is isomorphic to the functor which applies $k\otimes_F(-)$ and multiplies degrees by $p$.
\end{example}

Recall we have the Frobenius-multiplicative maps of multiplicative $k$-sets
$$
St^H:H^n(X,k)\to H^{pn}_{\mup}(X^{\mup},k).
$$
We view cohomology rings as commutative ring objects of $\bZ$-super graded vector spaces; in particular we can apply functor $(-)^{(1)}$ to them. By Lemma \ref{perf}, if $k$ is perfect then it gives a map of $\bZ$-super graded $k$-sets
$$
St_{ex}:H^*(X,k)^{(1)}\to H^{*}_{\mup}(X^{\mup},k).
$$
The following fact is immediate from the constructions.
\begin{prop}$St_{ex}$ respects the multiplicative $k$-monoidal structures.\end{prop}

\begin{remark}If $k$ is not perfect, then the map $H^*(X,k)\to H^{*}_{\mup}(X^{\mup},k)$ of $\bZ/2$-super graded sets respects the multiplicative monoidal structures up to the sign change of Lemma \ref{signchange}. There is presumably an appropriate non-linear version of Tate's construction which would allow us to say that we really have a certain $\bZ$-super graded $k$-monoid $H^*(X,k)^{(1)}_{nl}$ and a map of monoids $H^*(X,k)^{(1)}_{nl}\to H^{*}_{\mup}(X^{\mup},k)$, but we prefer for simplicity not to do it.\end{remark}

\subsection{Borel-Moore homology} 

We return to the setting of Subsection \ref{six}. Let $\omega$ denote the $G$-equivariant dualizing complex on $Y$ with coefficients in $R$. We have a canonical isomorphism\footnote{here of course the second $\omega$ denotes the $G^{\mup}\rtimes\mup$-equivariant dualizing complex on $Y^{\mup}$ with coefficients in $R$} $St_D(\omega)\cong\omega$. By definition, the $G$-equivariant Borel-Moore homology of $Y$ is
$$
H_n^{BM,G}(Y,R):=\Hom_{D^b_G(Y,R)}(R,\Sigma^n\omega).
$$
See Subsection \ref{BM} for more about this. Altogether $H_n^{BM,G}(Y,R)$ form a $\bZ$-super graded $H^*_G(Y,R)$-module; in particular it is a module for $H^*_G(*,R)$. By functoriality we have the non-linear maps
$$
St^{BM}:H_n^{BM,G}(Y,R)\to H_{pn}^{BM,G^{\mup}\rtimes\mup}(Y^{\mup},R).
$$
This is a map of $St^H$-monoids. Its discrepancy from additivity it averaged from $H_{pn}^{BM,G^{\mup}}(Y^{\mup},R)$. If $R$ is a perfect field $k$ of characteristic $p$, we can say that we have a non-linear graded map of $St_{ex}$-monoids:
$$
St^{BM}_{ex}:H_*^{BM,G}(Y,k)^{(1)}\to H_{*}^{BM,G^{\mup}\rtimes\mup}(Y^{\mup},k).
$$

\subsection{Steenrod operations}For simplicity let's assume $k$ to be perfect from now on. Let us compose $St_{ex}$ with the restriction map to the diagonal:
$$
St_{in}:H^*(X,k)^{(1)}\xrightarrow{St_{ex}} H^{*}_{\mup}(X^{\mup},k)\xrightarrow{\Delta^*} H^{*}_{\mup}(X,k)\cong H^*(X,k)[a,\hbar].
$$
This is again a map of multiplicative $k$-monoids. Tautologically we have $St_{in}(x)=x^p\mod(a,\hbar)$. Also $St_{in}$ is compatible with pull-back maps in cohomology in the natural way. Since induction commutes with restriction, the difference between $St_{in}(x+y)$ and $St_{in}(x)+St_{in}(y)$ is induced from a cohomology class $z\in H^\bullet(X;k)$. Since $\mup$ acts trivially on $X$, that means that it is equal to $pz=0$, so $St_{in}$ is linear. That is, we have a map of super-commutative $k$-algebras
$$
St_{in}:H^*(X,k)^{(1)}\to H^*(X,k)[a,\hbar].
$$

\begin{remark}The coefficients of $\hbar^m,a\hbar^m$ in $St_{in}$ are \emph{not} the Steenrod operations. More precisely, they are the Steenrod operations only up to some non-zero scalars. Even more precisely, let $x\in H^n(X,k)$ and let $p=2q+1$. Consider 
$$
(-1)^{qn(n-1)/2}(q!)^{-n}St_{in}(x).
$$
where $x$ is viewed as a degree $pn$ element of $H^*(X,k)^{(1)}$. The coefficient of $\hbar^m$ in this expression vanishes unless $m=\frac{1}{2}(p-1)(n-2s)$ for some $s$ such that $2s\leq n$, in which case that coefficient is equal to $(-1)^sP^s(x)$ where $P^s$ is the $s^{th}$ Steenrod operation. Similarly, the coefficient of $a\hbar^m$ in that expression vanishes unless $m=\frac{1}{2}(p-1)(n-2s)-1$ for some $s$ such that $2s\leq n$, in which case that coefficient is equal to $(-1)^{s+1}\beta P^s(x)$ where $\beta$ is the Bockstein operation.
\end{remark}

\subsection{Artin-Schreier}

We indicate how the Artin-Schreier map comes naturally out of the above considerations. First note that if $n$ is even then the number 
$$
(-1)^{qn(n-1)/2}(q!)^{-n}
$$
boils down to $(-1)^{n/2}$. It is a standard fact that on a degree $2$ class $x$ we have $P^0(x)=x,P^1(x)=x^p$, and higher powers vanish. Therefore
$$
St_{in}(x) = x^p - \hbar^{p-1}x + \hbar^{p-2}\beta(x).
$$
Let $X=BT$ for some compact torus $T$. Since its cohomology is supported in even degrees, the Bockstein operator acts as zero and $St_{in}$, on the level of $k$-cohomology, is exactly the $\hbar$-Artin-Schreier map
$$
\cO(\liet_{k})^{(1)}\xrightarrow{AS_\hbar} \cO(\liet_{k})[\hbar] \subset \cO(\liet_{k})[a,\hbar]
$$
as defined in Fact \ref{torusfact}.\

Recall that if $G$ is a compact Lie group with maximal torus $T$, and $p$ is large enough with respect to the Weyl group of $G$, then the projection $BT\to BG$ induces an inclusion
$$
H^*(BG,k)\to H^*(BT,k)
$$ 
which is identified with
$$
\cO(\liet_{k}//W)\to \cO(\liet_{k}).
$$
The $\hbar$-Artin-Schreier map induces a map on subspaces
$$
\cO(\liet_{k}//W)^{(1)}\xrightarrow{AS_\hbar} \cO(\liet_{k}//W)[\hbar]
$$
which is also important in the theory of Frobenius-constant quantization. The point is that this map is also induced by $St_{in}$, since it is compatible with pullbacks.\

It is entertaining to show more directly how $AS_\hbar$ arises, without relying on any outside facts about Steenrod operations. We can reduce to the rank one case $T=S^1$. Let $b$ denote the degree $2$ generator (first Chern class of tautological line bundle) of $BS^1$; we need to show that $St_{in}(b)=b^p-\hbar^{p-1}b$. Let $C_p\subset S^1$ denote the cyclic group of order $p$, considered as distinct from $\mup$. Consider the projection
$$
BC_p\to BS^1.
$$
It induces an injective map
$$
k[b]\to k[s,b]
$$
in cohomology, where $s$ is a degree $1$ generator. By functoriality it is enough to prove the equality when $b$ is regarded as a cohomology class of $BC_p$. Note that amongst degree $2p$ elements of $k[b,\hbar]$, the desired element $b^p-\hbar^{p-1}b$ is the unique one which gives $0$ when we set $b$ to any multiple in $\bfp$ of $\hbar$, and gives $b^p$ when we set $\hbar=0$. The latter statement is automatic, so we have to check the former. So fix some $t\in\bfp$. Having chosen an isomorphism $\mup\cong C_p$, $t$ determines a group homomorphism $\mup\to C_p$.\

The constant sheaf $\Sigma^nk$ of $BC_p$ is contained in the full subcategory 
$$
D^b(k[C_p]\mmod)=D_{C_p}(*)\subset D(BC_p).
$$
Our coefficients are $k$, which we drop from the notation. It is easier for our purpose to work in $D^b(k[C_p]\mmod)$. Compatible with the functor $St_D$ out of $D(BC_p)$ we have the functor
$$
St_D:D^b(k[C_p]\mmod)\to D^b(k[\mup\ltimes(C_p)^{\times p}]\mmod)
$$
This is then composed with the diagonal restriction
$$
D^b(k[\mup\ltimes(C_p)^{\times p}]\mmod)\xrightarrow{\Delta^*} D^b(k[\mu_p\times C^p]\mmod).
$$
By definition $St_{in}(b)$ is given by applying that composition to the morphism $k\xrightarrow{b}k[2]$, where $k$ is the trivial $C_p$-module. We want further to set $b=t\hbar$; this corresponds to restricting along the map
$$
\mup\xrightarrow{id\times t}\mup\times C_p.
$$
Write 
$$
(id\times t)^*:D^b(k[\mup\times C_p]\mmod)\to D^b(k[ \mup]\mmod)
$$
for the corresponding restriction map. We need to show that $(id\times t)^*\circ \Delta^*\circ St(b)=0$. But actually there is an isomorphism of functors
$$
(id\times t)^*\circ\Delta^*\circ St_D\cong St_D\circ i^*
$$
where $i^*$ is the forgetful functor $D^b(k[C_p]\mmod)=D_{C_p}(*)\to D(*)$. Indeed for an object $A^\bullet$ of $D^b(k[C_p]\mmod)$, the underlying complex of both functors is $(A^\bullet)^{\otimes \mup}$, and the automorphism which sends each summand 
$$
A^{i_1}\otimes\ldots\otimes A^{i_p}
$$ 
to itself by $1\otimes \sigma^t\otimes \sigma^{2t}\otimes\ldots\otimes \sigma^{(p-1)t}$ intertwines the two actions of $\mup$. Here $\sigma$ is some generator of $\mup$. But the functor $St\circ i^*$ kills $b$, since $i^*$ does.

\section{Coulomb branch}\label{sec3}



\subsection{Prelude: Frobenius-constant quantizations}\label{prel}

Let $k$ be a field of characteristic $p$ and let $\cC$ be a symmetric monoidal category over $k$. The reader may assume that $\cC$ is the category of comodules of some commutative Hopf algebra in $\SVect_k$. Let $A$ be a commutative (and associative) algebra in $\cC$. Let
$$
F:A^{(1)}\to A
$$
be the Frobenius map. Let $Q$ be an augmented commutative algebra in $\cC$ with augmentation $\epsilon:Q\to k$. Following \cite{BK} we make the following definition:

\begin{definition}\begin{enumerate}\item A $Q$\emph{-quantization} of $A$ is a flat associative $Q$-algebra $A_Q$ in $\cC$ such that $A_Q\otimes_Qk=A$.\
\item A \emph{Frobenius-constant }$Q$\emph{-quantization} of $A$ is a $Q$-quantization $A_Q$ of $A$ together with a map
$$
F_Q:A^{(1)}\to Z(A_Q)
$$
of algebras which lifts the Frobenius map, i.e. such that $\epsilon\circ F_Q\equiv F$. Here $Z(A_Q)$ denotes the center of $A_Q$.\end{enumerate}\end{definition}

The main example for us is the following. We take $K$ to be some $\bG_m$-equivariant algebraic group in $\Vect_k$, and view $\cO:=\cO(K\rtimes\bG_m)$ as a Hopf algebra in $\SVect_k$ concentrated in degree $0$. We take $\cC=\cO\comod$. Let $\hbar$ be a basis vector of the $1$-dimensional representation of $K\rtimes\bG_m$ in which $K$ acts trivially and $\bG_m$ acts with weight $2$. Let $Q=k[\hbar]$. In this case, we will call a $Q$-quantization simply an $\hbar$\emph{-quantization}, or just a \emph{quantization} if the meaning is clear.

\begin{fact}\label{frobconstfact}\begin{enumerate}\item Let $X$ be a smooth affine algebraic variety over $k$. Then the ring of asymptotic crystalline differential operators, $\cD_\hbar(X)$, is a canonical $\hbar$-quantization of $\cO(T^*X)$. Here $\bG_m$ acts trivially on $\cO(X)$ and on vector field with weight $2$. Let $\partial$ be a vector field on $X$. Then $\partial^p$ acts as a derivation on $\cO(X)$, so that $\partial^p-\partial^{[p]}$ annihilates $\cO(X)$ for a unique vector field $\partial^{[p]}$. Then $\cD_\hbar(X)$ has a canonical Frobenius-constant structure determined by
$$
\begin{matrix}F_\hbar:&x&\mapsto&x^p\hfill&x\in\cO(X)\hfill\\
&\partial&\mapsto&\partial^p-\hbar^{p-1}\partial^{[p]}&\partial\in Vect(X).\end{matrix}
$$\
\item Let $J$ be a smooth algebraic group over $k$. Then $F_\hbar$ as above is $K=J\times J$-equivariant (induced by left and right regular actions). In particular if we take invariants for the left factor, we obtain a Frobenius-constant structure for the quantization $\cU_\hbar(J)$ of $\cO(\lie(J)^*)$. \
\item Let $T$ be a complex torus and let $T^\vee$ be the Langlands dual split torus over $k$, that is:
$$
T^\vee=\spec(k[\bX_\bullet(T)])
$$
where $\bX_\bullet(T)$ is the cocharacter lattice of $T$ and $k[\bX_\bullet(T)]$ is its group algebra. We have canonical identifications
$$
\cO((\liet^\vee)^*)=\cO(\liet_k)
$$
$$
\cU_\hbar(\liet^\vee)=\cO(\liet_k\times\bG_a).
$$
If we take $\spec$ of the Frobenius-constant structure we recover the $\hbar$-Artin-Schreier map
$$
F_\hbar=AS_\hbar:\liet_k\times\bG_a\to \liet_k^{(1)}
$$
of Fact \ref{torusfact}.
\end{enumerate}\end{fact}

\begin{remark}If a commutative algebra and its quantization contain in a natural way $H^*_G(*,k)$ for some complex reductive group $G$ with maximal torus $T$, then when searching for a Frobenius-constant structure it is natural to look for one which is compatible with the $\hbar$-Artin-Schreier map.\end{remark}



\subsection{Formal neighborhoods}Let $X$ be a smooth complex curve and let $S$ be a finite set. Given a commutative ring $R$ and an $R$-point $x$ of $X^S$, we denote the coordinates of $x$ by $x_s$ ($s\in S$), write $\Gamma(x_s)$ for the graph of $x_s$ in $X_R$ and write $I(x_s)$ for its ideal. We write $\Delta_{S}(x)$ for the formal neighborhood of the union of the graphs of $x_s$ ($s\in S$). That is, $\Delta_S(x)$ is the direct limit in affine schemes over $X$:
$$
\Delta_S(x)=\colim_i \Delta_{S,i}(x)
$$
where
$$
\Delta_{S,i}(x) = \spec\left(\cO_{X_R}\left/\prod_{s\in S}I(x_s)^i\right.\right).
$$
Given a subset $S'\subset S$ and an $R$-point $x$ of $X^S$ we will write
$$
\Delta_S^{S'}(x)
$$
for the $S'$-punctured formal neighborhood, i.e. the complement of the union of the graphs of $x_s$ ($s\in S'$) in $\Delta_S(x)$. As a sheaf of algebras on $\Delta_S(x)$, $\cO(\Delta_S^{S'}(x))$ has an exhaustive increasing filtration:
$$
F^j\cO(\Delta_S^{S'}(x)) = \cO(\Delta_S(x)).\prod_{s\in S'}I(x_s)^{-j}.
$$
Suppose we have $S''\subset S'\subset S$ and $x\in X^S(R)$. The inclusion $S'\subset S$ defines a projection $f:X^S\to X^{S'}$, and we will occasionally write
$$
\Delta_{S'}^{S''}(x)
$$
for $\Delta_{S'}^{S''}(f(x))$. We have a closed embedding $\Delta_{S'}^{S''}(x)\to \Delta_S^{S''}(x)$. Note however that this is in a sense non-uniform in $x$: for instance if for every $s\in S$ there exists an $s'\in S'$ such that $x_s=x_{s'}$, then the embedding is an isomorphism; and conversely. This is essentially the fact underlying Beilinson-Drinfeld's `fusion' Grassmannian \cite{BD}. We will make more of this when we discuss co-placid morphisms, see Example \ref{RT}.\

For notational simplicity, we frequently remove commas and braces from $S$, $S'$, and also drop the part $(x)$, when it is clear which point we refer to. So for example the expression:
$$
\Delta_{\{1,2\}}^{\{1\}}(x)
$$
becomes:
$$
\Delta_{12}^1.
$$

\subsection{Global groups; pro-smoothness}\label{prosaic}Now fix an affine algebraic group $G$ over $\bC$. Consider the following functor from commutative rings to groups over $X^S$:
$$
G_{S}(R):=\{(x,f)|x\in X^S(R),f:\Delta_S(x)\to G\}.
$$
Then $G_{S}$ is represented by the limit of an inverse system of smooth affine group schemes over $X^S$:
$$
G_{S}=\nlim_i  G_{S,i}
$$
such that each transition morphism is a smooth homomorphism. Here $G_{S,i}$ may be taken to represent the functor
$$
G_{S,i}(R)=\{(x,f)|x\in X^S(R),f:\Delta_{S,i}(x)\to G\}.
$$
Later, the notation $G_{S,i}$ may represent a piece of some other cofiltered system presenting $G_S$; we will refer to the specific group above by $\Map(\Delta_{S,i},G)$.  The fact that each transition morphism is smooth is directly verified using the valuative criterion. Indeed let $\spec(\widetilde{R})$ be a square-zero thickening of $\spec(R)$. A commutative diagram
$$
\begin{matrix}\spec(R)&\to & G_{S,i+1}\\
\downarrow&&\downarrow\\
\spec(\widetilde{R})&\to& G_{S,i}\end{matrix}
$$
is the same thing as a point $\widetilde{x}\in X^S(\widetilde{R})$, with residue $x\in X^S(R)$, and a commutative diagram
$$
\begin{matrix}\Delta_{S,i}(x)&\to& \Delta_{S,i+1}(x)\\
\downarrow&&\downarrow\\
\Delta_{S,i}(\widetilde{x})&\to& G.\end{matrix}
$$
This determines a morphism $P\to G$ where $P$ is the appropriate pushout in affine schemes. Since $\Delta_{S,i}(x)$ is equal to the intersection of $\Delta_{S,i+1}(x)$ with $\Delta_{S,i}(\widetilde{x})$ inside $\Delta_{S,i+1}(\widetilde{x})$, and $\Delta_{S,i+1}(\widetilde{x})$ is a square-zero thickening of $\Delta_{S,i+1}({x})$, it follows that $\Delta_{S,i+1}(\widetilde{x})$ is a square-zero thickening of $P$. Therefore since $G$ is smooth we can extend $P\to G$ to $\Delta_{S,i+1}(\widetilde{x})\to G$, as required. Note that $G_{S,0}=X^S$ so in particular each $G_{S,i}$ is smooth over $X^S$.

Now fix $x\in X^S(\bC)$. It partitions $S$ into subsets $S_1,\ldots, S_n$ according to coincidence amongst the coordinates. Write $y_m$ for the coordinate $x_s$ for any $s\in S_m$, and $z_m$ for the $\bC$-point of $X^{S_m}$ with coordinates $y_m$. We have
$$
\begin{matrix}
\Delta_{S,i}(x)&=&\spec\left(\cO_X\left/\prod_{m=1}^nI(y_m)^{i|S_m|}\right.\right)\\
&=&\coprod_{m=1}^n\spec\left(\cO_X\left/I(y_m)^{i|S_m|}\right.\right).\end{matrix}
$$
Therefore we have
$$
G_{S,i}\times_{X^S}\{x\} = \prod_{m=1}^{n}G_{S_m,i}\times_{X^{S_m}}\{z_m\} = \prod_{m=1}^{n}G_{\{m\},i|S_m|}\times_{X^{\{m\}}}\{y_m\}.
$$
The smooth transition map $G_{\{m\},(i+1)|S_m|}\times_{X^{\{m\}}}\{y_m\}\to G_{\{m\},i|S_m|}\times_{X^{\{m\}}}\{y_m\}$ is surjective for all $i\geq0$ and has a unipotent kernel for all $i\geq1$. It follows that $G_{S,i+1}\to G_{S,i}$ has the same property. Thus $G_S$ is a \emph{prosaic} affine group scheme over $X^S$ in the following sense:
\begin{definition}\begin{enumerate}
\item A scheme $T$ over $B$ is said to be \emph{pro-smooth} over $B$ if it can be written as the limit of a inverse system of schemes $T_i$ smooth over $B$ and with smooth transition morphisms. If $T$ is pro-smooth then it is formally smooth (in particular flat) over $T$.\
\item An affine groupoid scheme $\cG$ over $B$ is \emph{pro-smooth} over $B$ if it can be written as the limit of an inverse system of affine groupoid schemes $\cG_i$ over $B$ whose structure maps to $B$ are both smooth, and which has smooth transition homomorphisms\footnote{Is this the same thing as an affine groupoid scheme $\cG$ over $B$ such that both structure morphisms $\cG\to B$ are pro-smooth in the sense of (2)?}.\
\item In (2) and (3) we can upgrade to the property of being \emph{a pro-smooth cover} by demanding that each transition map and structure map is a smooth cover.\
\item Let the affine groupoid scheme $\cG=\lim_{i\in\bZ_{\geq0}^{op}}\cG_i$ over $B$ be a pro-smooth cover. Then each $\cG_i$ is the fpqc quotient over $B$ of $\cG$ by some pro-smooth affine subgroup $K_i$. We say that $\cG$ is \emph{prosaic} if the $K_i$ can be chosen to be also pro-unipotent.\
\item Let $\cG$ be an affine group scheme over the same base $B$. Then $\cG$ is said to be \emph{pro-smooth}, \emph{a pro-smooth cover}, \emph{prosaic} over $B$ if it is so when regarded as a groupoid. \end{enumerate}\end{definition}

From now on, `groupoid' will mean `affine pro-smooth covering groupoid', unless it is clear from the context that this is not the case. All examples of groupoids will actually be prosaic.

\begin{remark}Recall the construction of $G_S$. If the affine algebraic group $G$ is replaced by an arbitrary smooth affine variety $T$ over $\bC$, we get a pro-smooth affine variety $T_S$ over $X^S$ in exactly the same way.\end{remark}

\subsection{Beilinson-Drinfeld Grassmannians; reasonableness}We also consider the functor
$$
G_{S}^{S'}(R):=\{(x,f)|x\in X^S(R), f:\Delta_S^{S'}\to G\}.
$$
Then $G_{S}^{S'}$ is represented by an ind-affine ind-scheme, formally smooth over $X^S$. It is a group in ind-schemes over $X^S$, but not an inductive limit of group schemes. It is a \emph{reasonable} ind-scheme in the following sense (taken from \cite{D}): 


\begin{definition}\begin{enumerate}\item An ind-scheme $T$ is \emph{reasonable} if it admits a \emph{reasonable presentation}, that is an expression
$$
T=\colim_{j\in\cJ}T^j
$$
where $\cJ$ is some (countable) filtered indexing category, and the transition morphisms in the filtered system of schemes $(T^j)_{j\in\cJ}$ are all finitely presented (f.p.) closed embeddings\footnote{That is, they have finitely generated ideal sheaves.}. Note that any two reasonable presentations admit a common refinement, so that the category of reasonable presentations of $T$ is filtered.\
\item A closed subscheme of a reasonable ind-scheme $T$ is \emph{reasonable} if it is a term in some reasonable presentation of $T$.\
\item A morphism $U\to T$ of reasonable ind-schemes is \emph{co-reasonable} if for some, equivalently any, reasonable presentation $T=\colim_{j\in\cJ}T^j$ of $T$, the presentation $U=\colim_{j\in\cJ}U\times_TT^j$ of $U$ as an ind-scheme is reasonable. Warning: this is not a relative version of reasonableness for ind-schemes.\end{enumerate}\end{definition}


\begin{example}\begin{enumerate}\item Let $T$ be a reasonable ind-scheme and let $U\to T$ be either ind-f.p. or an ind-flat cover. Then $U\to T$ is co-reasonable.\
\item In the case of $G_S^{S'}$, one reasonable presentation is given as follows. Fix a finite set $\{a_1,\ldots,a_n\}$ of generators of $\cO(G)$. Then set $\cJ=\bZ_{\geq0}$ and set $G_S^{S',j}$ to be the closed subscheme of $G_S^{S'}$ which on the level of $R$-points is given by
$$
G_S^{S',j}(R) = \{(x,f)|x\in X^S(R), f:\Delta_S^{S'}\to G,a_k\circ f\in H^0 F^j\cO(\Delta_S^{S'})\}.
$$
Here we have taken $G_S^{S',0}=G_S$. The left- and right-regular actions of the subgroup $G_S$ preserve the inductive structure, meaning that each $G_S^{S',j}$ has a free action on both sides by $G_S$ over $X^S$, even though it is not itself a group. Moreover the fpqc quotient $G_S^{S',j}/G_S$ is of finite-type over $X^S$, and flat, although generally quite singular. The result is that the fpqc quotient
$$
G_S^{S'}/G_S
$$
has the structure of ind-finite-type ind-flat ind-scheme over $X^S$. In particular, it is reasonable, and $G_S^{S'}\to G_S$ is an ind-flat cover and thus co-reasonable.\end{enumerate}\end{example}

On $R$-points, we may identify
$$
G_S^{S'}/G_S(R) = \left\{(x,\cE,f)  \left|~\begin{matrix}   x\in X^S(R)\hfill\\
\cE\text{ a principal }G\text{-bundle over }\Delta_S(x)\hfill\\
 f\text{ a trivialization of }\cE\text{ over }\Delta_S^{S'}(x)\hfill\end{matrix}\right.\right\}/\sim.
$$
Here the symbol `$/\sim$' means `taken up to isomorphism', i.e. we identify two $R$-points $(x,\cE,f),(x',\cE',f')$ if $x=x'$ and there exists an isomorphism of $\cE$ with $\cE'$ which intertwines $f,f'$. Such an isomorphism is unique if it exists. The following fact is due to \cite{BD}:

\begin{lem}\begin{enumerate}\item $G_S^{S'}/G_S$ is ind-projective over $X^S$ if and only if $G$ is reductive.\
\item $G_S^{S'}/G_S$ is ind-reduced if and only if $G$ has no non-trivial characters.\end{enumerate}\end{lem}

\begin{remark}Ultimately we are concerned only with the analytifications of these ind-schemes, so point (2) appears merely for interest's sake. But point (1) is crucial for the definition of convolution in Borel-Moore homology.\end{remark}

We may re-identify the $R$-points of $G_S^{S'}$ in a way more compatible with the above identification of $G_S^{S'}/G_S(R)$:
$$
G_S^{S'}(R) = \left\{(x,\cE,f,g)  \left|~\begin{matrix}   x\in X^S(R)\hfill\\
\cE\text{ a principal }G\text{-bundle over }\Delta_S(x)\hfill\\
 f\text{ a trivialization of }\cE\text{ over }\Delta_S^{S'}(x)\hfill\\
 g\text{ a trivialization of }\cE\text{ over }\Delta_S(x)\end{matrix}\right.\right\}/\sim.
$$

Notice that the inclusion $S'\subset S$ induces a closed embedding $\Delta_{S'}^{S''}\to \Delta_S^{S''}$ for any $S''\subset S'$. This in turn induces restriction homomorphisms
$$
G_S^{S''}\to G_{S'}^{S''}.
$$
These maps are co-reasonable. One readily checks by looking at points that the induced maps
$$
G_S^{S''}/G_S\to X^S\times_{S'}G_{S'}^{S''}/G_{S'}
$$
are isomorphisms. In particular we have $G_S^{S''}/G_S\xrightarrow{\sim} X^S\times_{S''}G_{S''}^{S''}/G_{S''}$.


\begin{remark}$G_{S}^{S}/G_S$ is known as the Beilinson-Drinfeld grassmannian $Gr_S$ (on $|S|$ points). In particular the fibers of $G_S^{S'}/G_S$ over $X^S$ are products of copies of the ordinary affine Grassmannian $Gr_G$ of $G$.\end{remark}

\subsection{Jet bundles; placidity} We will use the following notion, due to Raskin \cite{Rask}:

\begin{definition} \begin{enumerate}\item A scheme $T$ is called \emph{placid} if it admits a \emph{placid presentation}, that is, an expression
$$
T = \nlim_{i\in\cI^{op}}(T_i)
$$
for some filtered (countable) indexing category $\cI$, such that each $T_i$ is of finite type over $\bC$ and each transition morphism $T_i\to T_{i'}$ is a smooth affine covering. We will denote this placid presentation by $T_{\cI}$.\
\item An ind-scheme $T$ is called \emph{placid} if it admits a \emph{placid presentation}, that is, an expression
$$
T = \colim_{j\in\cJ}\nlim_{i\in (\cI^j)^{op}}(T^j_i).
$$
Here $\cJ,\cI^j$ are filtered (countable) indexing categories, $T^j_{\cI^j}$ is a placid presentation of its limit scheme $T^j:=\nlim_{i\in (\cI^j)^{op}}(T^j_i)$, and the transition morphisms $T^j\to T^{j'}$ are ind-f.p. closed embeddings.\
\end{enumerate}\end{definition}

\begin{remark}\label{flexy}\begin{enumerate} \item If the placid ind-scheme $T$ maps to some base $B$ of finite type over $\bC$, then the placid presentation may be taken over $B$.\
\item Let $T$ be a placid ind-scheme and let $T=\colim_{j\in\cJ}(T_j)$ be a reasonable presentation of $T$. Then each $T_j$ is a placid scheme, so that this reasonable presentation can be extended to a placid presentation.\
\item Any two placid presentations of the placid ind-scheme $T$ admit a common refinement (which is again a placid presentation). Thus the collection of placid presentations of $T$ forms a filtered category $\cP(T)$.

\item Suppose that $T$ is a placid ind-scheme and $U$ is an ind-scheme with an ind-f.p. map $f:U\to T$. Then $U$ is automatically placid. The short explanation is `by Noetherian approximation'. We spell it out: given any reasonable presentation
$$
T=\colim_{j\in\cJ}(T^j)
$$
of $T$, we set $U_j:=U\times_TT_j$ and obtain the reasonable presentation
$$
U=\colim_{j\in\cJ}(U^j)
$$
of $U$. Then, given any placid presentation
$$
T^j=\nlim_{i\in(\cI^j)^{op}}T^j_i
$$
of $T^j$, there exists some index $a$ of $\cI^j$ such that there is a $T^j_a$-scheme $U^j_a$ fitting into a Cartesian diagram
$$
\begin{matrix}U^j&\to&T^j\\
\downarrow&&\downarrow\\
U^j_a&\to&T^j_a.
\end{matrix}
$$
Moreover since $T^j\to T^j_a$ is a covering, the choice of $T^j_a$-scheme $U^j_a$ is unique. Thus if we replace $\cI^j$ by its final subcategory based at $a$, we can present $U^j\to T^j$ as the limit of a cofiltered system of f.p. morphisms
$$
(U^j_i\to T^j_i)_{i\in(\cI^j)^{op}}
$$
such that for each $i\to i'$ in $\cI$, the square
$$
\begin{matrix} U^j_{i'} & \to & T^j_{i'}\\
\downarrow && \downarrow\\
U^j_i & \to & T^j_i\end{matrix}
$$
is Cartesian. Placid presentations of this form will be called \emph{Cartesian}. \
\item The product (over $B$) of placid ind-schemes is placid.\

\item Consider placid presentations
$$
T=\colim_{j\in \bZ_{\geq0}}\nlim_{i\in ( \bZ_{\geq j})^{op}}(T^j_i)
$$
of $T$ with the property that $T^j_i$ is formed out of $T^{j'}_i$ as in point (4) of this Remark whenever $i\geq j'\geq j$. We call such placid presentations \emph{neat}. This is possibly a technically useless notion. But every placid presentation admits a refinement which is neat up to replacing the indexing categories $\cJ,\cI^j$ by final subcategories. Thus many constructions on placid ind-schemes can be phrased in terms of neat presentations. \end{enumerate}\end{remark}

Note that for any morphism $f:U\to T$ of placid schemes and any placid presentations $U=\nlim_{i\cI_U^{op}}U_i$, $T=\nlim_{i\cI_T^{op}}T_i$, then for any $i\in\cI_T$ there exists $i'\in\cI_U$ and a unique map $U_{i'}\to T_i$ making the square
$$
\begin{matrix}U&\to&T\\
\downarrow&&\downarrow\\
U_{i'}&\to&T_i\end{matrix}
$$
commutative. Thus, by changing the indexing sets appropriately we can choose placid presentations of $U$, $T$ with a common indexing set $\cI$ and write $f=\nlim_{i\in\cI}(U_i\to T_i)$. Such a presentation will be called \emph{compatible}. This notion extends immediately to morphisms of placid ind-schemes. A Cartesian presentation is a compatible presentation in which all appropriate squares are Cartesian. If $\cG$ is an affine groupoid scheme over some base $B$ of finite type over $\bC$ and $f$ is $\cG$-equivariant (over $B$), then we can find a $\cG$-equivariant compatible presentation. If in addition $f$ is f.p. so that it admits a Cartesian presentation, then this can also be chosen to be $\cG$-equivariant.\

The following definition is due to \cite{Rask}\footnote{In \emph{loc. cit.} a co-placid map of placid ind-schemes is called simply `placid'. As noted in \emph{loc. cit.}, it is \emph{not} a relative version of placidity for ind-schemes. This is the reason for the present renaming.}.

\begin{definition}\begin{enumerate}\item A morphism $f:U\to T$ between placid schemes is called \emph{co-placid} if for some, equivalently every, pair of placid presentations $U_{\cI_U}$, $T_{\cI_T}$ of $U$, $T$ and for every index $i\in \cI_T$, then for some, equivalently every, index $i'\in\cI_U$ such that we have a commutative square:
$$
\begin{matrix}U&\xrightarrow{f}&T\\
\downarrow&&\downarrow\\
U_{i'}&\xrightarrow{f'}&T_i\end{matrix},
$$
the morphism $f'$ is a smooth covering.\ 
\item A morphism $f:U\to T$ between placid ind-schemes is called \emph{placid} if it is co-reasonable and for some, equivalently every, reasonable presentation $T=\colim_{j\in\cJ}T^j$ of $T$, the map of placid schemes $U\times_{T}T^j\to T^j$ is co-placid.\
\end{enumerate}\end{definition}

\begin{example}\label{RT}\begin{enumerate}\item Let $T$ be a placid ind-scheme. Let $U\to T$ be either ind-smooth or an ind-pro-smooth cover. Then it is co-placid.\
\item $G_S^{S'}$ is a placid ind-scheme, and the fpqc quotient map
$$
G_S^{S'}\to G_S^{S'}/G_S
$$
is an ind-pro-smooth ind-affine affine cover of a placid (indeed, ind-finite type) ind-scheme, so is co-placid.\
\item Given $S''\subset S'\subset S$, the morphism
$$
f:G_S^{S''}\to X^S\times_{X^{S'}}G_{S'}^{S''}
$$
is co-placid. Indeed, consider $G_{S,i}:=\Map(\Delta_{S,i},G)$. We have $G_S=\nlim_{i\in\bZ_{\geq0}}G_{S,i}$ and $X^S\times_{X^{S'}}G_{S'}=\nlim_{i\in\bZ_{\geq0}}X^S\times_{X^{S'}}G_{S,i}$ and the morphism
$$
G_{S,i}\to X^{S}\times_{X^{S'}}G_{S',i}
$$
induced by the closed embeddings $\Delta_{S',i}(x)\subset\Delta_{S,i}(x)$ for any $x\in X^S(R)$. This is a smooth covering, by the same argument of Subsection \ref{prosaic} for the prosaicness of $G_S$. This shows that 
$$
g:G_S\to X^S\times_{X^{S'}}G_{S'}
$$
is co-placid. To conclude, note that the morphism $f$ is an ind-locally trivial $g$-bundle over the ind-finite type ind-scheme $G_S^{S''}/G_S=X^S\times_{X^{S'}}G_{S'}^{S''}/G_{S'}$.

\begin{warning} Morphism $g$ is \emph{not} a pro-smooth cover\footnote{Correspondingly, $f$ is \emph{not} an ind-pro-smooth cover.}. It is tempting to imagine that it is the quotient map by some group scheme $\ker(G_S\to X^S\times_{X^{S'}}G_{S'})$, but there is no such group scheme. To see this, fix some section $S\to S'$ and consider the corresponding `multi-diagonal' embedding $X^{S'}\to X^S$. Then $g$ is an isomorphism over $X^{S'}$. However, over a generic point of $X^S$, $g$ is a non-trivial projection from $G(\cO)^S\to G(\cO)^{S'}$, see Subsection \ref{notrem} for the notation.\end{warning}

\item Fix a representation $N$ of $G$ of dimension $d$. Let $\bbN:=\spec(\Sym(N^*))$ be the corresponding $G$-module, i.e. vector space in the category of schemes over $\bC$ with $G$-action. Then $\bbN_S$ is a $G_S$-module\footnote{Indeed, each $\Map(\Delta_{S,i},\bbN)$ is a $\Map(\Delta_{S,i},G)$-module and the transition maps are $G_S$-equivariant.}. We have the placid ind-scheme
$$
\widetilde{\cT}_S^{S'}:=G_S^{S'}\times_{X^S}\bbN_S.
$$
This is an ind-pro-smooth covering of its fpqc quotient (relative to $X^S$) ind-scheme
$$
\cT_S^{S'}:=G_S^{S'}\frac{\times_{X^S}}{^{G_S}}\bbN_S.
$$
This is an infinite-dimensional vector bundle over $G_S^{S'}/G_S$. On the level of $R$-points we identify
$$
\cT_S^{S'}(R)=\left\{(x,\cE,f,\widetilde{v})  \left|~\begin{matrix}   x\in X^S(R)\hfill\\
\cE\text{ a principal }G\text{-bundle over }\Delta_S(x)\hfill\\
 f\text{ a trivialization of }\cE\text{ over }\Delta_S^{S'}(x)\hfill\\
\widetilde{v}\text{ an }\bbN\text{-section of }\cE\hfill\end{matrix}\right.\right\}/\sim.
$$
Here by `an $\bbN$-section of $\cE$' we mean a section of the associated $\bbN$-bundle. Of course $\cT_S^{S'}$ is the inverse limit of vector bundles over $G_S^{S'}/G_S$:
$$
\cT_S^{S'}=\nlim_i \cT_{S,i}^{S'}
$$
where $\cT_{S,i}^{S'}$ is the associated bundle of $\bbN_{S,i}$, a vector bundle of rank $di|S|$. In particular, $\cT_S^{S'}$ is a placid ind-scheme with $\cT_S^{S',j}$ being the infinite-dimensional vector bundle $\cT_S^{S'}|_{G_S^{S',j}/G_S}$ over $G_S^{S',j}/G_S$. We identify the $R$-points of $\widetilde{\cT}_S^{S'}$ compatibly as follows:
$$
\widetilde{\cT}_S^{S'}(R)=\left\{(x,\cE,f,g,\widetilde{v})  \left|~\begin{matrix}   x\in X^S(R)\hfill\\
\cE\text{ a principal }G\text{-bundle over }\Delta_S(x)\hfill\\
 f\text{ a trivialization of }\cE\text{ over }\Delta_S^{S'}(x)\hfill\\
g\text{ a trivialization of }\cE\text{ over }\Delta_S(x)\hfill\\
\widetilde{v}\text{ an }\bbN\text{-section of }\cE\hfill\end{matrix}\right.\right\}/\sim.
$$\

\item $\bbN_S^{S'}$ is a $G_S^{S'}$-module (in ind-schemes). Therefore multiplication gives a map between the placid ind-schemes
$$
\cT_S^{S'}\to\bbN_S^{S'}.
$$
We define $\cR_S^{S'}$ to be the fiber product: 
$$
\cR_S^{S'}:=\cT_S^{S'}\times_{\bbN_S^{S'}}\bbN_S.
$$
This is an ind-scheme over $G_S^{S'}/G_S$, with $\cR_S^{S',j}:=\cT_S^{S',j}\times_{\bbN_S^{S'}}\bbN_S = \cR_S^{S'}|_{G_S^{S',j}/G_S}$. Moreover $\cR_S^{S'}$ is a \emph{vector space} over $G_S^{S'}/G_S$, but unlike $\cT_S^{S'}$ it is not a vector bundle because the fibers jump. Furthermore, $\cR_S^{S',j}$ contains $\ker(\cT_{S}^{S',j}\to \cT_{S,i}^{S',j})$ for $i$ large enough (depending on $j$), that is we have a diagram
$$
\ker(\cT_{S}^{S',j}\to \cT_{S,i}^{S',j})\subset \cR_S^{S',j}\subset \cT_S^{S',j}
$$
of vector spaces over $G_S^{S',j}/G_S$. Therefore $\cR_S^{S'}$ is placid: we may take $\cR_{S,i}^{S',j}$ to be the image in $\cT_{S,i}^{S',j}$ of $\cR_{S}^{S',j}$, for $i$ large enough. Also, $\cR_S^{S'}$ is of ind-finite codimension in $\cT_S^{S'}$, i.e. it is an ind-f.p. closed sub-ind-scheme. On the level of $R$-points, we have
$$
\cR_S^{S'}(R)=\left\{(x,\cE,f,{v})  \left|~\begin{matrix}   x\in X^S(R)\hfill\\
\cE\text{ a principal }G\text{-bundle over }\Delta_S(x)\hfill\\
 f\text{ a trivialization of }\cE\text{ over }\Delta_S^{S'}(x)\hfill\\
{v}\text{ an }\bbN\text{-section of }\cE\text{ such that }f(v)\text{ extends to }\Delta_S(x)\hfill\end{matrix}\right.\right\}/\sim.
$$
Here $f(v)$ is a section of the trivial $\bbN$-bundle on $\Delta_S^{S'}(x)$, and we require that it extends to a section of the trivial $\bbN$-bundle over $\Delta_S(x)$. Such an extension is unique if it exists. We denote the preimage of $\cR_S^{S'}$ in $\widetilde{\cT}_S^{S'}$ as $\widetilde{\cR}_S^{S'}$. It is an ind-f.p. closed sub-ind-scheme of $\widetilde{\cT}_S^{S'}$, an ind-pro-smooth cover of $\cR_S^{S'}$, its fpqc quotient (locally on $X^S$) by $G_S$, and on $R$-points we identify:
$$
\widetilde{\cR}_S^{S'}(R)=\left\{(x,\cE,f,g,{v})  \left|~\begin{matrix}   x\in X^S(R)\hfill\\
\cE\text{ a principal }G\text{-bundle over }\Delta_S(x)\hfill\\
 f\text{ a trivialization of }\cE\text{ over }\Delta_S^{S'}(x)\hfill\\
g\text{ a trivialization of }\cE\text{ over }\Delta_S(x)\hfill\\
{v}\text{ a section of the trivial }\bbN\text{-bundle on }\Delta_S(x)\\
\text{ such that }fg^{-1}(v)\text{ extends to }\Delta_S(x)\hfill\end{matrix}\right.\right\}/\sim.
$$\
\item The product (over $B$) of co-placid maps is co-placid.
\end{enumerate}\end{example}

\begin{remark}\begin{enumerate}\item Suppose $X=\bG_a$ with parameter $t$. Then $I(x_s)$ is trivialized by $t - t_{x_s}$. There is an isomorphism from $\cT_S^{S'}$ to the kernel of the covering $\cT_S^{S'}\to \cT_{S,i}^{S'}$, given on $R$-points by
$$
(x,f,\cE,\widetilde{v})\mapsto (x,f,\cE,\prod_{s\in S}(t - t_{x_s})\widetilde{v}).
$$
We call this isomorphism the \emph{fiberwise shift map} of $\cT$.\
\item The placid ind-scheme $T=G_S^{S'}/G_S,\cT_S^{S'},\cR_S^{S'}$ is \emph{special} in that one may take the smooth affine covering maps $T^j_i\to T^j_{i'}$ to be vector bundles. But placidity seems to be the more flexible definition: for instance I do not know if point (3) of Remark \ref{flexy} holds if we replace `placid' by `special'.\end{enumerate}\end{remark}

\subsection{Equivariance} \begin{enumerate}\item Note that $G_S^{S'}/G_S,\cT_S^{S'},\cR_S^{S'}$ are all acted on by $G_S$, and the various maps between them are $G_S$-equivariant. In fact, each `approximation' $G_S^{S',j}/G_S$, $\cT^j_i$, $\cR^j_i$ is acted on by some quotient $G_{S,i'}$ of $G_S$, and the transition morphisms are all $G_S$-equivariant.\
\item Suppose that $X=\bG_a$. Then we have the action of $\bG_m$ on $X$ by multiplication. It also acts diagonally on $X^S$. Therefore we may consider $\bC^*\times X^S$ as a smooth groupoid over $X^S$. The group $G_S^{S'}$ over $X^S$ is $\bC^*\times X^S$-equivariant, so we may form the semidirect product
$$
G_S^{S'}\rtimes\bC^*,
$$
a placid affine groupoid ind-scheme over $X^S$. The special case, $G_S\rtimes \bC^*$, is a prosaic affine groupoid scheme over $X^S$. Then, the $G_S$-equivariant structures of $G_S^{S'}/G_S,\cT_S^{S'},\cR_S^{S'}$ and their above approximations upgrade to $G_S\rtimes\bC^*$-equivariant structures (over $X^S$). Again all morphisms, transition or otherwise, of the previous Subsection are $G_S\rtimes\bC^*$-equivariant.\end{enumerate}

\begin{remark}\begin{enumerate}\item Let $T$ be a placid ind-scheme over some base $B$ and $\cG$ be a groupoid scheme over $B$ which acts on $T$. Then we can always choose a $\cG$-equivariant placid presentation of $T$, simply by `smoothing out' any placid presentation by the action of $\cG$. The category $\cP_{\cG}(T)$ of $\cG$-equivariant placid presentations of $T$ is filtered.\
\item Now suppose that $T$ is special. I do not know whether we can choose the special presentation of $T$ to be $\cG$-equivariant. However, if $T$ is special by virtue of being a placid vector space over some intermediate $\cG$-equivariant ind-scheme $U$ of ind-finite type, $T\to U\to B$, and $\cG$ acts linearly, then we can do it. This is what happens for $G_S^{S'},\cT_S^{S'},\cR_S^{S'}$. In the latter two cases, we can take $U=G_S^{S'}/G_S$. In the former case, we can take $U$ to be the quotient of $G_S^{S'}$ by the kernel of the surjection $G_S\to \Map(\Delta_{S,1},G)$. \end{enumerate}\end{remark}

\subsection{Dimension theories}

The following definitions can be found in \cite{Rask}, and essentially in the earlier work \cite{D}, and probably in many other texts.

\begin{definition}\begin{enumerate}\item Let $T$ be a placid scheme, with a placid presentation $T=\nlim_{i\in (\cI)^{op}}(T_i)$. A \emph{dimension theory} on $T_{\cI}$ is a function $d:\cI\to\bZ$, whose value on $i\in\cI$ will be written $d(T_i)$, satisfying the condition 
$$
d(T_{i'})-d(T_{i})=\dim(T_{i'})-\dim(T_{i})
$$
whenever $i\to i'$ in $\cI$. \
\item Take a placid presentation $T=\nlim_{i\in (\cI)^{op}}(T_i)$ of the scheme $T$ and a finer placid presentation $T=\nlim_{i_1\in (\cI_1)^{op}}(T_{i_1})$, $\cI\subset\cI_1$. We may extend a dimension theory $d$ on $T_{\cI}$ to a unique dimension theory, denoted $d$, on $T_{\cI_1}$ by setting 
$$
d(T_{i_1}):=d(T_{i'})-\dim(T_{i'})+\dim(T_{i})
$$
for any $i'\in\cI$ such that $i_1\to i'$ in $\cI_1$. We have thus constructed a filtered system
$$
(\{\text{dimension theories on }T_{\cI}\})_{T_{\cI}\in\cP}
$$
indexed by the filtered category $\cP$ of placid presentations of $T$.\
\item A \emph{dimension theory} on $T$ is an element of the colimit of the above filtered system.\
\end{enumerate}\end{definition}

Now let $f:U\to T$ be an f.p. map of placid schemes and choose Cartesian placid presentations indexed by $\cI$ as in Remark \ref{flexy}. In this presentation, a dimension theory $d$ on $T_{\cI}$ defines one on $U_{\cI}$, denoted\footnote{And often abusively denoted $d$.} $f^*d$ and given by the formula
$$
f^*d(U_i):=d(T_i).
$$
That is, we have a map $\{\text{dimension theories on }T_{\cI}\}\to \{\text{dimension theories on }U_{\cI}\}$. The composition of this with the map $\{\text{dimension theories on }U_{\cI}\}\to \{\text{dimension theories on }U\}$ factors through the map $\{\text{dimension theories on }T_{\cI}\}\to \{\text{dimension theories on }T\}$, yielding a map
$$
f^*:\{\text{dimension theories on }T\}\to \{\text{dimension theories on }U\}
$$
independent of any choices of presentation.\

Similarly, suppose $f:U\to T$ is a co-placid map of placid schemes and fix a compatible presentation $f=\nlim_{i\in\cI^{op}}(U_i\to T_i)$. The dimension theory $d$ on $T_{\cI}$ defines one on $U_{\cI}$, denoted\footnote{And certainly not denoted as $d$!} $f^!d$ and given by the formula
$$
f^!(d)(U_i):=d(T_i)+\dim(U_i)-\dim(T_i).
$$
This procedure again determines a map
$$
f^!:\{\text{dimension theories on }T\}\to \{\text{dimension theories on }U\}
$$
independent of any choices of presentation.

\begin{definition}\begin{enumerate}\item Let $T$ be a placid ind-scheme. Fix a reasonable presentation $T=\colim_{j\in\cJ}T^j$. We write this as $T^{\cJ}$. A \emph{dimension theory} on $T$ is an element of the limit of the cofiltered system
$$
(\{\text{dimension theories on }T^j\})_{j\in\cJ^{op}}
$$
with transition morphisms given by the $*$-pullback. For different reasonable presentations these limits are canonically isomorphic.\
\item Given an ind-f.p. map $f:U\to T$ of placid ind-schemes, we get a map
$$
f^*:\{\text{dimension theories on }T\}\to\{\text{dimension theories on }U\}
$$
determined by the condition that the square
$$
\begin{matrix}
\{\text{dimension theories on }T\}&\xrightarrow{f^*}&\{\text{dimension theories on }U\}\\
\downarrow&&\downarrow\\
\{\text{dimension theories on }T_j\}&\xrightarrow{(f^j)^*}&\{\text{dimension theories on }U\times_TT_j\}\end{matrix}
$$
commutes for every reasonable closed subscheme $T_j$ of $T$.
\
\item Given a co-placid map $f:U\to T$ of placid ind-schemes, we get a map
$$
f^!:\{\text{dimension theories on }T\}\to\{\text{dimension theories on }U\}.
$$
determined by the condition that the square
$$
\begin{matrix}
\{\text{dimension theories on }T\}&\xrightarrow{f^!}&\{\text{dimension theories on }U\}\\
\downarrow&&\downarrow\\
\{\text{dimension theories on }T_j\}&\xrightarrow{(f^j)^!}&\{\text{dimension theories on }U\times_TT_j\}\end{matrix}
$$
commutes for every reasonable closed subscheme $T_j$ of $T$.
\
\item Given placid ind-schemes $U$, $T$ over $B$ we get a map
$$
(-)+_B(-): \{\text{dimension theories on }U\}\times \{\text{dimension theories on }T\}\to\{\text{dimension theories on }U\times_BT\}.
$$
Indeed if $U=\colim_{j\in\cJ_U}\nlim_{i\in(\cI_U^j)^{op}}U_i^j$, $T=\colim_{j\in\cJ_T}\nlim_{i\in(\cI_T^j)^{op}}T_i^j$ are placid presentations over $B$ and $d_U$, $d_T$ are dimension theories on $U$, $T$ then $U\times_BT=\colim_{(j_U,j_T)\in\cJ_U\times\cJ_T}\nlim_{(i_U,i_T)\in(\cI_U^{j_U})^{op}\times (\cI_T^{j_T})^{op}}U_{i_U}^{j_U}\times_BT_{i_T}^{j_T}$ is a placid presentation and we define
$$
(d_U+d_T)(U_{i_U}^{j_U}\times_BT_{i_T}^{j_T}):= d_U(U_{i_U}^{j_U})+d_T(T_{i_T}^{j_T}).
$$
If $f:U\times_BT\to U\times B$ is the ind-f.p. closed embedding, then we have $(-)+_B(-)=f^*((-)+_{\spec{C}}(-))$. We will usually write $(-)+_B(-)$ simply as $(-)+(-)$.\end{enumerate}\end{definition}

\begin{remark}\begin{enumerate}\item The set of dimension theories on a connected placid ind-scheme $T$ is a $\bZ$-torsor. Thus the set of dimension theories on a general placid ind-scheme $T$ is a $\bZ^{\pi_0(T)}$-torsor. For an ind-f.p., (resp. co-placid) map $f:U\to T$ of placid ind-schemes, $f^*$ (resp. $f^!$) is a map of $\bZ^{\pi_0(T)}$-sets.\
\item In fact, there is a sheaf (in an appropriate sense) of $\bZ$-torsors on any placid ind-scheme $T$ whose set of global sections equals the set of dimension theories on $T$. We do not need to consider it since in every example of this paper, this sheaf is trivial. \end{enumerate}\end{remark}

\begin{example}\begin{enumerate}\item Recall that $\cT_S^{S'}$ is an infinite-dimensional vector bundle over $G_S^{S'}/G_S$, with `approximations' $\cT_{S,i}^{S',j}$ for $j\in \bZ_{\geq0}$ and $i\in\bZ_{\geq ?}$ for some positive integer $?$ depending on $j$ (see Example \ref{RT}). The approximation $\cT_{S,i}^{S',j}$ is a vector bundle over $G_S^{S',j}/G_S$ of rank $di|S|$. Thus $f:\cT_S^{S'}\to G_S^{S'}/G_S$ is co-placid, and since $G_S^{S'}/G_S$ is ind-finite type it has a dimension theory $d_0$ with constant value $0$. We will denote the dimension theory $f^!d_0$ on $\cT_S^{S'}$ by $\rank(\cT_S^{S'})$. We have
$$
\rank(\cT_S^{S'})(\cT_{S,i}^{S',j}):=di|S|.
$$
It would perhaps be safer to call this $\rank_{G_S^{S'}/G_S}(\cT_S^{S'})$, but the notation becomes too unwieldy. The reader should bear this in mind.\
\item Let $f: \cR_S^{S'}\to \cT_S^{S'}$ be the defining ind-f.p. closed embedding. Then we have the dimension theory $f^*\rank(\cT_S^{S'})$ on $ \cR_S^{S'}$. We will call this simply $\rank(\cT_S^{S'})$.\
\item We also have the dimension theory on $\widetilde{\cT}_S^{S'}$ obtained as the $!$-pullback of $\rank(\cT_S^{S'})$ (or of $d_0$ directly). Its $*$-pullback to $\widetilde{R}_S^{S'}$ coincides with the $!$-pullback of the dimension theory $\rank(\cT_S^{S'})$ on $\cR_S^{S'}$. These dimension theories on $\widetilde{\cT}_S^{S'}$, $\widetilde{\cR}_S^{S'}$ will both be denoted $\rank(\widetilde{\cT}_S^{S'})$.\
\item We will denote the $!$-pullback to $G_S^{S'}$ of the constantly $0$ dimension theory on $G_S^{S'}/G_S$ by $\rank(G_S^{S'})$. It satisfies 
$$
\rank(G_S^{S'})(G_{S,i}^{S',j}):= \dim_{X^S}(G_{S,i})=\dim(G_{S,i})-|S|.
$$
\end{enumerate}\end{example}

\subsection{Notational remark}\label{notrem}

We will use the same notational simplification for $G_S^{S'}$, $\cT_S^{S'}$, $\cR_S^{S'}$, $\widetilde{\cT}_S^{S'}$, $\widetilde{\cR}_S^{S'}$ as for formal neighborhoods: for instance
$$
\cR_{\{1,2\}}^{\{1\}}
$$
may be written as
$$
\cR_{12}^1.
$$
Fix a $\bC$-point $x\in X$ and a local parameter $t$ at $x$. This determines isomorphisms $\Delta_1(x)=\spec(\cO)$, $\Delta_1^1(x)=\spec(\cK)$ where $\cO=\bC[[t]]$, $\cK=\bC((t))$. The groups of $\bC$-points of $(G_1)_x$, $(G_1^1)_x$ are put in isomorphism with $G(\cO)$, $G(\cK)$. Note that since $G_1$ is pro-smooth over $X$, we have
$$
(G_1^1/G_1)_x=(G_1^1)_x/(G_1)_x.
$$
We will write informally
$$
(G_1)_x=G(\cO),
$$
$$
(G_1^1)_x=G(\cK),
$$
$$
(G_1^1/G_1)_x = Gr_G.
$$
We will write $(\cT_1^1)_x$, $(\cR_1^1)_x$ as $\cT,\cR$. Since these are ind-f.p. sub-ind-schemes of $\cT_1^1$, they have dimension theories given by pulling back $\rank(\cT_1^1)$. We will write the resulting dimension theories both as $\rank(\cT)$. These dimension theories correspond to the $\dim\bbN(\cO)$ of \cite{BFN} (although dimension theories are not explicitly used in \emph{loc. cit.}).

\subsection{Borel-Moore homology}\label{BM}

Rather than recall the general formalism of equivariant constructible derived categories on placid ind-schemes (see \cite{Rask}), we content ourselves with the following definition.

\begin{definition}\begin{enumerate}\item Let $R$ be a commutative ring\footnote{There is a clash of notation: here $R$ denotes a ring of homological coefficients (not necessarily over $\bC$), rather than a geometric test ring over $\bC$. I sincerely hope this is not a source of confusion.}. Let $T$ be a scheme of finite type over the base $B$ of finite type over $\bC$ with an action of the affine algebraic groupoid $\cG$ over $B$. Then $T^{an}$ has a $\cG^{an}$-equivariant constant sheaf $R$ and dualizing complex $\omega$ with coefficients in $R$, and we will write
$$
H^{n}_{\cG}(T,R):= H^n_{\cG^{an}}(T^{an},R) = \Hom_{D^b_{\cG^{an}}(T^{an})}(R,\Sigma^nR).
$$
$$
H_{n}^{BM,\cG}(T,R):= H^n_{\cG^{an}}(T^{an},\omega) = \Hom_{D^b_{\cG^{an}}(T^{an})}(R,\Sigma^n\omega).
$$\
\item Now suppose that $\cG$ is an affine groupoid scheme over $B$. We may write $\cG$ as the limit of its fpqc quotient affine algebraic groupoids $(\cG_i)_{i\in\cI^{op}}$. We may assume that action of $\cG$ on $T$ factors through each $\cG_i$. We set
$$
H^*_{\cG}(T,R):=\colim_{i\in\cI}H^*_{\cG_i}(T,R)
$$
and
$$
H_*^{BM,\cG}(T,R):=\colim_{i\in\cI}H_*^{BM,\cG_i}(T,R).
$$
Here we have used the fact that for any $i\to i'$ in $\cI$, the $\cG_{i'}$-equivariant complexes obtained from the $\cG_i$-equivariant constant sheaf, respectively dualizing complex, are canonically isomorphic to their $\cG_{i'}$-equivariant counterparts. Thus these restriction functors determine the maps of equivariant cohomology, respectively Borel-Moore homology, which we take colimits over.\
\item Let $T$ be a placid scheme over some base $B$ of finite type over $\bC$ and let $\cG$ be an affine algebraic groupoid over $B$ which acts on $T$. Let $d$ be a dimension theory on $T$. Then the $2d$\emph{-shifted} $\cG$\emph{-equivariant Borel-Moore homology of} $T$, $H_{*-2d}^{BM,\cG}(T,R)$ is defined as follows. Let $T = \nlim_{i\in (\cI)^{op}}(T_i)$ be a $\cG$-equivariant placid presentation of $T$. Observe that pullback defines a graded map of $R$-modules:
$$
H_{*-2d(T_i)}^{BM,\cG}(T_i,R)\to H_{*-2d(T_{i'})}^{BM,\cG}(T_{i'},R)
$$
whenever $i\to i'$ in $\cI$, since $T_{i'}\to T_i$ is a $d(T_{i'})-d(T_i)$-dimensional smooth covering. We then set
$$
H_{*-2d}^{BM,\cG}(T,R):= \colim_{i\in\cI}H_{*-2d(T_i)}^{BM,\cG}(T_i,R).
$$
For different choices of placid presentation, we get canonically isomorphic answers, which justifies the definition.\
\item Let $f:T^0\to T^1$ be a $\cG$-equivariant ind-f.p. embedding of placid ind-schemes over $B$. Let $d$ be a dimension theory on $T^1$. Then there is a pushforward map:
$$
f_*: H_{*-2f^*d}^{BM,\cG}(T^0,R)\to H_{*-2d}^{BM,\cG}(T^1,R)
$$
defined by choosing Cartesian placid presentations of $T^0,T^1$ indexed by $\cI$ and observing that for each $i\in\cI$ the diagram
$$
\begin{matrix}
H_{*-2d(T_{i'}^0)}^{BM,\cG}(T_{i'}^0,R) & \xrightarrow{f_*} & H_{*-2d(T_{i'}^1)}^{BM,\cG}(T_{i'}^1,R) &  \\ 
\uparrow &  & \uparrow \\ 
H_{*-2d(T_{i}^0)} ^{BM,\cG}(T_{i}^0,R) & \xrightarrow{f_*} & H_{*-2d(T_{i}^0)}^{BM,\cG}(T_{i}^1,R) & 
\end{matrix}
$$
commutes. Here the horizontal maps are pushforwards maps along the closed embeddings $T_i^0\to T_i^1$, while the vertical maps are the pullback maps of point (1). The resulting map is independent of any choices we have made.\
\item Now let $T$ be a placid ind-scheme over some base $B$ of finite type over $\bC$ and let $\cG$ be an affine algebraic groupoid over $B$ which acts on $T$. Let $d$ be a dimension theory on $T$. Then we set 
$$
H_{*-2d}^{BM,\cG}(T,R):=\colim_{j\in\cJ}H_{*-2f^*d}^{BM,\cG}(T^j,R)
$$
using the pushforward maps of point (2) of this definition, for any choice $T = \colim_{j\in\cJ}(T^j)$ of $\cG$-equivariant reasonable presentation of $T$. For different presentations we get canonically isomorphic colimits. Here we have written $d$ for the unique dimension theory on $T^j$ compatible with the dimension $d$ on $T$.\end{enumerate}\end{definition}

\begin{remark}\begin{enumerate}
\item Since $\cG$ is a pro-smooth covering groupoid, its action on any finite type approximation $T_i^j$ to the placid ind-scheme factors through the quotient $\cH$ by some (pro-smooth covering) subgroup. We then have
$$
H_{*-2d(T_i^j)}^{BM,\cG}(T_i^j,R)=H^*_{\cG}(B,R)\otimes_{H^*_{\cH}(B,R)}H_{*-2d(T_i^j)}^{BM,\cH}(T_i^j,R).
$$\
\item If moreover $\cG$ is prosaic, then we can choose the sub-group in question to be also pro-unipotent, in which case we have $H^*_{\cG}(B,R)\cong {H^*_{\cH}(B,R)}$ so that
$$
H_{*-2d(T_i^j)}^{BM,\cG}(T_i^j,R)=H_{*-2d(T_i^j)}^{BM,\cH}(T_i^j,R).
$$\
\item It may even happen that we can choose a section $\cH\to\cG$. Take for example $X=\bG_a$, $\cG=G_{1}\rtimes\bC^*$, $\cH=G\times\bC^*\times X$ where $G$ embeds in $G_{1}$ as the subgroup of constant functions. In this case, $\cH$ acts on all of $T$, and we have
$$
H_{*-2d}^{BM,\cG}(T,R)=H_{*-2d}^{BM,\cH}(T,R).
$$
Though it may give a psychological advantage since $\cH$ is an actual smooth algebraic groupoid, this reduction is usually technically unhelpful.
\end{enumerate}
\end{remark}

The following procedures in ordinary equivariant Borel-Moore homology are also defined in the world of placid ind-schemes. Fix a $\cG$-equivariant placid ind-scheme $T$ over $B$ (of finite type over $\bC$) and a dimension theory $d$ on $T$.

\begin{enumerate}\item\textbf{Change of groupoid base.} Suppose that we have some finite type $\cG$-space $A$ over $B$. Then there exists a semi-direct groupoid scheme $\cG\ltimes_B A$, affine pro-smooth over $A$. Suppose that $T\to B$ factors through $A$. Then $T$ is also $\cG\ltimes_BA$-equivariant (over $A$) and we have a canonical isomorphism
$$
H_{*-2d}^{BM,\cG}(T,R)\xrightarrow{\sim} H_{*-2d}^{BM,\cG\ltimes_BA}(T,R).
$$
\item \textbf{Open restriction.} Let $A\subset B$ be a $\cG$-equivariant open subscheme and let $j:T|_A\to T$ be the $\cG$-equivariant ind-f.p. ind-open embedding\footnote{We can do something along these lines for more general ind-smooth maps.} of placid ind-schemes over $A$. The restriction to $A$ of a $\cG$-equivariant placid presentation of $T$ is a $\cG|_A$-equivariant placid presentation of $T|_A$, and applying the ordinary open restriction in Borel-Moore homology one obtains a map
$$
j^!: H_{*-2d}^{BM,\cG}(T,R)\to H_{*-2j^*d}^{BM,\cG}(T|_A,R)   
$$
One usually goes on to compose this map with the isomorphism $H_{*-2j^*d}^{BM,\cG}(T|_A,R)\cong H_{*-2j^*d}^{BM,\cG\ltimes_BA}(T|_A,R)$. \
\item \textbf{Proper pushforward.} Let $f:U\to T$ be a $\cG$-equivariant ind-proper (in particular ind-f.p.) map of placid ind-schemes over $B$. By choosing a $\cG$-equivariant Cartesian placid presentation and applying the ordinary proper pushforward in Borel-Moore homology, one obtains a map
$$
f_*: H_{*-2f^*d}^{BM,\cG}(U,R)\to H_{*-2d}^{BM,\cG}(T,R).
$$ \
\item \textbf{Restriction of equivariance.} Let $\cH\to \cG$ be a morphism of groupoid schemes over $B$. Then we have `restriction of equivariance' maps
$$
H_{*-2d}^{BM,\cG}(T,R)\to H_{*-2d}^{BM,\cH}(T,R).
$$\
\item \textbf{Co-placid restriction.} Let $f:U\to T$ be a $\cG$-equivariant co-placid map of placid ind-schemes over $B$. By choosing a $\cG$-equivariant compatible presentation of $f$ and applying the ordinary smooth pullback in Borel-Moore homology, one obtains a map
$$
f^!: H_{*-2d}^{BM,\cG}(T,R)\to H_{*-2f^!d}^{BM,\cG}(U,R).
$$
\
\item \textbf{Restriction with supports.} Let $p:T'\to U'$ be a $\cG$-equivariant co-placid map of placid ind-schemes over $B$, and let $f':U'\to T'$ be a $\cG$-equivariant section of $p$. Let $g:T\to T'$ be ind-f.p. and let $U=T\times_{T'}U'$, so that we have a Cartesian square:
$$
\begin{matrix}U&\xrightarrow{f}& T\\
\downarrow{g'}&&\downarrow{g}\\
U'&\xrightarrow{f'}&T'.\end{matrix}
$$
Suppose we have a dimension theory $d'$ on $U'$ such that $g^*p^!d'=d$, and set $d^g:=(g')^*d'$. We can choose a $\cG$-equivariant compatible presentation of this Cartesian square, i.e. write it as
$$
\colim_{j\in\cJ}\nlim_{i\in\cI^j}\left(\begin{matrix}U_i^j&\xrightarrow{f_i^j}& T_i^j\\
\downarrow{}&&\downarrow{}\\
(U')_i^j&\xrightarrow{(f')_i^j}&(T')_i^j.\end{matrix}\right)
$$
such that the induced presentations of the vertical morphisms $U\to U'$, $T\to T'$ are Cartesian, and such that there exists a $\cG$-equivariant presentation $\colim_{j\in\cJ}\nlim_{i\in\cI^j}((T')_i^j\xrightarrow{p_i^j}(U')_i^j)$ of $p$. Then $p_i^j$ is a smooth covering and $(f')_i^j$ is its section, and we have a `restriction with supports' morphism $(f_i^j)^!:H_*^{BM,\cG}(T_i^j,R)\to H_{*-2\dim(p_i^j)}^{BM,\cG}(U_i^j,R)$, which assemble to give a morphism
$$
f^!:H_{*-2d}^{BM,\cG}(T,R)\to H_{*-2d^g}^{BM,\cG}(U,R).
$$\
\item \textbf{Averaging.} Suppose that $\cH\subset \cG$ is a normal subgroup over $B$ of finite index. Then we have the averaging maps
$$
H_{*-2d}^{BM,\cH}(T,R)\to H_{*-2d}^{BM,\cG}(T,R).
$$\
\item \textbf{Specialization.} Suppose now that $B=\bG_a$ and $\cG$ is a group scheme over $B=\bG_a$. We have the closed subscheme $i:\{0\}\to\bG_a$, and the complementary open subscheme $j:\bG_m\to\bG_a$. Let us write $T^*,\cG^*$ for the restrictions to $\bG_m$ and $T|_0,\cG|_0$ for the restrictions to $\{0\}$. Choose a $\cG$-equivariant placid presentation
$$
T=\colim_{j\in\cJ}\nlim_{i\in(\cI^j)^{op}}T^j_i
$$
of $T$ over $\bG_a$. Then by restriction we obtain a $\cG^*$-equivariant placid presentation
$$
T^*=\colim_{j\in\cJ}\nlim_{i\in(\cI^j)^{op}}(T^j_i)^*
$$
of $T^*$ over $\bG_a$ and a $\cG|_0$-equivariant placid presentation
$$
T|_0=\colim_{j\in\cJ}\nlim_{i\in(\cI^j)^{op}}(T^j_i)|_0
$$
of $T|_0$ over $\{0\}$, both of which are Cartesian with the original placid presentation. Since $\cG$ is pro-smooth covering, we have specialization maps
$$
s^j_i:H_{*}^{BM,\cG^*}((T^j_i)^*,R)\to H_{*+2}^{BM,\cG|_0}((T^j_i)|_0,R)
$$
which are compatible, so yield
$$
s:H_{*-2j^*d}^{BM,\cG^*}(T^*,R)\to H_{*+2-2i^*d}^{BM,\cG|_0}(T|_0,R).
$$
We will write $d^*:=j^*d$, $d|_0:=i^*d$.\
\item \textbf{Steenrod's construction.} Let $\cG$ be an affine algebraic group and $B=*$\footnote{We will not need the relative situation.}. Write $pd$ for the dimension theory $\underbrace{d+\ldots+d}_{\text{$p$ times}}$ on the $\cG^{\mup}\rtimes\mup$-equivariant placid ind-scheme $T^{\mup}$. We have non-linear maps
$$
St^{BM}:H_{n-2d}^{BM,\cG}(T,R)\to H_{pn-2pd}^{BM,\cG^{\mup}\rtimes\mup}(T^{\mup},R)
$$
which are monoidal with respect to the map $St^H:H^n_{\cG}(*,R)\to H^{pn}_{\cG^{\mup}\rtimes\mup}(*,R)$ and whose discrepancy from additivity is averaged from $H_{pn-2pd}^{BM,\cG^{\mup}}(T^{\mup},R)$. If $R$ is a perfect field $k$ of characteristic $p$, we have the non-linear graded $St_{ex}$-monoidal maps
$$
St^{BM}_{ex}:  H_{*}^{BM,\cG}(T,k)^{(1)}\to H_{*}^{BM,\cG^{\mup}\rtimes\mup}(T^{\mup},k).
$$
\end{enumerate}

The following facts carry over from the ordinary case.

\begin{enumerate}
\item \textbf{Descent.}\begin{enumerate}\item Suppose that $\cH$ is a normal (pro-smooth covering) subgroup of $\cG$ over $B$, and that the fpqc quotient group $\cG/\cH$ exists as a pro-smooth covering group over $B$. The main example here is $\cG=\cH\times_B\cL$ for pro-smooth covering groups $\cH$, $\cL$ over $B$. Suppose that $\cH$ acts freely on $T$, so that in particular the fpqc quotient (relative to $B$) $T/\cH$ exists as a $\cG/\cH$-equivariant placid ind-scheme over $B$. So the quotient map $f:T\to T/\cH$ is $\cG$-equivariant and co-placid. Suppose that there exists a dimension theory $d'$ on $T/\cH$ satisfying $f^!d'=d$. Then the composition
$$
H_{*-2d'}^{BM,\cG/\cH}(T/\cH,R)\xrightarrow{\text{}}H_{*-2d'}^{BM,\cG}(T/\cH,R)\xrightarrow{f^!}H_{*-2d}^{BM,\cG}(T,R)
$$
is an isomorphism.\
\item Suppose that $\cG=\cL\rtimes Q$ where $Q$ is an algebraic group acting regularly\footnote{i.e. the stabilizer groupoid $Q_B:=B\times_B(Q\ltimes B)$ is smooth over $B$.} on $B$, and $\cL$ is a $Q$-equivariant pro-smooth covering group over $B$. Thus the maximal subgroup $\cH=\cL\rtimes_BQ_B$ is a pro-smooth covering group over $B$, and the quotient groupoid $P:=\cG/\cH=(Q\ltimes B)/Q_B$ is smooth over $B$. Let $A=B/P$, $\pi:B\to A$. Suppose that we are given a slice $i:A\subset B$. Consider the placid $\cG\times_B\pi^*(\cH|_A)$-equivariant ind-scheme 
$$
\cG|_A\times_AT|_A.
$$
On the one hand, the normal subgroup $\pi^*(\cH|_A)$ acts freely and the quotient is $T$, giving
$$
H_{*-2d}^{BM,\cG}(T,R)\xrightarrow{\sim} H_{*-2d'}^{BM,\cG\times_A\cH|_A}(\cG|_A\times_AT|_A,R)
$$
where $d'$ is the $!$-pullback of $d$. On the other hand, the normal subgroup $\cL$ acts freely, and the quotient is $Q\times T|_A$ with its residual action of $(Q\ltimes B)\times_B\pi^*(\cH|_A) = (Q\times\cH|_A)\ltimes_AB$, where $\cH|_A$ acts trivially on $B$ over $A$. Thus we have the isomorphisms
$$\begin{matrix}
H_{*-2d''}^{BM,\cH|_A}(T|_A,R) & \xrightarrow{\sim} & H_{*-2d''-2\dim Q}^{Q\times \cH|_A}(Q\times T|_A,R)\\
&&\downarrow{\vsim}\\
H_{*-2d'}^{BM,\cG\times_A\cH|_A}(\cG|_A\times_AT|_A,R) & \xleftarrow{\sim} & H_{*-2d''-2\dim Q}^{(Q\times \cH|_A)\ltimes_AB}(Q\times T|_A,R)\\
 \end{matrix}
$$
if $d''$ is a dimension theory whose $!$-pullback along $\cG|_A\times_AT|_A\to Q\times T|_A\to T|_A$ equals $d'$. Thus we obtain an isomorphism
$$
H_{*-2d''}^{BM,\cH|_A}(T|_A,R)\xrightarrow{\sim} H_{*-2d}^{BM,\cG}(T,R).
$$
\end{enumerate}

\item \textbf{Compatibilities.} These various maps between Borel-Moore homology groups all commute with each other, whenever this makes sense. For instance, we have:
\begin{enumerate}\item If $f:U\to T$, resp. $p: V\to T$, are ind-proper, resp. co-placid, $\cG$-equivariant maps, so that we have a $\cG$-equivariant Cartesian square
$$
\begin{matrix}W&\xrightarrow{f'}&V\\
\downarrow{p'}&&\downarrow{p}\\
U&\xrightarrow{f}&T\end{matrix}
$$
then we have $p^!f_*=f'_*(p')^!$.\

\item Specialization commutes with proper pushforward, restriction of equivariance, co-placid restriction and averaging. To spell this out in the most complicated case: take a `restriction with supports framework' as in (6) above with $B=\bG_a$ and $\cG$ a group. We may restrict all the data over $\{0\}$ or over $\bG_a-\{0\}$, and obtain again `restriction with supports frameworks'. We therefore obtain a square
$$
\begin{matrix}
H_{*-2d^*}^{BM,\cG^*}(T^*,R) & \xrightarrow{s_T} & H_{*-2d|_0}^{BM,\cG|_0}(T|_0,R)\\
\downarrow{^{(f|_{\bG_a-\{0\}})^!}}&& \downarrow{^{(f|_{\{0\}})^!}}\\
H_{*-2(d^g)^*}^{BM,\cG^*}(U^*,R) &\xrightarrow{s_U}& H_{*-2(d^g)|_0}^{BM,\cG|_0}(U|_0,R)
\end{matrix}
$$
which is commutative.\
\item Averaging commutes with change of groupoid base, open restriction, proper pushforward, restriction of equivariance\footnote{That is, if $\cH\subset\cG$ is of finite index over $B$ and $\cG'\to\cG$ is a map, the fiber product $\cH\times_{\cG}\cG'$ is of finite index in $\cG'$, and the two possible maps from the $\cH$-equivariant BM homology to the $\cG'$-equivariant homology coincide. For some reason, we have steadfastly avoided using quotient stacks. If we had used them, this would be an example of proper base change.}, co-placid restriction. 
\end{enumerate}
\end{enumerate}


\subsection{The branch}

We assume from now on that $X=\bG_a$ with global parameter $t$. Thus the $0$-fibers $G(\cO)$ of $G_1$, $\cR$ of $\cR_1^1$, \emph{etc.}, have actions of $\bC^*$. We assume also that $G$ is reductive, so that $G_S^{S'}/G_S$ is ind-projective over $X^S$. We recall the definitions of \cite{BFN} with respect to our notations.

\begin{definition}\begin{enumerate}\item The \emph{Coulomb branch} (over $R$) is the graded $H^*_{G(\cO)}(*,R)$-module
$$
A^*:=H_{*-2\rank(\cT)}^{BM,G(\cO)}(\cR,R).
$$
\
\item The \emph{quantum Coulomb branch} (over $R$) is the graded $H^*_{G(\cO)\rtimes\bC^*}(*,R)=H^*_{G(\cO)}(*,R)[\hbar]$-module
$$
A_\hbar^*:=H_{*-2\rank(\cT)}^{BM,G(\cO)\rtimes\bC^*}(\cR,R).
$$
Here $\hbar$ has degree $2$.\
\item We will often write $A^*,A_\hbar^*$ as simply $A,A_\hbar$.
\end{enumerate}\end{definition}

\begin{lem}$A^*$, $A_\hbar^*$ are evenly graded and free over $H^*_{G(\cO)}(*,R)$, $H^*_{G(\cO)}(*,R)[\hbar]$. We have a canonical isomorphism $A_\hbar^*/\hbar\cong A_\hbar$.\end{lem}

\begin{proof}The proof in \cite{BFN} in the case $R=\bC$ works for any $R$. The essential point is that the equivariant parameters are in even degrees, and an equivariant placid presentation may be chosen such that each `approximation' has a complex cell decomposition.\end{proof}

We will also consider $\cA^*: = H_{*-2\rank(\cT)}^{BM,G(\cO)\rtimes\mup}(\cR,R)$. The same proof shows that the natural map
$$
H^*_{\mup}(*,R)\otimes_{R[\hbar]}A^*_\hbar\to \cA^*
$$
is an isomorphism. In particular in the case $R=\bfp$ we have $\cA^*=A^*_\hbar[a]$. We have an averaging map $A^*\to\cA^*$, which after identifying $A^*=R\otimes_{R[\hbar]}A^*_\hbar$, $\cA^*=H^*_{\mup}(*,R)\otimes_{R[\hbar]}A^*_\hbar$ is induced by the averaging map of $R[\hbar]$-modules $R\to H^*_{\mu}(*,R)$. This is the map which multiplies by $p$ in degree $0$. Therefore in the case $R=\bfp$, the averaging map equals $0$.

\begin{remark}\begin{enumerate}\item In \cite{BFN} $A^*,A_\hbar^*$ are given ring structures\footnote{$\cA^*$ is also a ring in the same way.} by a form of convolution in Borel-Moore homology. They show\footnote{For $R=\bC$, but it is true for any $R$.} that $A^*$ is commutative and $A_\hbar^*$ is an $\hbar$-quantization of $A^*$. We will recall the construction in the course of the proof of our main theorem. The idea (due originally to Beilinson and Drinfeld \cite{BD}) is to express the multiplication in $A^*$ by a manifestly commutative specialization map.\
\item Recall that $\cT$ is the fpqc quotient $G(\cK)\frac{\times}{^{G(\cO)}} \bbN(\cO)$ of the placid ind-scheme $\widetilde{T}:=G(\cK)\times \bbN(\cO)$. Both are ind-pro-smooth covers\footnote{In fact, ind-pro-smooth ind-fiber bundles.} of $Gr_G$, and have respective dimension theories $\rank(\cT)$, $\rank(\widetilde{\cT})$ given by their ranks over $Gr_G$. Let us denote by $\widetilde{\cR}$ the corresponding $G(\cO)$-bundle over $\cR$; it has a compatible dimension theory $\rank(\widetilde{\cT})$. By descent, we have the isomorphism
$$
A^*_\hbar=H_{*-2\rank(\cT)}^{BM,G(\cO)\rtimes\bC^*}(\cR,R)\xrightarrow{\sim} H_{*-2\rank(\widetilde{\cT})}^{BM,(G(\cO)\times G(\cO))\rtimes\bC^*}(\widetilde{\cR},R),
$$
and similarly for $A^*$. This shows that $A^*,A_\hbar^*$ have two module structures over $H^*_{G(\cO)}(*,R)$. In fact, these module structures coincide with the left- and right- multiplication by a subalgebra $H^*_{G(\cO)}(*,R)\subset A^*,A^*_\hbar$. Since $A^*$ is commutative, these two module structures coincide everywhere. However $H^*_{G(\cO)}(*,R)$ is \emph{not} in the center of $A^*_\hbar$, so these two module structures are different there.\end{enumerate}\end{remark}

\subsection{The map $F_\hbar$}\label{alpha}We will set $R=\bfp$ from now on. The rest of this section is devoted to the proof (and explanation) of the following theorem:

\begin{thm}\label{mainthm} $A_\hbar$ is a Frobenius-constant quantization of $A$.\end{thm}

We will construct the requisite map $F_\hbar$ using Steenrod's construction and a specialization map. First we introduce some new notation. We will fix $X=\bG_a$, with parameter $t$. We will use another base curve $Y=\bG_a$ with parameter $t^p$. We have $Y=X//\mup$, where $\mup\subset \bC^*$ acts on $X$ through the character $\chi$ given by restricting the weight one action of $\bC^*$. Let $\pi:X\to Y$ be the quotient map; under the identifications $X=\bG_a=Y$, $\pi$ is the $p^{th}$-power map, and is $\bC^*$-equivariant when $\bC^*$ acts on $Y$ with weight $p$. For $y\in Y^S(R)$, we will use $t^p$ to identify the coordinates $y_s$ with elements of $R$. Then for $y\in Y^S(R)$, we have the affine scheme 
$$
\pi^*\Delta_S^{S'}(y) = \spec(R[t][[\prod_{s\in S}(t^p-{y_s})]][\prod_{s\in S'}(t^p-y_{s'})^{-1}]).
$$
Now $\chi$ determines a `twisted-diagonal' embedding $X\subset X^{\mup}$ as the $\chi$-eigenline for the cyclic action of $\mup$. Let $\alpha$ be one of the symbols $G$, $\bbN$, $\cT$, $\cR$, $\widetilde{\cT}$, $\widetilde{\cR}$. Then we will set
$$
\alpha_{(p)}:=\alpha_{\mup}//_{\chi}\mup
$$
$$
\alpha_{(p)}^{(p)}:=\alpha_{\mup}^{\mup}//_{\chi}\mup.
$$
Here the symbol `$//_{\chi}\mup$' means `restrict along the twisted-diagonal then take categorical quotient by $\mup$'. The action of $\mup$ in question is the one that does not involve loop-rotation, i.e. that which scales $x\in X(R)$ but does not change any of the data $\cE,f,v$ etc. These are all placid ind-schemes over $Y$, and  behave in essentially the same way as their earlier counterparts: $G_{(p)}$ is an affine pro-smooth covering group scheme over $Y$, $Gr_{(p)}:=G_{(p)}^{(p)}/G_{(p)}$ is an ind-projective ind-scheme over $Y$, $\cT_{(p)}^{(p)}$ is an infinite-dimensional vector bundle over $Gr_{(p)}$, $\cR_{(p)}^{(p)}$ is its sub-vector space of ind-finite codimension over $Gr_{(p)}$. They are all $G_{(p)}\rtimes\bC^*$-equivariant. They also have chosen dimension theories, denoted $\rank(G_{(p)}^{(p)})$, $\rank(\cT_{(p)}^{(p)})$, $\rank(\widetilde{\cT_{(p)}^{(p)}})$ etc., which are compatible with each other in the same way as for the $\alpha_S^{S'}$, and compatible with the dimension theories on $\alpha_{\mup}^{\mup}$ in the natural way\footnote{That is, the $*$- (or $!$-)pullback along the $\mup$-fppf quotient map of the chosen dimension theory on $\alpha_{(p)}^{(p)}$ coincides with the $*$-pullback along the ind-f.p. closed embedding $X\times_{X^{\mup}}\alpha_{\mup}^{\mup} \to \alpha_{\mup}^{\mup}$ of the chosen dimension theory on $\alpha_{\mup}^{\mup}$.}. Unlike $\alpha_S^{S'}$ which is globally trivial over the coincidence-free open subset of $X^S$, these spaces $\alpha_{(p)}$, $\alpha_{(p)}^{(p)}$ are only \emph{locally} trivial away from $\{0\}$. On the level of $R$-points, they admit similar interpretations to $\alpha_S^{S'}$, only involving $\pi^*\Delta_1(y)$, $\pi^*\Delta_1^1(y)$. For example, we have:
$$
{\cR}_{(p)}^{(p)}(R):=\left\{(y,\cE,f,v)  \left|~\begin{matrix}   y\in Y(R)\hfill\\
\cE\text{ a principal }G\text{-bundle over }\pi^*\Delta_1(y)\hfill\\
 f\text{ a trivialization of }\cE\text{ over }\pi^*\Delta_1^{1}(y)\hfill\\
{v}\text{ an }\bbN\text{-section of }\cE\text{ such that }f(v)\text{ extends to }\pi^*\Delta_1(y)\hfill\end{matrix}\right.\right\}/\sim.
$$
The fiber of this space over $\{0\}$ is a copy of $\cR$, while the fibers over $Y-\{0\}:=Y^*$ are copies of $\cR^p$. 
Another example is 
$$
{G}_{(p)}(R):=\left\{(y,g)  \left|~\begin{matrix}   y\in Y(R)\hfill\\
g:\pi^*\Delta_1(y)\to G\hfill\end{matrix}\right.\right\}.
$$
The action of $G_{(p)}\rtimes\bC^*$ on $\cR_{(p)}^{(p)}$ is given as:
$$
\begin{matrix}
\hfill (y,g).(y,\cE,f,v) & = & (y,\cE,g\circ f,v)\hfill \\
\hfill z.(y,\cE,f,v) & = & (z^py,z_*\cE,z_*f,z_*v)\hfill 
\end{matrix}
$$
where for $z\in R^\times$, $z_*$ denotes the pushforward along the multiplication-by-$z$ endomorphism of $X_R$, $t\mapsto zt$, which transforms $\pi^*\Delta_{1}(y)$ into $\pi^*\Delta_{1}(z^py)$. The key new feature is that $G_{(p)}\rtimes\mup$ is a \emph{subgroup}\footnote{In fact, $G_{(p)}\rtimes\mup$ is one component of the maximal subgroup of $G_{(p)}\rtimes\bC^*$, the other being $\{0\}\times\bC^*$. Contrast with $G_S\rtimes\bC^*$, whose maximal subgroup is $G_S\cup (\{0\}\times\bC^*)$. } of $G_{(p)}\rtimes\bC^*$.\

Note that the $0$-fiber of $G_{(p)}\rtimes\bC^*\acts\cR_{(p)}^{(p)}$ is identified with $G(\cO)\rtimes\bC^*\acts\cR$ ($\bC^*$ acting in the usual way, i.e. with weight $1$ on $t\in \cO$). Meanwhile, the $1$-fiber of $G_{(p)}\rtimes\mup\acts \cR_{(p)}^{(p)}$ is identified with $G(\cO)^{\mup}\rtimes\mup\acts\cR^{\mup}$, where now $\mup$ acts in the usual cyclic way of Section \ref{sec2} (without any loop-rotation). This latter identification comes from the defining identification of $\pi^*\{1\}$ with $\mup$.\

\begin{warning}Notation $\cR^{\mup}$ means $\Map(\mup,\cR)$. This $\mup$-superscript is not to be confused with the $\mup$-superscript in $\cR_{\mup}^{\mup}$, where it indicates a set of allowed poles as in $\cR_S^{S'}$.\end{warning}

Consider the following composition $F'_\hbar: A^n\to\cA^{pn}$.

\begin{equation}\label{fhbar}
\begin{tikzcd}[column sep=6pc]
A^n = H_{n-2\rank(\cT)}^{BM,G(\cO)}(\cR,\bfp) \arrow{r}{\text{Steenrod}} & 
 H_{pn-2p\rank(\cT)}^{BM,G(\cO)^{\mup}\rtimes\mup}(\cR^{\mup},\bfp) \arrow{d}{\text{Descent isomorphism (b)}}\\
H_{pn-2-2p\rank(\cT_{(p)}^{(p)})^*}^{BM,G_{(p)}^*\rtimes\mup}((\cR_{(p)}^{(p)})^*,\bfp) \arrow{d}{\text{Specialize}}&
 \arrow{l}{\text{Restrict}}  H_{pn-2-2\rank(\cT_{(p)}^{(p)})^*}^{BM,G_{(p)}^*\rtimes\bC^*}((\cR_{(p)}^{(p)})^*,\bfp)\\
H_{pn-2\rank(\cT)}^{BM,G(\cO)\rtimes\mup}(\cR,\bfp)=\cA^{pn}. & 
\end{tikzcd}
\end{equation}

For degree reasons it factors through the inclusion $A^*_\hbar\subset A^*_\hbar[a]=\cA^*$. We may form the graded $H^*_{G(\cO)}(*,\bfp)^{(1)}$-module $(A^*)^{(1)}$, and we have a map of $\bZ$-graded multiplicative $\bfp$-sets
$$
F_\hbar:(A^*)^{(1)}\to A_\hbar^*.
$$

\subsection{Linearity}\label{linearity!}We have the following:

\begin{prop} $F_\hbar$ is $St_{in}$-linear. That is, it is linear and transports multiplication by $r\in H^m_{G(\cO)}(*,\bfp)$ to multiplication by $St_{in}(r)\in H^{pm}_{G(\cO)\rtimes\bC^*}(*,\bfp)\subset H^{pm}_{G(\cO)\rtimes\bC^*}(*,\bfp).$\end{prop}

\begin{proof}First we show that $F_\hbar$ is $St_{in}$-multiplicative. Recall that restriction of equivariance commutes with specialization. We have a closed embedding from the $Y$-version of $G_{\{1\}}$ to $G_{(p)}$ determined by the formula
$$
(y,g:\Delta_{\{1\}}(y)\to G)\mapsto (y,g\circ\pi: \pi^*\Delta_{\{1\}}(y)\to G).
$$
Over $0$, this is identified with the embedding $G(\cO)\to G(\cO)$, $t\mapsto t$. Over $1$ this is identified with the diagonal embedding $G(\cO)\to G(\cO)^{\mup}$. Since the restriction of equivariance along the former embedding,
$$
H_{*-2\rank(\cT)}^{BM,G(\cO)\rtimes\mup}(\cR,\bfp)\to H_{*-2\rank(\cT)}^{BM,G(\cO)\rtimes\mup}(\cR,\bfp),
$$
is an isomorphism, it follows that the map $H_{*-2-2p\rank(\cT_{(p)}^{(p)})^*}^{BM,G_{(p)}^*\rtimes\mup}((\cR_{(p)}^{(p)})^*,\bfp)\to \cA^{*}$ factors as
$$
H_{*-2-2p\rank(\cT_{(p)}^{(p)})^*}^{BM,G_{(p)}^*\rtimes\mup}((\cR_{(p)}^{(p)})^*,\bfp)\xrightarrow{\text{Restrict}} H_{*-2-2p\rank(\cT_{(p)}^{(p)})^*}^{BM,G_{\{1\}}^*\rtimes\mup}((\cR_{(p)}^{(p)})^*,\bfp)\xrightarrow{\text{Specialize}} \cA^{*}.
$$
Since $G_{\{1\}}\rtimes\mup$ is a constant group over $Y$ with fibers $G(\cO)\rtimes\mup$, this latter specialization map is $H^*_{G(\cO)\rtimes\mup}(*,\bfp)$-linear. Certainly the descent isomorphism (b) and the restriction of equivariance from $\bC^*$ to $\mup$ in the diagram defining $F'_\hbar$ commute with restriction of equivariance from $G_{(p)}$ to $G_{\{1\}}$. It follows that $F_\hbar$ is $St_{in}$-multiplicative, by definition of $St_{in}$.\

Now we show linearity. Averaging (over $\mup$) commutes with the descent isomorphism (b) and restriction of equivariance in that we have a commutative diagram:

\begin{equation*}
\begin{tikzcd}[column sep=6pc]
H_{pn-2p\rank(\cT)}^{BM,G(\cO)^{\mup}}(\cR^{\mup},\bfp) \arrow{r}{\text{Averaging}} \arrow{d}{\text{Descent isom.}} &
H_{pn-2p\rank(\cT)}^{BM,G(\cO)^{\mup}\rtimes\mup}(\cR^{\mup},\bfp) \arrow{d}{\text{Descent isom.}} \\
H_{pn-2-2\pi^*\rank(\cT_{(p)}^{(p)})^*}^{BM,\pi^*G_{(p)}^*\rtimes\bC^*}(\pi^*(\cR_{(p)}^{(p)})^*,\bfp) \arrow{r}{\text{Averaging}} \arrow{d}{\text{Restriction}} &
H_{pn-2-2\rank(\cT_{(p)}^{(p)})^*}^{BM,G_{(p)}^*\rtimes\bC^*}((\cR_{(p)}^{(p)})^*,\bfp) \arrow{d}{\text{Restriction}}\\
H_{pn-2-2\pi^*\rank(\cT_{(p)}^{(p)})^*}^{BM,\pi^*G_{(p)}^*\rtimes\mup}(\pi^*(\cR_{(p)}^{(p)})^*,\bfp) \arrow{r}{\text{Averaging}}&
H_{pn-2-2\rank(\cT_{(p)}^{(p)})^*}^{BM,G_{(p)}^*\rtimes\mup}((\cR_{(p)}^{(p)})^*,\bfp).
\end{tikzcd}
\end{equation*}

Here for the second two averaging maps we have identified $H_{pn-2-2\rank(\cT_{(p)}^{(p)})^*}^{BM,G_{(p)}^*\rtimes ?}((\cR_{(p)}^{(p)})^*,\bfp)$ with $H_{pn-2-2\pi^*\rank(\cT_{(p)}^{(p)})^*}^{BM,(\pi^*G_{(p)}^*)\rtimes (?\times\mup)}(\pi^*(\cR_{(p)}^{(p)})^*,\bfp)$ for $?=\bC^*$,$\mup$.

Since averaging commutes also with specialization, it follows that the discrepancy from additivity lies in the image of the averaging map $A^*\to\cA^*$. But this map is equal to $0$.\end{proof}

\subsection{Centrality}We have the following:

\begin{prop}$F_\hbar$ maps into the center of $A^*_\hbar$.\end{prop}

\begin{proof} The idea is to adapt the proof\footnote{Which is itself an adaptation of the construction, using Beilinson-Drinfeld grassmannians, of the commutativity constraint on the Satake category, see \cite{MV}.} of commutativity of $A^*$ given in the Appendix to \cite{BFN2} to the present setup. 
Consider the diagram:


\begin{equation*}
\begin{tikzcd}
& \cR_{(p)}^{(p)}\times \cR  &\\
& \cR_{(p)0}^{(p)}\times_Y \cR_{(p)0}^{0} \arrow{u}{\beta} & \\
\widetilde{\cR}_{(p)0}^{(p)}\times_{\bbN_{(p)0}} \cR_{(p)0}^{0} \arrow{ru}{\gamma_l} \arrow{d}{\delta_l} & & \widetilde{\cR}_{(p)0}^{0} \times_{\bbN_{(p)0}} {\cR}_{(p)0}^{(p)} \arrow{lu}{\text{twist}\circ\gamma_r} \arrow{d}{\delta_r}\\
\widetilde{\cR}_{(p)0}^{(p)}\frac{\times_{\bbN_{(p)0}}}{{G_{(p)0}}} \cR_{(p)0}^{0} \arrow{rd}{\epsilon_l} & & \widetilde{\cR}_{(p)0}^{0} \frac{\times_{\bbN_{(p)0}}}{{G_{(p)0}}} {\cR}_{(p)0}^{(p)} \arrow{ld}{\epsilon_r} \\
& \cR_{(p)0}^{(p)0}. &
\end{tikzcd}
\end{equation*}


Notation: let $\alpha$ be as before\footnote{i.e. any of the symbols $G$, $\bbN$, $\cT$, $\cR$, $\widetilde{\cT}$, $\widetilde{\cR}$.}. We set
$$
\alpha_{(p)0}^{(p)0}:= \left(X\times_{(X^{\mup\cup\{0\}})}\alpha_{\mup\cup\{0\}}^{\mup\cup\{0\}}\right)//\mup.
$$
Here the map $X\to X^{\mup\cup\{0\}}$ is the product of the `twisted-diagonal' embedding $X\to X^{\mup}$ determined $\chi$ and the inclusion of $\{0\}$ in $\bG_a$ under the identification $X^{\mup\cup\{0\}}=X^{\mup}\times X^{\{0\}}=X^{\mup}\times\bG_a$. The action of $\mup$ is again the one which does not involve any loop-rotation. Removing the superscript $(p)$ (resp. $0$, resp. $(p)0$) from $\alpha_{(p)0}^{(p)0}$ corresponds to removing the superscript $\mup\cup$ (resp. $\cup\{0\}$, resp. $\mup\cup\{0\}$) from $\alpha_{\mup\cup\{0\}}^{\mup\cup\{0\}}$ in its defining equation. We may write the spaces of the `left path' as follows:

$$
\cR_{(p)}^{(p)}\times \cR(R)=\left\{(y,\cE,f,v,\cE_0,f_0,v_0)  \left|~\begin{matrix}   y\in Y(R)\hfill\\
\cE\text{ a principal }G\text{-bundle over }\pi^*\Delta_1(y)\hfill\\
f\text{ a trivialization of }\cE\text{ over }\pi^*\Delta_1^{1}(y)\hfill\\
{v}\text{ an }\bbN\text{-section of }\cE\text{ such that }f(v)\text{ extends}\hfill\\
~\text{ to }\pi^*\Delta_1(y)\hfill\\
\cE_0\text{ a principal }G\text{-bundle over }\Delta_1(\{0\})\hfill\\
f_0\text{ a trivialization of }\cE_0\text{ over }\Delta_1^1(\{0\})\hfill\\
{v_0}\text{ an }\bbN\text{-section of }\cE_0\text{ such that }f_0(v_0)\text{ extends}\hfill\\
~\text{ to }\Delta_1(\{0\})\hfill
\end{matrix}\right.\right\}/\sim
$$
$$
{\cR}_{(p)0}^{(p)}\times_{Y} \cR_{(p)0}^{0}(R)=\left\{(y,\cE,f,v,\cE_0,f_0,v_0)  \left|~\begin{matrix}   y\in Y(R)\hfill\\
\cE\text{ a principal }G\text{-bundle over }\pi^*\Delta_1(y)\cup\Delta_1(\{0\})\hfill\\
f\text{ a trivialization of }\cE\text{ over }\pi^*\Delta_1^{1}(y)\cup\Delta_1(\{0\})\hfill\\
{v}\text{ an }\bbN\text{-section of }\cE\text{ such that }f(v)\text{ extends}\hfill\\
~\text{ to }\pi^*\Delta_1(y)\cup\Delta_1(\{0\})\hfill\\
\cE_0\text{ a principal }G\text{-bundle over }\pi^*\Delta_1(y)\cup\Delta_1(\{0\})\hfill\\
f_0\text{ a trivialization of }\cE_0\text{ over }\pi^*\Delta_1(y)\cup\Delta_1^1(\{0\})\hfill\\
{v_0}\text{ an }\bbN\text{-section of }\cE_0\text{ such that }f_0(v_0)\text{ extends}\hfill\\
~\text{ to }\pi^*\Delta_1(y)\cup\Delta_1(\{0\})\hfill
\end{matrix}\right.\right\}/\sim
$$
$$
\widetilde{\cR}_{(p)0}^{(p)}\times_{\bbN_{(p)0}} \cR_{(p)0}^{0}(R)=\left\{(y,\cE,f,g,v,\cE_0,f_0,v_0)  \left|~\begin{matrix}   y\in Y(R)\hfill\\
\cE\text{ a principal }G\text{-bundle over }\pi^*\Delta_1(y)\cup\Delta_1(\{0\})\hfill\\
f\text{ a trivialization of }\cE\text{ over }\pi^*\Delta_1^{1}(y)\cup\Delta_1(\{0\})\hfill\\
g\text{ a trivialization of }\cE\text{ over }\pi^*\Delta_1(y)\cup\Delta_1(\{0\})\hfill\\
{v}\text{ an }\bbN\text{-section of }\cE\text{ such that }f(v)\text{ extends}\hfill\\
~\text{ to }\pi^*\Delta_1(y)\cup\Delta_1(\{0\})\hfill\\
\cE_0\text{ a principal }G\text{-bundle over }\pi^*\Delta_1(y)\cup\Delta_1(\{0\})\hfill\\
f_0\text{ a trivialization of }\cE_0\text{ over }\pi^*\Delta_1(y)\cup\Delta_1^1(\{0\})\hfill\\
{v_0}\text{ an }\bbN\text{-section of }\cE_0\text{ such that }f_0(v_0)=g(v)\hfill
\end{matrix}\right.\right\}/\sim
$$
$$
\widetilde{\cR}_{(p)0}^{(p)} \frac{\times_{\bbN_{(p)0}}}{^{G_{(p)0}}} {\cR}_{(p)0}^{0}(R)=\left\{(y,\cE,f,v,\cE_0,h_0,v_0)  \left|~\begin{matrix}   y\in Y(R)\hfill\\
\cE\text{ a principal }G\text{-bundle over }\pi^*\Delta_1(y)\cup\Delta_1(\{0\})\hfill\\
f\text{ a trivialization of }\cE\text{ over }\pi^*\Delta_1^{1}(y)\cup\Delta_1(\{0\})\hfill\\
{v}\text{ an }\bbN\text{-section of }\cE\text{ such that }f(v)\text{ extends}\hfill\\
~\text{ to }\pi^*\Delta_1(y)\cup\Delta_1(\{0\})\hfill\\
\cE_0\text{ a principal }G\text{-bundle over }\pi^*\Delta_1(y)\cup\Delta_1(\{0\})\hfill\\
h_0\text{ an isomorphism of }\cE_0\text{ with }\cE\hfill\\
~\text{ over }\pi^*\Delta_1(y)\cup\Delta_1^1(\{0\})\hfill\\
{v_0}\text{ an }\bbN\text{-section of }\cE_0\text{ such that }h_0(v_0)=v\hfill
\end{matrix}\right.\right\}/\sim
$$
and
$$
\cR_{(p)0}^{(p)0}(R)=\left\{(y,\cF,e,w)  \left|~\begin{matrix}   y\in Y(R)\hfill\\
\cF\text{ a principal }G\text{-bundle over }\pi^*\Delta_1(y)\cup\Delta_1(\{0\})\hfill\\
e\text{ a trivialization of }\cF\text{ over }\pi^*\Delta_1^{1}(y)\cup\Delta_1^1(\{0\})\hfill\\
{w}\text{ an }\bbN\text{-section of }\cF\text{ such that }e(w)\text{ extends to }\pi^*\Delta_1(y)\cup\Delta_1(\{0\})\hfill\end{matrix}\right.\right\}/\sim.
$$

Here the $\{0\}$ of $\Delta_1^?(\{0\})$ denotes the fixed $R$-point $\{0\}$ of $X$. By definition, $\beta$ is the product of the co-placid map $\cR_{(p)0}^{(p)}\to \cR_{(p)}^{(p)}$ induced by the embedding $\pi^*\Delta_{1}^{1}(y)\to \pi^*\Delta_{1}^{1}(y)\cup\Delta_1(\{0\})$, and the co-placid map $\cR_{(p)0}^{0}\to Y\times \cR$ induced by the embedding $\Delta_{1}^{1}(\{0\})\to \pi^*\Delta_{1}(y)\cup\Delta^1_1(\{0\})$. The map $\gamma_l$ factors as
$$
\widetilde{\cR}_{(p)0}^{(p)}\times_{\bbN_{(p)0}} \cR_{(p)0}^{0} \xrightarrow{u_l} \widetilde{\cR}_{(p)0}^{(p)}\times_{Y} \cR_{(p)0}^{0} \xrightarrow{v_l} \cR_{(p)0}^{(p)}\times_Y \cR_{(p)0}^{0}
$$
where $v_l$ is a $G_{(p)0}$-torsor and $u_l$ fits into a `restriction with supports framework': a Cartesian diagram
$$
\begin{matrix} 
\widetilde{\cR}_{(p)0}^{(p)}\times_{\bbN_{(p)0}} \cR_{(p)0}^{0} & \xrightarrow{u_l} & \widetilde{\cR}_{(p)0}^{(p)}\times_{Y} \cR_{(p)0}^{0}\\
\downarrow&&\downarrow\\
\widetilde{\cT}_{(p)0}^{(p)}\times_{\bbN_{(p)0}} \cR_{(p)0}^{0} & \xrightarrow{u'_l} & \widetilde{\cT}_{(p)0}^{(p)}\times_{Y} \cR_{(p)0}^{0}\\
\end{matrix}
$$
such that $u_l'$ is a section of a vector bundle map
$$
\widetilde{\cT}_{(p)0}^{(p)}\times_{Y} \cR_{(p)0}^{0} = G_{(p)0}^{(p)}/G_{(p)0} \times_Y \bbN_{(p)0}\times_Y \cR_{(p)0}^{0}\to G_{(p)0}^{(p)}/G_{(p)0} \times_Y \cR_{(p)0}^{0}
$$
and whose vertical arrows are ind-f.p closed embeddings. The map $\delta_l$ is a $G_{(p)0}$-torsor, defined by
$$
(y,\cE,f,g,v,\cE_0,f_0,v_0)\mapsto (y,\cE,f,v,\cE_0,h_0=g^{-1}f_0,v_0).
$$
The map $\epsilon_l$ is ind-proper, defined by
$$
(y,\cE,f,v,\cE_0,h_0,v_0)\mapsto (y,\cF=\cE_0,e=hf_0,w=v_0).
$$ 
This describes the `left path'. The `right path' exactly mirrors it\footnote{Just exchange superscripts $\mup$, $0$, and on the level of points exchange $\cE$ with $\cE_0$, $f$ with $f_0$, $v$ with $v_0$, $g$ with $g_0$, $h$ with $h_0$ etc. We have labelled our data $\cF$, $e$, $w$ in $\cR_{(p)0}^{(p)0}$ is because in the `left path' we have $(\cF,e,w)=(\cE_0,fg^{-1}f_0,v_0)$ while in the `right path' we have $(\cF,e,w)=(\cE,f_0g_0^{-1}f,v)$.} and has all the same properties. If we restrict our diagram to $Y^*$, then the subscripts $(p)$, $0$ `split apart', and the result is rather degenerate. That is, it coincides with the restriction to $Y^*$ of:

\begin{equation*}
\begin{tikzcd}[column sep=-3pc]
& \cR_{(p)}^{(p)}\times \cR  &\\
& (\cR_{(p)}^{(p)}\times\bbN)\times_Y (\bbN_{(p)}\times\cR) \arrow{u}{\beta'} & \\
(\widetilde{\cR}_{(p)}^{(p)}\times G(\cO)\times\bbN)\times_{(\bbN_{(p)}\times\bbN)} (\bbN_{(p)}\times\cR) \arrow{ru}{\gamma'_l}  & & (\bbN_{(p)}\times_YG_{(p)}\times\widetilde{\cR}) \times_{(\bbN_{(p)}\times\bbN)} (\cR_{(p)}^{(p)}\times\bbN) \arrow{lu}{\gamma'_r} \\
\widetilde{\cR}_{(p)}^{(p)}\times G(\cO)\times\cR  \arrow{d}{\delta'_l} \arrow{u}{\veq} & & G_{(p)}\times\widetilde{\cR}\times \cR_{(p)}^{(p)} \arrow{d}{\delta'_r} \arrow{u}{\veq}\\
\cR_{(p)}^{(p)}\times \cR \arrow{rd}{\epsilon'_l} & & \cR\times {\cR}_{(p)}^{(p)} \arrow{ld}{\epsilon'_r} \\
& \cR_{(p)}^{(p)}\times \cR. &
\end{tikzcd}
\end{equation*}

Here the maps from the fourth row to the top are the obvious projection maps, while the maps from the fourth row to the bottom are the obvious action maps. If instead we restrict our diagram to $\{0\}\subset Y$, the subscripts $(p)$, $0$ `fuse' and the result is again degenerate: we get

\begin{equation*}
\begin{tikzcd}
& \cR\times \cR  &\\
& \cR\times \cR \arrow{u}{\veq} & \\
\widetilde{\cR}\times_{\bbN(\cO)} \cR \arrow{ru}{} \arrow{d}{} & & \widetilde{\cR} \times_{\bbN(\cO)} {\cR} \arrow{lu}{\text{twist}} \arrow{d}{}\\
\widetilde{\cR}\frac{\times_{\bbN(\cO)}}{G(\cO)} \cR \arrow{rd}{} & & \widetilde{\cR} \frac{\times_{\bbN(\cO)}}{G(\cO)} {\cR} \arrow{ld}{} \\
& \cR. &
\end{tikzcd}
\end{equation*}

We leave it to the reader to write out the appropriate equivariant structures implicit in the following chain of maps, and to check that the quoted dimension theories are appropriately compatible:

\begin{equation*}
\begin{tikzcd}
H_{*-2\rank(\cT_{(p)}^{(p)})}^{BM,G_{(p)}\rtimes\mup}(\cR_{(p)}^{(p)})\otimes_{\bfp[a,\hbar]} A^*_\hbar[a]\arrow{r}{=}&
H_{*-2\rank(\cT_{(p)}^{(p)})}^{BM,G_{(p)}\rtimes\mup}(\cR_{(p)}^{(p)})\otimes_{\bfp[a,\hbar]} H_{*-2\rank{\cT}}^{BM,G(\cO)\rtimes\mup}(\cR)\arrow{d}{\veq} \\
H_{*-2\rank(\cT_{(p)}^{(p)})-2\rank(\cT)}^{BM,(G_{(p)0}\times_Y G_{(p)0})\rtimes\mup}(\cR_{(p)}^{(p)}\times \cR) \arrow{d}{\beta^!}  &
 \arrow{l}{\text{restrict}} H_{*-2\rank(\cT_{(p)}^{(p)})-2\rank(\cT)}^{BM,(G_{(p)}\times G(\cO))\rtimes\mup}(\cR_{(p)}^{(p)}\times \cR)\\
H_{*-2\rank(\cT_{(p)0}^{(p)})-2\rank(\cT_{(p)0}^{0})}^{BM,(G_{(p)0}\times_Y G_{(p)0})\rtimes\mup}(\cR_{(p)0}^{(p)}\times_Y \cR_{(p)0}^{0}) \arrow{r}{\text{descent}}& 
H_{*-2\rank(\widetilde{\cT}_{(p)0}^{(p)})-2\rank(\cT_{(p)0}^{0})}^{BM,(G_{(p)0}\times_Y G_{(p)0}\times_Y G_{(p)0})\rtimes\mup}(\widetilde{\cR}_{(p)0}^{(p)}\times_{Y} \cR_{(p)0}^{0}) \arrow{d}{\text{restrict on }id\times_Y\text{diag}}\\
H_{*-2\delta_l^!\epsilon_l^*\rank(\widetilde{\cT}_{(p)0}^{(p)0})}^{BM,(G_{(p)0}\times_Y G_{(p)0})\rtimes\mup}(\widetilde{\cR}_{(p)0}^{(p)}\times_{\bbN_{(p)0}} \cR_{(p)0}^{0}) \arrow{d}{\text{descent}} &
 H_{*-2\rank(\widetilde{\cT}_{(p)0}^{(p)})-2\rank(\cT_{(p)0}^{0})}^{BM,(G_{(p)0}\times_Y G_{(p)0})\rtimes\mup}(\widetilde{\cR}_{(p)0}^{(p)}\times_{Y} \cR_{(p)0}^{0}) \arrow{l}{u_l^!}\\
 H_{*-2\epsilon_l^*\rank(\widetilde{\cT}_{(p)0}^{(p)0})}^{BM,(G_{(p)0})\rtimes\mup}(\widetilde{\cR}_{(p)0}^{(p)}\frac{\times_{\bbN_{(p)0}}}{G_{(p)0}} \cR_{(p)0}^{0}) \arrow{r}{(\epsilon_l)_*} &
 H_{*-2\rank(\widetilde{\cT}_{(p)0}^{(p)0})}^{BM,(G_{(p)0})\rtimes\mup}(\cR_{(p)0}^{(p)0}).
\end{tikzcd}
\end{equation*}\

Here we have dropped the homological coefficients `$\bfp$' for brevity. This is called the `homological left path over $Y$'. Similarly there is a `homological right path over $Y$'. Moreover, we have versions of both `homological paths' for the restrictions of our original diagram to $Y^*$, $\{0\}$, and specialization map of paths
$$
\text{`homological left path over }Y^*\text{'} \to \text{`homological left path over }\{0\}\text{'}
$$
$$
\text{`homological right path over }Y^*\text{'} \to \text{`homological right path over }\{0\}\text{'}
$$
since every step of both paths is compatible with specialization. One the one hand, both `homological paths over $Y^*$' give as their composition the identity map
$$
H_{*-2-2\rank(\cT_{(p)}^{(p)})^*}^{BM,G_{(p)}\rtimes\mup}((\cR_{(p)}^{(p)})^*)\otimes_{\bfp[a,\hbar]} A^*_\hbar[a]\to H_{*-2-2\rank(\cT_{(p)}^{(p)})^*}^{BM,G_{(p)}\rtimes\mup}((\cR_{(p)}^{(p)})^*)\otimes_{\bfp[a,\hbar]} A^*_\hbar[a].
$$
On the other hand, the `left homological path over $\{0\}$' gives as its composition the multiplication map\footnote{Indeed, this is the definition of convolution from \cite{BFN}.}
$$
A^*_\hbar[a]\otimes_{\bfp[a,\hbar]}A^*_\hbar[a]\xrightarrow{\text{`convolution'}} A^*_\hbar[a]
$$ 
while the `right homological path over $\{0\}$' gives as its composition the \emph{twisted} multiplication map
$$
A^*_\hbar[a]\otimes_{\bfp[a,\hbar]}A^*_\hbar[a]\xrightarrow{\text{`convolution'}\circ\text{twist}} A^*_\hbar[a].
$$ 
It follows that in fact the image of the specialization map 
$$
H_{*-2-2\rank(\cT_{(p)}^{(p)})^*}^{BM,G_{(p)}\rtimes\mup}((\cR_{(p)}^{(p)})^*)\otimes_{\bfp[a,\hbar]} A^*_\hbar[a]\to A^*_\hbar[a]
$$
is in the center of $A^*_\hbar[a]$. By its very definition, $F_\hbar$ factors through this map.\end{proof}

\subsection{Completion of proof}

\begin{enumerate}\item \underline{$F_\hbar$ is multiplicative.} The proof is essentially the same as the proof of centrality, but instead of keeping one copy of $\cR$ fixed and allowing the other to deform to $\cR^{\mup}$ with its cyclic $\mup$-action, we allow \emph{both} copies of $\cR$ to deform in that way. In fact it is easier because we only need one `path'. We will content ourselves with drawing the defining diagram; the conscientious reader can plug in the method of specialization.

$$
\cR_{(p)}^{(p)}\times \cR_{(p)}^{(p)}\xleftarrow{} \widetilde{\cR}_{(p)}^{(p)}\times_{\bbN_{(p)}} {\cR}_{(p)}^{(p)}\to \widetilde{\cR}_{(p)}^{(p)}\frac{\times_{\bbN_{(p)}}}{^{G_{(p)}}} {\cR}_{(p)}^{(p)}\to \cR_{(p)}^{(p)}.
$$\
\item \underline{$F_\hbar$ sends $1$ to $1$.} Note that $1\in A^*$ is the fundamental class of the fiber $\bbN(\cO)$ of $\cR$ over the base point of $Gr$. Certainly Steenrod's construction sends this to the fundamental class of the fiber $\bbN(\cO)^{\mup}$ of $\cR^{\mup}$ over the base point of $Gr^{\mup}$, and this is sent by the `descent isomorphism (b)' construction to the fundamental class of the fiber $\bbN_{(p)}^*$ of $(\cR_{(p)}^{(p)})^*$ over the base section of $Gr_{(p)}^*$. But this section extends to a base section of $Gr_{(p)}$ - namely, the trivial $G$-bundle with the trivial trivialization. The fiber of $\cR_{(p)}^{(p)}$ over this section is $\bbN_{(p)}$. Since this is a vector bundle over $Y$, specialization sends its fundamental class to the fundamental class of its zero fiber $\bbN(\cO)$, as required.\
\item \underline{$F_\hbar\mod\hbar$ is the Frobenius map.} This is essentially clear from the construction. It amounts to showing that the specialization (over $Y$) of the class $b$ in
$$
H_{*-2\rank(\cT_{(p)}^{(p)})^*}^{BM,G_{(p)}}((\cR_{(p)}^{(p)})^*)
$$
obtained by applying Steenrod's construction to $a\in (A^*)^{(1)}$, then applying descent isomorphism (b) and then restricting all the way to $G_{(p)}$-equivariance\footnote{Rather than $G_{(p)}\rtimes\mup$-equivariance. This is the same as killing $\hbar$, since the restriction from $G_{(p)}\rtimes\mup$-equivariance to $G_{(p)}$-equivariance commutes with specialization, and over the $0$ fiber is exactly the map $A^*_\hbar[a]\to A^*$ which kills $\hbar$ and $a$.}, equals $x^p$ (recall Diagram \ref{fhbar}). But by a general property of specialization, it is equal to the specialization over $X$ of the class $\pi^*(b)$ in
$$
H_{*-2\pi^*\rank(\cT_{(p)}^{(p)})^*}^{BM,\pi^*G_{(p)}^*}(\pi^*(\cR_{(p)}^{(p)})^*).
$$
Under the identifications $\pi^*(\cR_{(p)}^{(p)})^*=\cR^{\mup}\times Y^*$, $\pi^*G_{(p)}^*=G(\cO)^{\mup}\times Y^*$, $\pi^*(b)$ is just the pull-back of $a^{\boxtimes p}$ along the projection away from $Y^*$. Thus it is enough to prove the more general statement that for $a_1,\ldots,a_p\in A^*$, the convolution product $a_1\ldots a_p$ is equal to the specialization in $\pi^*\cR_{(p)}^{(p)}$ of $a_1\boxtimes\ldots\boxtimes a_p$. That is achieved by choosing an enumeration $\mup=\{1,\ldots,p\}$ and considering the restriction along the `twisted-diagonal' embedding $X\to X^{\mup}$ of the global convolution diagram (see Appendix to \cite{BFN2}):

\begin{equation*}
\begin{tikzcd}
{\cR_{\{1\}}^{\{1\}}\times\ldots\times\cR_{\{p\}}^{\{p\}}} \arrow{r} & {\widetilde{\cR}_{\mup}^{\{1\}}\times_{\bbN_{\mup}}\ldots\times_{\bbN_{\mup}}\widetilde{\cR}_{\mup}^{\{p-1\}}}\times_{\bbN_{\mup}}\cR_{\mup}^{\{p\}} \arrow{d}\\
& {\widetilde{\cR}_{\mup}^{\{1\}}\frac{\times_{\bbN_{\mup}}}{G_{\mup}}\ldots\frac{\times_{\bbN_{\mup}}}{G_{\mup}}\widetilde{\cR}_{\mup}^{\{p-1\}}}\frac{\times_{\bbN_{\mup}}}{G_{\mup}}\cR_{\mup}^{\{p\}} \arrow{d}\\
& \cR_{(\mup)}^{(\mup)}.
\end{tikzcd}
\end{equation*}\end{enumerate}

\subsection{Closing remarks}\label{closing}

\begin{enumerate}\item There is a closed embedding
$$
\begin{matrix}
\cR_1^1&\to &\cR_{(p)}^{(p)}\\
(y,\cE,f,v)&\mapsto& (y,\pi^*\cE,\pi^*f,\pi^*v)
\end{matrix}
$$
and similarly compatible closed embeddings $\alpha_1^1\to\alpha_{(p)}^{(p)}$ for any symbol $\alpha$ (see Subsection \ref{alpha}). We also have the compatible closed embeddings of groups $G_1\to G_{(p)}$, $\bbN_1\to\bbN_{(p)}$. In fact we have already used one of these to prove linearity of $F_\hbar$ in Subsection \ref{linearity!}.\ 

\item For large $p$, the $\bbN=0$-version of $A_\hbar$ has a name: it is\footnote{In characteristic $0$ this is due to \cite{BF}. In characteristic $p$ it requires some proof - a note will shortly be available.} the \emph{quantum Toda lattice}, denoted $Toda_\hbar$ and given as the two-sided quantum Hamiltonian reduction $N^\vee_\psi\backslash\backslash\cD_\hbar(G^\vee)//_\psi$ of the Rees algebra of crystalline differential operators, $\cD_\hbar(G^\vee)$, of the Langlands dual group $G^\vee$ over $\bfp$, with respect to a regular character $\psi$ of a maximal unipotent $N^\vee\subset G^\vee$. As a quantum Hamiltonian reduction of a ring of differential operators, it has a canonical Frobenius-constant structure. It follows from a torus localization argument that this Frobenius-constant structure coincides with the one we have produced in this paper. The ind-f.p. closed embedding $\cR\to\cT$ induces a pushforward map of $H^*_{G\times\bC^*}(*,\bfp)$-algebras
$$
A^*_\hbar\to H_{*-2\rank{\cT}}^{BM,G\times\bC^*}(\cT,\bfp) \cong H_{*}^{BM,G\times\bC^*}(Gr,\bfp).
$$
For all $p$, this map is compatible with the Frobenius-constant structure. For large $p$ this map is an embedding. So for large $p$, Theorem \ref{mainthm} can be understood as saying that the subalgebra $A^*_\hbar$ of the quantum Toda lattice contains the image of $A^*\subset Toda:=Toda_\hbar/\hbar$ under the canonical Frobenius-constancy map $Toda^{(1)}\to Toda_\hbar$.\ 

\item\label{funexa} \underline{An example.} Let $G=\bC^*$, $\bbN=\bC_{-r}$, $r\geq0$. Then on $\bC$-points we identify:
\begin{enumerate}\item $Gr=\bZ$\
\item $\cT=\bZ\times \bC_{-r}[[t]]$\
\item $\cR= \bZ_{\leq0}\times\bC_{-r}[[t]]\cup \{1\}\times t^{r}\bC_{-r}[[t]]\cup \{2\}\times t^{2r}\bC_{-r}[[t]]\cup\ldots$.\end{enumerate}

For $\bbN=0$, $A_\hbar$ is the Weyl algebra $\bfp[\hbar]\langle x^{\pm},\partial\rangle/([\partial,x]=\hbar)$. The equivariant BM homology of a point $n\in\bZ$ is identified with $\bfp[\hbar,x\partial].x^n$. It is a direct calculation that $F_\hbar$ is the map
$$
\begin{matrix}x^{(1)}&\mapsto &x^p\hfill\\
y^{(1)} &\mapsto&\partial^p\hfill\\
(xy)^{(1)}&\mapsto& x^p\partial^p=\prod_{i=0}^{p-1}(x\partial-i\hbar)=(x\partial)^p-\hbar^{p-1}x\partial=AS_\hbar(x\partial).\end{matrix}
$$\
Here $y=\partial\mod\hbar$. For $\bbN=\bC_{-r}$ with $r\geq0$, $A_\hbar$ is the reduction modulo $p$ of the subalgebra of the integral Weyl algebra $\bZ[\hbar]\langle x^{\pm},\partial\rangle/([\partial,x]=\hbar)$ with $\bfp[\hbar,x\partial]$-basis
$$
\ldots x^{-2},x^{-1},1, \left(\prod_{i=1}^{r}(rx\partial-i\hbar)\right)x,  \left(\prod_{i=1}^{2r}(rx\partial-i\hbar)\right)x^{2},\ldots.
$$
It satisfies $(\prod_{i=1}^{nr}(rx\partial-i\hbar)x^n)(\prod_{i=1}^{mr}(rx\partial-i\hbar)x^m)=\prod_{i=1}^{(m+n)r}(rx\partial-i\hbar)x^{m+n}$. Note that this is a subalgebra of the mod $p$ Weyl algebra if and only if $p$ does not divide $r$. It is again a direct computation that $F_\hbar$ is the map 
$$
\begin{matrix}(x^{-1})^{(1)}&\mapsto &x^{-p}\hfill\\
((rxy)^rx)^{(1)}&\mapsto& \prod_{i=1}^{pr}(rx\partial-i\hbar)x^{p}\hfill\\
(xy)^{(1)}&\mapsto& \prod_{i=0}^{p-1}(x\partial-i\hbar).\end{matrix}
$$
It is an interesting exercise to check that these are really central (of course their images in the mod $p$ Weyl algebra are).\

\item For large $p$ one can prove the centrality and multiplicativity of the map $F_\hbar$, constructed in Subsection \ref{alpha}, via torus localization in conjunction with the above example. But Theorem \ref{mainthm} is true for all odd primes $p$. In fact, the same construction works for $p=2$ but one has to be a little careful to account for the fact that $a^2=\hbar$ in that case.\

\item Suppose that the action of $G$ on $N$ extends to an action of a normalizing supergroup $\widetilde{G}$ of $G$. Then the same proof shows that the corresponding flavor deformation is also a Frobenius-constant quantization.\

\item In \cite{BFN2}, an algebra ind-object $\Omega_\hbar$ of the Satake category $D^b_{G(\cO)\rtimes\bC^*}(Gr,\bC)$ is constructed; its cohomology algebra is the quantum Coulomb branch. The commutativity of the Coulomb branch corresponds to commutativity of the image algebra $\Omega$ of $\Omega_\hbar$ in $D^b_{G(\cO)}(Gr,\bC)$. It is also possible, by essentially the same method given in the Appendix to \emph{loc. cit.}, to tell the same story with $\bfp$ coefficients. In Part II, we will use the \emph{functor} $F_\hbar$ of Theorem \ref{hoho} to upgrade the Frobenius-constant structure in the same way.

\end{enumerate}

\section{$K$-theoretic version}\label{sec4}

\subsection{$K$-theory and $K$-homology}Let $X$ be a scheme over some base $B$ over $\bC$ and let $\cG$ be a (affine, pro-smooth) groupoid scheme over $B$ acting on $X$. We have the \emph{$\cG$-equivariant $K$-homology} of $X$:
$$
K^{\cG}(X):=K_0(D^b_{\cG} \Coh(X))
$$
which is by definition the Grothendieck group of $D^b_{\cG}\Coh(X)$, the $\cG$-equivariant derived category of complexes of sheaves on $X$ with bounded, coherent cohomology sheaves. We have also the \emph{$\cG$-equivariant $K$-theory} of $X$:
$$
K_{\cG}(X):=K_0(\Perf_{\cG}(X))
$$
which is by definition the Grothendieck group of the full subcategory $\Perf_{\cG}(X)$ of $D^b_{\cG}\Coh(X)$ consisting of perfect complexes. We recall some basic facts (see \cite{CG}):
\begin{enumerate}\item $K_{\cG}(X)$ forms a ring, since $\Perf_{\cG}(X)$ is monoidal. The unit element is given by the class of the structure sheaf.\
\item $K^{\cG}(X)$ is a module over $K_{\cG}(X)$, since $D^b_{\cG}\Coh(X)$ is a module category over $\Perf_{\cG}(X)$. Under suitable conditions\footnote{For instance, if $X$ is smooth and $\cG$ is a connected linear algebraic group.}, every equivariant coherent sheaf has a bounded equivariant resolution by vector bundles, so that the defining functor $\Perf_{\cG}(X)\to D^b_{\cG}(X)$ is an equivalence, and in particular the map from $K_{\cG}(X)$ to $K^{\cG}(X)$ is an isomorphism. But this will certainly not be the case in most of our examples.\
\item Let $f:X\to Y$ be a $\cG$-equivariant map of schemes over $B$. We have a monoidal pullback map
$$
f^*:\Perf_{\cG}(Y)\to\Perf_{\cG}(X)
$$
hence the ring map $f^*:K_{\cG}(Y)\to K_{\cG}(X)$. If the derived functor $f^*:D^b_{\cG}\QCoh(Y)\to D^b_{\cG}\QCoh(X)$ sends\footnote{For instance, $f$ may be flat, or a regular closed embedding.} $D^b_{\cG}\Coh(Y)$ to $D^b_{\cG}\Coh(X)$, then we also get a map
$$
f^*:K^{\cG}(Y)\to K^{\cG}(X)
$$
of $K_{\cG}(Y)$-modules.\
\item If instead the derived functor $f_*:D^+_{\cG}\Sh(X)\to D^+_{\cG}\Sh(Y)$ sends\footnote{For instance, if $f$ is proper and $Y$ is finite type.} $D^b_{\cG}\Coh(X)$ to $D^b_{\cG}\Coh(Y)$, then we have a map
$$
f_*:K^{\cG}(X)\to K^{\cG}(Y)
$$
of $K_{\cG}(Y)$-modules.\
\item There is also a version of specialization in equivariant $K$-homology, due to \cite{VV}. Let $f:X\to B\times\bG_a$ be $\cG$-equivariant map, where the factor $\bG_a$ is a `multiplicity space' ignored by the action of $\cG$. Let $i:X_0\to X$ denote the inclusion of the fiber of $B\times\{0\}$, and $j:X^o\to X$ denote the inclusion of the complement. Assume that $i$ is a regular embedding. Then the map $i^*i_*$ on $K$-homology vanishes, and so we get a map
$$
K^{\cG}(X)/i_*K^{\cG}(X_0)\to K^{\cG}(X_0).
$$
Note that the restriction map $j^*:K^{\cG}(X)\to K^{\cG}(X^o)$ has kernel $i_*K^{\cG}(X_0)$, so we get an injection $K^{\cG}(X)/i_*K^{\cG}(X_0)\to K^{\cG}(X^o)$. Assume that this injection is also a surjection\footnote{For instance, if $X$ is quasi-projective.}. Then we have obtained a map
$$
s: K^{\cG}(X^o)\to K^{\cG}(X_0)
$$
which is the promised \emph{specialization} map.\
\item There is also a version of restriction with supports. Suppose $f:X\to Y$ is a $\cG$-equivariant regular closed embedding, and $g:Z\to Y$ is an arbitrary $G$-equivariant map. Then $f_*\cO_X$ is isomorphic to an object of $\Perf_{\cG}(Y)$, so that $g^*f_*\cO_X$ is isomorphic\footnote{Without regularity assumption, it is just some bounded above equivariant coherent complex.} to an object of $\Perf_{\cG}(Z)$. Moreover, this perfect complex is set-theoretically supported on $W:=X\times_YZ$, i.e. its restriction to the complement of $W$ in $Z$ is isomorphic to $0$. Thus tensoring with $g^*f_*\cO_X$ gives an exact functor:
$$
(-)\otimes_{\cO_Z}^{\bL}g^*f_*\cO_X: D^b_{\cG}\Coh(Z)\to D^b_{\cG}\Coh(Z)_{W}.
$$
The RHS is the full subcategory of $D^b_{\cG}\Coh(Z)$ consisting of complexes set-theoretically supported on $W$; the pushforward functor $D^b_{\cG}\Coh(W)\to D^b_{\cG}\Coh(Z)$ factors through this category. The resulting functor $D^b_{\cG}\Coh(W)\to D^b_{\cG}\Coh(Z)_W$ induces\footnote{Since the embedding $W\to Z$ is f.p., being the base change of the f.p. embedding $X\to Y$.} an isomorphism in $K$-homology:
$$
K^{\cG}(W)\xrightarrow{\sim} K_0(D^b_{\cG}\Coh(Z)_{W}).
$$
Thus, we have produced a map
$$
K^{\cG}(Z)\to K^{\cG}(W)
$$
which is the promised \emph{restriction with supports}. \
\item Change of groupoid base, restriction of equivariance, averaging work exactly as for equivariant Borel-Moore homology, see subsection \ref{BM}.
\end{enumerate}

The same compatibilities which were used in the previous section to extend the analogous procedures in equivariant Borel-Moore homology to the case of ind-schemes hold just as well for equivariant $K$-homology. Thus we are able to define the $K$-theoretic versions of the Coulomb branch and the quantum Coulomb branch by using precisely the same underlying geometry: we have rings (under convolution):
$$
KA:=K^{G(\cO)}(\cR)
$$
and
$$
KA_q:=K^{G(\cO)\rtimes\bC^*}(\cR)
$$
which receive ring maps from, respectively
$$
K_{G(\cO)}(*) = K_G(*) = R(G)
$$
and 
$$
K_{G(\cO)\rtimes \bC^*}(*) = K_{G\times\bC^*}(*) = R(G)[q,q^{-1}].
$$
Here $R(G)$ is the (integral) representation ring of $G$. For a maximal torus $T$ of $G$, we have $R(G)=\bZ[\bX^\bullet(T)]^W$ where $W$ is the Weyl group. Moreover, the various compatibilities between specialization and the other procedures hold here as in the case of Borel-Moore homology (see subsection \ref{BM}), so that the ring structure on $KA$ may also be defined using specialization on the appropriate Beilinson-Drinfeld Grassmannian, and is commutative. Also, $KA$ is free over $R(G)$, $KA_q$ is free over $R(G)[q,q^{-1}]$ with respect to both left and right multiplication, $q$ is in the center $Z(KA_q)$ of $KA_q$, and $KA=KA_1:=KA_q|_{q=1}$. We will show:

\begin{thm}\label{thmk}Fix a positive integer $n$ and a primitive $n^{th}$ root of unity $\zeta$. Then there is an injective map of algebras
$$
KA\to Z(KA_\zeta).
$$
Here $KA_\zeta:=KA_q/\Phi_n(q)$ where $\Phi_n$ is the $n^{th}$ cyclotomic polynomial. Equivalently, since $KA_q$ is free over $\bZ[q,q^{-1}]$, this is the same as the subalgebra $1\otimes KA_q$ of $\bC_\zeta\otimes_{\bZ[q,q^{-1}]}KA_q$, where $\bC_\zeta$ is the $\bZ[q,q^{-1}]$-algebra whose underlying ring is $\bC$ and in which $q$ acts as $\zeta$.
\end{thm}

The proof is essentially the same as for the quantum Coulomb branch, except that:
\begin{remark} We must generalize from $\cR^{\mup}$, $\cR_{(p)}^{(p)}$ etc. to $\cR^{\mu_n}$, $\cR_{(n)}^{(n)}$, which are defined exactly as before by replacing any instance of $p$ with $n$. Indeed we never used that $p$ was prime in any of our previous constructions, nor in any of our proofs \emph{except for questions of linearity}. So for instance it is true that we have maps from the $\mod~n$ rings:
$$
F_{\hbar;n}:A^*\to Z(A^*_\hbar)
$$
which lift the $n^{th}$ power map $A^*\to A^*$; to linearize these maps, we have to kill all non-unit factors of $n$. If $n$ is not a prime power, this means we have to kill everything, so we do not obtain an interesting linear map. If $n=p^d$ is a prime power, this amounts to killing $p$, and the resulting map $F_{\hbar;p^d}\mod~p$ is the composition of $F_{\hbar;p}$ with the $(d-1)^{th}$ power of the Frobenius endomorphism of $A\mod~p$, so gives nothing new. However, the map of Theorem \ref{thmk} is a linear map between algebras free over $\bZ$, and is something genuinely different for all $q$. \end{remark}

\subsection{Adams operations}

Let $X$, $B$, $\cG$ be as in the previous subsection, and $n$ be a positive integer. We have a monoidal (nonlinear) functor
$$
St:D^b_{\cG}Coh(X)\to D^b_{\cG^{\mu_n}\rtimes\mu_n}Coh(X^{\mu_n})
$$
whose composition with the functor $D^b_{\cG^{\mu_n}\rtimes\mu_n}Coh(X^{\mu_n})\to D^b_{\cG^{\mu_n}}Coh(X^{\mu_n})$ which forgets the $\mu_n$-equivariant structure coincides with the $n^{th}$ external tensor power functor. The construction is exactly the same as Steenrod's construction of subsection \ref{stconst}, except we work with coherent complexes rather than constructible ones (and with $n$ rather than $p$). Proposition \ref{induced} also holds in this situation with the sole caveat that by `is an induced map' we mean `is a sum of maps induced from various proper subgroups of $\mu_n$' (rather than only from the trivial subgroup). The analogous fact holds also for objects, so we get linear maps
$$
Ad^n: K^{\cG}(X)\to K^{\cG^{\mu_n}\rtimes\mu_n}(X^{\mu_n})/I
$$
where $I$ is the subgroup of $K^{\cG^{\mu_n}\rtimes\mu_n}(X^{\mu_n})$ spanned by all classes of $\cG^{\mu_n}\rtimes\mu_n$-equivariant complexes induced from $\cG^{\mu_n}\rtimes\Gamma$-equivariant complexes, for some proper subgroup $\Gamma$ of $\mu_n$. Furthermore, the functor $St$ preserves perfectness: we have
$$
St:\Perf_{\cG}(X)\to \Perf_{\cG^{\mu_n}\rtimes\mu_n}(X^{\mu_n})
$$
and thus linear maps
$$
Ad_n: K_{\cG}(X)\to K_{\cG^{\mu_n}\rtimes\mu_n}(X^{\mu_n})/J
$$
where $J$ is the subgroup of $K_{\cG^{\mu_n}\rtimes\mu_n}(X^{\mu_n})$ spanned by all classes of $\cG^{\mu_n}\rtimes\mu_n$-equivariant complexes induced from $\cG^{\mu_n}\rtimes\Gamma$-equivariant \emph{perfect} complexes, for some proper subgroup $\Gamma$ of $\mu_n$. In fact, $J$ is an ideal (by the projection formula), $I$ is a $J$-stable submodule for the same reason, $Ad_n$ is a map of rings, and $Ad^n$ is a map of $Ad_n$-modules.\

\begin{remark}[True Adams operations]Induction commutes with restriction, so we have a ring map
$$
K_{\cG}(X)\xrightarrow{Ad_n} K_{\cG^{\mu_n}\rtimes\mu_n}(X^{\mu_n})/J\xrightarrow{\Delta^*} K_{\cG\times\mu_n}(X)/J' = K_{\cG}(X)[q,q^{-1}]/\Phi_n(q).
$$
Here $J'$ is the subgroup of $K_{\cG\times\mu_n}(X)=K_{\cG}(X)[q,q^{-1}]/(q^n-1)$ spanned by induced classes. This equals the ideal generated by the elements $\sum_{j=1}^dq^{nj/d}$ for all $d>1$ dividing $n$; and the lcf of these is $\Phi_n(q)$. By the splitting principle, the image of this map is contained in 
$$
K_{\cG}(X)\subset K_{\cG}(X)[q,q^{-1}]/\Phi_n(q).
$$ The resulting ring endomorphism of $K_{\cG}(X)$ is the $n^{th}$ Adams operation. The relevant example for us is with $\cG=G(\cO)$, $X=*$. Fix a maximal torus $T$ of $G$; then the $n^{th}$ Adams operation is identified with the ring map
$$
\bZ[\bX^\bullet(T)]^W\to \bZ[\bX^\bullet(T)]^W
$$
which sends a $W$-invariant sum $\sum_i\chi_i$ of characters $\chi_i$ to the $W$-invariant sum $\sum_i\chi_i^n$.
\end{remark}

\subsection{Proof of Theorem \ref{thmk}}

The map in question is constructed as in equation \ref{fhbar}: as the composition

\begin{equation}
\begin{tikzcd}[column sep=6pc]
KA = K^{G(\cO)}(\cR) \arrow{r}{Ad^n} & 
K^{G(\cO)^{\mu_n}\rtimes\mu_n}(\cR^{\mu_n})/I \arrow{d}{\text{Descent isomorphism (b)}}\\
K^{G_{(n)}^*\rtimes\mu_n}((\cR_{(n)}^{(n)})^*)/I'' \arrow{d}{\text{Specialize}}&
 \arrow{l}{\text{Restrict}}  K^{G_{(n)}^*\rtimes\bC^*}((\cR_{(n)}^{(n)})^*)/I'\\
K^{G(\cO)\rtimes\mu_n}(\cR)/I'''. & 
\end{tikzcd}
\end{equation}

Here:
\begin{enumerate}\item $I$ is the ideal of $K^{G(\cO)^{\mu_n}\rtimes\mu_n}(\cR^{\mu_n})$ spanned by all classes induced from $K^{G(\cO)^{\mu_{n}}\rtimes\mu_{n/d}}(\cR^{\mu_n})$, for some $d>1$ dividing $n$;\
\item $I'$ is the ideal of $K^{G_{(n)}^*\rtimes\bC^*}((\cR_{(n)}^{(n)})^*)$ corresponding to $I$ under the descent isomorphism; it is equal to the ideal spanned by all classes pushed forward from $K^{\bC\times_{\bC}G_{(n)}^*\rtimes\bC^*}((\bC\times_{\bC}\cR_{(n)}^{(n)})^*)$, for some non-trivial $\bC^*$-equivariant cover\footnote{Recall that in this situation, the base copy of $\bC$ has the action of $\bC^*$ of weight $n$, so a $\bC^*$-equivariant non-trivial cover $\bC\to\bC$ is the $d^{th}$-power map for some $d>1$ dividing $n$.} $\bC\to\bC$;\
\item $I''$ is the image of $I'$ under restriction; it is the ideal of $K^{G_{(n)}^*\rtimes\mu_n}((\cR_{(n)}^{(n)})^*)$ spanned by all classes pushed forward from $K^{\bC\times_{\bC}G_{(n)}^*\rtimes\mu_n}((\bC\times_{\bC}\cR_{(n)}^{(n)})^*)$, for some non-trivial $\bC^*$-equivariant cover $\bC\to\bC$. Note that for the degree $d$ equivariant cover we have $K^{\bC\times_{\bC}G_{(n)}^*\rtimes\mu_n}((\bC\times_{\bC}\cR_{(n)}^{(n)})^*) = K^{G_{(n)}^*\rtimes\mu_{n/d}}((\cR_{(n)}^{(n)})^*)$, and pushing forward along $\bC\to\bC$ corresponds to inducing from $\mu_{n/d}$-equivariance to $\mu_n$-equivariance;\
\item $I'''$ is the image of $I''$ under specialization; using the second description\footnote{The first is inapplicable, since there is no specialization for $K^{G_{(n)}^*\rtimes\mu_n}((\cR_{(n)}^{(n)})^*)$ over the non-trivial cover $\bC$, since $\mu_n$ acts non-trivially on the base.} of $I''$ given above, we see that this is the ideal of $K^{G(\cO)\rtimes\mu_n}(\cR)$ spanned by all classes induced from $K^{G(\cO)\rtimes\mu_{n/d}}(\cR)$ for some $d>1$ dividing $n$.\end{enumerate}

Now on the one hand, by the projection formula we see that the composition
$$
K^{G(\cO)\rtimes\mu_{n}}(\cR)\xrightarrow{\text{restriction}}K^{G(\cO)\rtimes\mu_{n/d}}(\cR)\xrightarrow{\text{induction}}K^{G(\cO)\rtimes\mu_{n}}(\cR)
$$
coincides with multiplication by the class $\sum_{j=1}^dq^{nj/d}$. On the other hand, it is a consequence of the `cellularity' of $\cR$ that the restriction of equivariance $K^{G(\cO)\rtimes\mu_{n}}(\cR)\to K^{G(\cO)\rtimes\mu_{n/d}}(\cR)$ is surjective; so the ideal $I'''$ coincides with the ideal generated by the sums $\sum_{j=1}^dq^{nj/d}$, whose lowest common factor is $\Phi_n(q)$. For the same reason, the restriction
$$
K^{G(\cO)\rtimes\bC^*}(\cR)\to K^{G(\cO)\rtimes\mu_{n}}(\cR)
$$
is also surjective, and realizes $K^{G(\cO)\rtimes\mu_{n}}(\cR)\cong K^{G(\cO)\rtimes\bC^*}(\cR)/(q^n-1)$. Therefore we have produced the map
$$
KA\to KA_q/\Phi_n(q)=KA_\zeta.
$$
This map is linear by construction. The proof that it is a map of $Ad_n$-algebras, and lands in the center, is just as for the Borel-Moore homology.

\begin{example}We return to the situation of subsection \ref{closing}, Example \ref{funexa}. We have $R(G)=\bZ[y,y^{-1}]$ and, for $\bbN=0$, $KA_q$ is the $\bZ[y,y^{-1},q,q^{-1}]$-algebra
$$
\cO(T^{\vee\vee})\#_q\cO(T^{\vee}):=\bZ[y,y^{-1},q,q^{-1}]\langle x,x^{-1}\rangle /(yx=qxy).
$$
For $\bbN=\bC_{-r}$, $r\geq 0$, $KA_q$ is given instead by the subalgebra with basis
$$
\ldots,x^{-2},x^{-1},1,\left(\prod_{i=0}^{r-1}(1-y^r/q^i)\right)x, \left(\prod_{i=0}^{2r-1}(1-y^r/q^i)\right)x^2,\ldots
$$
as a left (or right) $\bZ[y,y^{-1},q,q^{-1}]$-module. The map of Theorem \ref{thmk} sends
$$
\begin{matrix}x&\mapsto & x^n\\
y&\mapsto& y^n\end{matrix}
$$
which are central\footnote{Thus for $\bbN=0$ it suffices to set $q^n=1$ to obtain such a map. Why?} when $q^n=1$, in particular when $q=\zeta$ is a primitive $n^{th}$ root of unity. It sends $(1-y^r)^rx$ to
$$
(1-y^{nr})^rx^n = \left(\prod_{i=0}^{n-1}(1-y^r/\zeta^i)^r\right)x^n = \left(\prod_{i=0}^{nr-1}(1-y^r/\zeta^i)\right)x^n,
$$
which is indeed an element of the appropriate subalgebra.\end{example}

\begin{remark}[$\bF_1$?] If we set $n=p^d$, and consider the rings $KA$, $KA_q/\Phi_n(q)$ modulo $p$, then we have $\Phi_{p^d}(1)=p=0$ so that $KA_q/\Phi_{p^d}(q)$ becomes a $\bF_p[q]/\Phi_{p^d}(q)$-quantization of $KA$, and the map of Theorem \ref{thmk} is its Frobenius-constant structure. However, for other values of $n$ there is no obvious analogue of this fact; and already for $n$ a prime power we had to lose some information (by killing $p$) in order to say that our central map lifts some operation of commutative algebra. We speculate that, on the contrary, our map \emph{is} a lift of some operation of commutative $\bF_1$-algebra.\end{remark}


\begin{thebibliography}{100}

\bibitem{BD} A. Beilinson, V. Drinfeld, Quantization of Hitchin's Integrable System and Hecke Eigensheaves, http://www.math.uchicago.edu/~mitya/langlands/hitchin/BD-hitchin.pdf (1991)\

\bibitem{BEF} A. Braverman, P. Etingof, M. Finkelberg, Cyclotomic Double Affine Hecke Algebras, https://arxiv.org/pdf/1611.10216.pdf (2016)\

\bibitem{BF} R. Bezrukavnikov, M. Finkelberg, Equivariant Satake category and Kostant-Whittaker reduction, Moscow Mathematical Journal, Volume 8, Number 1, January-March 2008, Pages 39-72.\

\bibitem{BFN} A. Braverman, M. Finkelberg, H. Nakajima, Towards a mathematical definition of Coulomb branches of $3$-dimensional $\mathcal{N}=4$ gauge theories, II, arXiv:1601.03586 (2016).\

\bibitem{BFN2} A. Braverman, M. Finkelberg, H. Nakajima, Ring objects in the equivariant derived Satake category arising from Coulomb branches, arXiv:1706.02112 (2017)\

\bibitem{BK} R. Bezrukavnikov, D. Kaledin, Fedosov quantization in positive characteristic, Journal of the American Mathematical Society, Volume 21, Number 2, April 2008, Pages 409-438.\

\bibitem{CG} N. Chriss, V. Ginzburg, Representation Theory and Complex Geometry, Birkh{\"a}user, Boston (1997).\

\bibitem{D} V. Drinfeld, Infinite-dimensional vector bundles in algebraic geometry (an introduction), arXiv:math/0309155 (2003)\

\bibitem{DL} E. Dyer, R. K. Lashof, Homology of iterated loop spaces, Amer. J. Math. 84 (1962), 35-88.\

\bibitem{Gaits} D. Gaitsgory, Construction of central elements in the affine Hecke algebra via nearby cycles, Inventiones Mathematicae, May 2001, Volume 144, Issue 2, pp 253-280.\

\bibitem{KA} T. Kudo, S. Araki, Topology of $H_n$-spaces and $H$-squaring operations, Mem. Fac. Sci. Ky{\=u}sy{\=u} Univ. Ser. A., 10 (1956), 85-120.\

\bibitem{May} J. P. May, A General Algebraic Approach to Steenrod Operations, The Steenrod Algebra and
its Applications, Lecture Notes in Mathematics, vol. 168, Springer-Verlag, 1970, pp. 153-229. \

\bibitem{MV} I. Mirkovic, K. Vilonen, Geometric Langlands duality and representations of algebraic groups over commutative rings, Annals of Mathematics, 166 (2007), 95-143.\\

\bibitem{Rask} S. Raskin, D-Modules on Infinite Dimensional Varieties, http://math.mit.edu/~sraskin/dmod.pdf (2015).\

\bibitem{Steenrod} N. E. Steenrod, D. B. A. Epstein, Cohomology operations, Ann. Math.
Studies 50, Princeton Univ. Press, 1962.\

\bibitem{VV} Vezzosi, G. \& Vistoli, A. Invent. math. (2003) 153: 1. https://doi.org/10.1007/s00222-002-0275-2.\

\bibitem{Web} B. Webster, Representation theory of the cyclotomic Cherednik algebra via the Dunkl-Opdam subalgebra, arXiv:1609.05494 (2016).\

\end{thebibliography}
\end{document}